\documentclass[aos,preprint,noinfoline]{imsart}

\RequirePackage[OT1]{fontenc}
\usepackage[bbgreekl]{mathbbol}
\RequirePackage{amsthm,amsmath, amssymb}
\RequirePackage[]{natbib}
\RequirePackage[colorlinks,citecolor=blue,urlcolor=blue]{hyperref}
\usepackage{todonotes}
\usepackage{mathrsfs}
\usepackage{bbm}
\usepackage{hyperref}
\usepackage{enumerate}
\usepackage{graphicx}
\usepackage{float}

\usepackage{cancel}
\usepackage{color}
\definecolor{Blue}{rgb}{0,0,1}
\definecolor{Red}{rgb}{1,0,0}

\startlocaldefs
\numberwithin{equation}{section}
\theoremstyle{plain}
\newtheorem{thm}{Theorem}[section]
\newtheorem{lemma}[thm]{Lemma}
\newtheorem{definition}[thm]{Definition}
\newtheorem{proposition}[thm]{Proposition}
\newtheorem{corollary}[thm]{Corollary}
\newtheorem{assumption}[thm]{Assumption}
\newtheorem{remark}{Remark}
\endlocaldefs

\DeclareMathOperator{\Var}{Var}
\DeclareMathOperator{\E}{E}

\DeclareMathOperator{\adm}{adm}

\def\E{\mathbb{E}}
\def\G{\mathbb{G}}

\def\N{\mathbb{N}}

\def\R{\mathbb{R}}

\def\Z{\mathbb{Z}}

\def\N{\mathbb{N}}
\def\P{\mathbb{P}}

\def\tn{\tilde{n}}

\newenvironment{theorem}{\begin{thm}}{\end{thm}}

\allowdisplaybreaks

\makeatletter
\providecommand*{\diff}%
        {\@ifnextchar^{\DIfF}{\DIfF^{}}}
\def\DIfF^#1{%
        \mathop{\mathrm{\mathstrut d}}%
                \nolimits^{#1}\gobblespace
}
\def\gobblespace{%
        \futurelet\diffarg\opspace}
\def\opspace{%
        \let\DiffSpace\!%
        \ifx\diffarg(%
                \let\DiffSpace\relax
        \else
                \ifx\diffarg\[%
                        \let\DiffSpace\relax
                \else
                        \ifx\diffarg\{%
                                \let\DiffSpace\relax
                        \fi\fi\fi\DiffSpace}

\begin{document}

\begin{frontmatter}
	\title{Locally Adaptive Confidence Bands\thanksref{T1}}
	\runtitle{Locally Adaptive Confidence Bands}
	\thankstext{T1}{Supported by the DFG Collaborative Research Center 823, Subproject C1, and DFG Research Grant RO 3766/4-1.}

	\begin{aug}
		\author{\fnms{Tim} \snm{Patschkowski}\ead[label=e1]{tim.patschkowski@ruhr-uni-bochum.de}}
		\and
		\author{\fnms{Angelika} \snm{Rohde}\ead[label=e2]{angelika.rohde@stochastik.uni-freiburg.de}}

		\affiliation{Ruhr-Universit\"at Bochum and Albert-Ludwigs-Universit\"at Freiburg}

		\address{Fakult\"at f\"ur Mathematik \\
		Ruhr-Universit\"at Bochum \\
		44780 Bochum \\
		Germany \\
		\printead{e1}\\\\
		and\\\\
		Mathematisches Institut\\
		Albert-Ludwigs-Universit\"at Freiburg\\
		Eckerstra{\ss}e 1\\
		79104 Freiburg im Breisgau\\
		Germany\\
		\printead{e2}}
	\end{aug}

\begin{abstract}
\scriptsize
We develop honest and locally adaptive confidence bands for probability densities. They provide substantially improved confidence statements in case of inhomogeneous smoothness, and are easily implemented and visualized. The article contributes conceptual work on locally adaptive inference as a straightforward modification of the global setting imposes severe obstacles for statistical purposes. Among others, we introduce a statistical notion of local H\"older regularity and prove a correspondingly strong version of local adaptivity. We substantially relax the straightforward localization of the self-similarity condition in order not to rule out prototypical densities. The set of densities permanently excluded from the consideration is shown to be pathological in a mathematically rigorous sense. On a technical level, the crucial component for the verification of honesty is the identification of an asymptotically least favorable stationary case by means of Slepian's comparison inequality. 
\end{abstract}


\end{frontmatter}

\small

\section{Introduction} 
Let $X_1,\dots, X_n$ be independent real-valued random variables which are identically distributed according to some unknown probability measure $\P_p$ with Lebesgue density $p$.  
Assume that $p$ belongs to a nonparametric function class $\mathcal P$. For any interval $[a,b]$ and any significance level $\alpha\in(0,1)$, a confidence band for $p$, described by a family of random intervals $C_n(t,\alpha), t \in [a,b]$, is said to be (asymptotically) honest with respect to $\mathcal P$ if the 
coverage inequality
\begin{align*}
	\liminf_n \inf_{p \in \mathcal P} \P^{\otimes n}_p \Big( p(t) \in C_n(t,\alpha) \text{ for all } \  t \in [a,b] \Big) \geq 1- \alpha
\end{align*}
is satisfied. The aim of this article is to develop honest confidence bands $C_n(t,\alpha), t \in [a,b]$, with smallest possible width $\vert C_n(t,\alpha) \vert$ for every $t \in [a,b]$. 
Adaptive confidence sets maintain specific coverage probabilities over a large union of models while shrinking at the fastest possible nonparametric rate simultaneously over all submodels. If $\mathcal P$ is some class of densities within a union of H\"older balls $\mathcal H(\beta,L)$ with fixed radius $L>0$, the confidence band is called globally adaptive, cf. \cite{cai_low2004}, if for every $\beta>0$ and for every $\varepsilon > 0$ there exists some constant $c>0$, such that
\begin{align*}
	\limsup_n\sup_{p \in \mathcal H(\beta,L) \cap \mathcal P} \P^{\otimes n}_p \left( \sup_{t \in [a,b]} \vert C_n(t,\alpha) \vert \geq c \cdot r_n(\beta) \right) < \varepsilon.
\end{align*}
Here, $r_n(\beta)$ denotes the minimax-optimal rate of convergence for estimation under supremum norm loss over $\mathcal H(\beta,L) \cap \mathcal P$, possibly inflated by additional logarithmic factors. However, if $\mathcal{P}$ equals the set of {\it all} densities contained in
$$
\bigcup_{0<\beta\leq\beta^*}\mathcal H(\beta,L),
$$
honest and adaptive confidence bands provably do not exist although adaptive estimation is possible. Indeed, \cite{low1997} shows that honest random-length intervals for a probability density at a fixed point cannot have smaller expected width than fixed-length confidence intervals with the size corresponding to the lowest regularity under consideration. 
Consequently, it is not even possible to construct a family of random intervals $C_n(t,\alpha), t \in [a,b]$, whose expected length shrinks at the fastest possible rate simultaneously over two distinct nested H\"older balls with fixed radius, and which is at the same time asymptotically honest for the union $\mathcal P$ of these H\"older balls. Numerous attempts have been made to tackle this adaptation problem in alternative formulations. Whereas \cite{genovese_wasserman2008} relax the coverage property and do not require the confidence band to cover the function itself but a simpler surrogate function capturing the original function's significant features, most of the approaches are based on a restriction of the parameter space. Under qualitative shape constraints, \cite{hengartner_stark1995}, \cite{duembgen1998, duembgen2003}, and \cite{davies_kovac_meise2009} achieve adaptive inference. Within the models of nonparametric regression and Gaussian white noise, \cite{picard_tribouley2000} succeeded to construct pointwise adaptive confidence intervals under a self-similarity condition on the parameter space, see also \cite{kueh2012} for thresholded needlet estimators. Under a similar condition, \cite{gine_nickl2010} even develop asymptotically honest confidence bands for probability densities whose width is adaptive to the global H\"older exponent. \cite{bull2012} proved that a slightly weakened version of the self-similarity condition is necessary and sufficient. \cite{kerkyacharian_nickl_picard2012} develop corresponding results in the context of needlet density estimators on compact homogeneous manifolds. Under the same type of self-similarity condition, adaptive confidence bands are developed under a considerably generalized Smirnov-Bickel-Rosenblatt assumption based on Gaussian multiplier bootstrap, see \cite{chernozhukov_chetverikov_kato2014Conf}. \cite{hoffmann_nickl2011} introduce a nonparametric distinguishability condition, under which adaptive confidence bands exist for finitely many models under consideration. Their condition is shown to be necessary and sufficient. \\
Similar important conclusions concerning adaptivity in terms of confidence statements are obtained under Hilbert space geometry with corresponding $L^2$-loss, see \cite{juditsky_lambert-lacroix2003}, \cite{baraud2004}, \cite{genovese_wasserman2005}, \cite{cai_low2006}, \cite{robins_vandervaart2006}, \cite{bull_nickl2013}, and \cite{nickl_szabo2016}. Concerning $L^p$-loss, we also draw attention to \cite{carpentier2013}. \\
In this article, we develop locally adaptive confidence bands. They provide substantially improved confidence statements in case of inhomogeneous smoothness. 
Conceptual work on locally adaptive inference is contributed as a straightforward modification of the global setting imposes severe obstacles for statistical purposes. It is already delicate to specify what a "locally adaptive confidence band" should be. Disregarding any measurability issues, one possibility is to require a confidence band $C_{n,\alpha} = (C_{n,\alpha}(t))_{t \in [0,1]}$ to satisfy for every interval $U \subset [a,b]$ and for every $\beta$ (possibly restricted to a prescribed range)
\begin{align*}
	\limsup_{n \to \infty}\sup_{\substack{p \in {\mathcal P}:\\ p_{\arrowvert U_{\delta}}\in \mathcal H_{U_{\delta}}(\beta,L^*)} }\P_p^{\otimes n} \left( \vert C_{n,\alpha}(t) \vert \geq \eta \, r_n(\beta) \text{ for some }Êt \in U \right) \rightarrow 0
\end{align*}
as $\eta \to \infty$, where $U_{\delta}$ is the $\delta$-enlargement of $U$. However, this definition reflects a weaker notion of local adaptivity than the statistician may have in mind. On the other hand, we prove that, uniformly over the function class $\mathcal P$ under consideration, adaptation to the local or pointwise regularity in the sense of \cite{daoudi_levyvehel_meyer1998}, \cite{seuret_levyvehel2002} or \cite{jaffard1995,jaffard2006} is impossible. Indeed, not even adaptive estimation with respect to pointwise regularity at a fixed point is achievable. On this way, we introduce a statistically suitable notion of local regularity $\beta_{n,p}(t), t \in [a,b]$, depending in particular on the sample size $n$. We prove a corresponding strong version of local adaptivity,  
while we substantially relax the straightforward localization of the global self-similarity condition in order not to rule out prototypical densities. 
The set of functions which is excluded from our parameter space diminishes for growing sample size and the set of permanently excluded functions is  shown to be pathological in a mathematically rigorous sense.
Our new confidence band appealingly relies on a discretized evaluation of a modified Lepski-type density estimator, including an additional supremum in the empirical bias term in the bandwidth selection criterion. A suitable discretization and a locally constant approximation allow to piece the pointwise constructions together in order to obtain a continuum of confidence statements. The complex construction makes the asymptotic calibration of the confidence band to the level $\alpha$ non-trivial. Whereas the analysis of the related globally adaptive procedure of \cite{gine_nickl2010} reduces to the limiting distribution of the supremum of a stationary Gaussian process, our locally adaptive approach leads to a highly non-stationary situation. A crucial component is therefore the identification of a stationary process as a least favorable case by means of Slepian's comparison inequality, subsequent to a Gaussian reduction using recent non-asymptotic techniques of \cite{chernozhukov_chetverikov_kato2014Approx}. Due to the discretization, the band is computable and feasible from a practical point of view without losing optimality between the mesh points. Our results are exemplarily formulated in the density estimation framework but can be mimicked in other nonparametric models. To keep the representation concise we restrict the theory to locally adaptive kernel density estimators. The ideas can be transferred to wavelet estimators to a large extent as has been done for globally adaptive confidence bands in \cite{gine_nickl2010}. \\
The article is organized as follows. Basic notations are introduced in Section~\ref{sec:preliminairies}. Section \ref{sec:kap3} presents the main contributions, that  is a substantially relaxed localized self-similarity condition in Subsection \ref{subsec:3.1}, the construction and in particular the asymptotic calibration of the confidence band in Subsection \ref{Subsec:construction} as well as its strong local adaptivity properties in Subsection \ref{subsec:3.3}. Important supplementary results are postponed to Section \ref{sec:aux}, whereas Section \ref{sec:proofs} presents the proofs of the main results. Appendix \ref{AuxResults} contains technical tools for the proofs of the main results. 

\section{Preliminaries and notation} \label{sec:preliminairies}

Let $X_1, \ldots, X_n$, $n \geq 4$, be independent random variables identically distributed according to some unknown probability measure $\P_p$ on $\R$ with continuous Lebesgue density $p$. 
Subsequently, we consider kernel density estimators
\begin{align*}
	\hat{p}_{n}(\cdot,h) = \frac{1}{n} \sum_{i=1}^n K_h \left( X_i - \cdot \right)
\end{align*}
with bandwidth $h > 0$ and rescaled kernel $K_h(\cdot) = h^{-1} K(\cdot/h)$, where $K$ is a measurable and symmetric kernel with support contained in $[-1,1]$, integrating to one, and of bounded variation. Furthermore, $K$ is said to be of order $l \in \N$ if
\begin{align*}
	\int x^j K(x) \diff x = 0 \ \text{ for } 1 \leq j \leq l, \quad \int x^{\l + 1} K(x) \diff x = c \ \text{ with } c \neq 0.
\end{align*}
For bandwidths of the form $h=c \cdot 2^{-j}, j \in \N$, we abbreviate the notation writing $\hat p_n(\cdot,h)=\hat p_n(\cdot,j)$ and $K_h=K_j$. The open Euclidean ball of radius $r$ around some point $t \in \R$ is referred to as $B(t,r)$. Subsequently, the sample is split into two subsamples. For simplicity, we divide the sample into two parts of equal size $\tn = \lfloor n/2 \rfloor$, leaving possibly out the last observation. Let
\begin{align*}
	\chi_1 = \{ X_1, \ldots, X_{\tn} \}, \quad \chi_2 = \{ X_{\tn+1}, \ldots, X_{2\tn} \}
\end{align*}
be the distinct subsamples and denote by $\hat p_n^{(1)}(\cdot, h)$ and $\hat p_n^{(2)}(\cdot,h)$ the kernel density estimators with bandwidth $h$ based on $\chi_1$ and $\chi_2$, respectively. $\E_p^{\chi_1}$ and $\E_p^{\chi_2}$ denote the expectations with respect to the product measures
\begin{align*}
	\P_p^{\chi_1} &= \text{joint distribution of }ÊX_1, \ldots, X_{\tn}, \\
	\P_p^{\chi_2} &= \text{joint distribution of }ÊX_{\tn+1}, \ldots, X_{2\tn}, 
\end{align*}
respectively. \\
For some measure $Q$, we denote by $\Vert \cdot \Vert_{L^p(Q)}$ the $L^p$-norm with respect to $Q$. Is $Q$ the Lebesgue measure, we just write $\Vert \cdot \Vert_p$, whereas $\Vert \cdot \Vert_{\sup}$ denotes the uniform norm. For any metric space $(M,d)$ and subset $K \subset M$, we define the covering number $N(K,d,\varepsilon)$ as the minimum number of closed balls with radius at most $\varepsilon$ (with respect to $d$) needed to cover $K$. If the metric $d$ is induced by a norm $\Arrowvert \cdot \Arrowvert$, we write also $N(K,\Vert \cdot \Vert,\varepsilon)$ for $N(K,d,\varepsilon)$. As has been shown by \cite{nolan_pollard1987} (Section 4 and Lemma~22), the class
\begin{align*}
	\mathcal{K} = \left\{ K \left( \frac{\cdot - t}{h} \right) : t \in \R, \, h>0 \right\}
\end{align*}
with constant envelope $\Vert K \Vert_{\sup}$ satisfies
\begin{align} \label{covering_K}
	N \left( \mathcal K, \Vert \cdot \Vert_{L^p(Q)}, \varepsilon \Vert K \Vert_{\sup} \right) \leq \left( \frac{A}{\varepsilon} \right)^\nu, \quad 0 < \varepsilon \leq 1, \quad p=1,2
\end{align}
for all probability measures $Q$ and for some finite and positive constants $A$ and $\nu$. \\
For $k \in \N$ we denote the $k$-th order Taylor polynomial of the function $p$ at point $y$ by $P^p_{y,k}$. Denoting furthermore by $\lfloor \beta \rfloor = \maxÊ\{ n \in \N \cup \{ 0 \} : n < \beta \}$, the H\"older class $\mathcal H_U(\beta)$ to the parameter $\beta > 0$ on the open interval $U \subset \R$ is defined as the set of functions $p:U \to \R$ admitting derivatives up to the order $\lfloor \beta \rfloor$ and having finite
H\"older norm
\begin{align*}
	\Vert p \Vert_{\beta,U} = \sum_{k=0}^{\lfloor \beta \rfloor} \Vert p^{(k)} \Vert_U + \sup_{\substack{x,y \, \in \, U \\ x \neq y}} \frac{\vert p^{(\lfloor \beta \rfloor)}(x) - p^{(\lfloor \beta \rfloor)}(y) \vert}{\vert x-y \vert^{\beta-\lfloor \beta \rfloor}} < \infty.
\end{align*}
The corresponding H\"older ball with radius $L>0$ is denoted by $\mathcal H_U(\beta,L)$. With this definition of $\Vert \cdot \Vert_{\beta,U}$, the H\"older balls are nested, that is
\begin{align*}
	\mathcal H_U(\beta_2,L) \subset \mathcal H_U(\beta_1,L)
\end{align*}
for $0<\beta_1 \leq \beta_2<\infty$ and $\vert U \vert < 1$. Finally, $\mathcal H_U(\infty,L) = \bigcap_{\beta>0} \mathcal H_U(\beta,L)$ and $\mathcal H_U(\infty) = \bigcap_{\beta>0} \mathcal H_U(\beta)$. Subsequently, for any real function $f(\beta)$, the expression $f(\infty)$ is to be read as $\lim_{\beta \to \infty} f(\beta)$, provided that this limit exists. Additionally, the class of probability densities $p$, such that $p_{\vert U}$ is contained in the H\"older class $\mathcal H_U(\beta,L)$ is denoted by $\mathcal P_U(\beta,L)$. The indication of $U$ is omitted when $U = \R$. 


\section{Main results} \label{sec:kap3}

In this section we pursue the new approach of locally adaptive confidence bands and present the main contribution of this article. A notion of local H\"older regularity tailored to statistical purposes, a corresponding condition of admissibility of a class of functions over which both asymptotic honesty and adaptivity (in a sense to be specified) can be achieved, as well as the construction of the new confidence band are presented. As compared to globally adaptive confidence bands, our confidence bands provide improved confidence statements for functions with inhomogeneous smoothness. 

\smallskip

		\begin{figure}[H] 
		\centering
  		\includegraphics[width=0.77\textwidth]{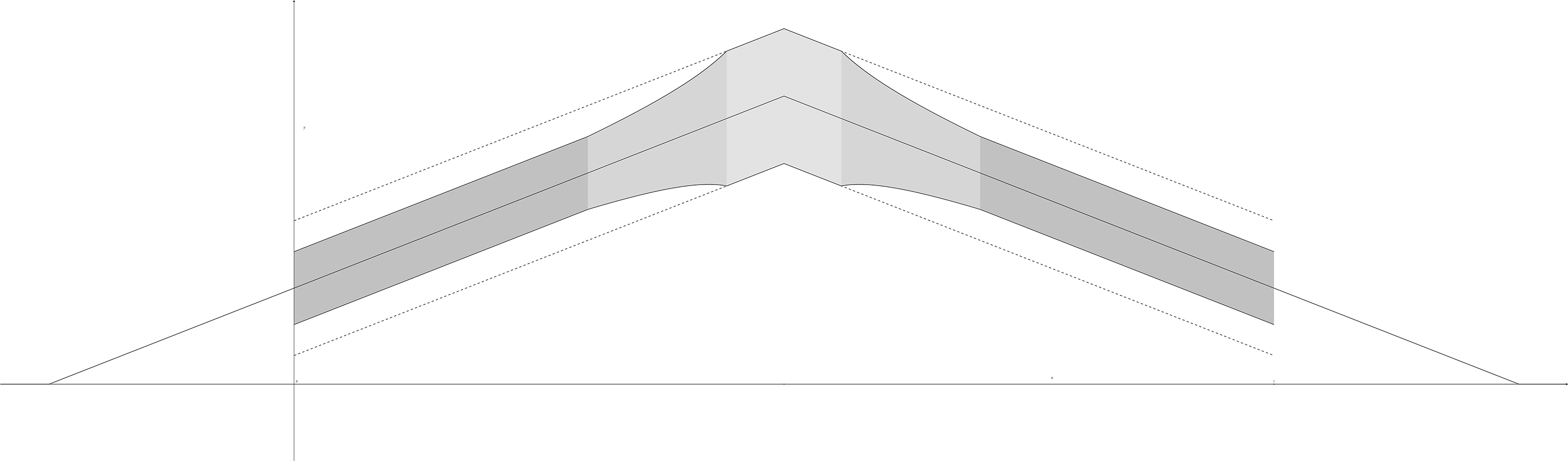}
		\caption{Comparison of locally and globally adaptive confidence bands}
		\label{fig:tri_den}
		\end{figure}

\noindent Figure \ref{fig:tri_den} illustrates the kind of adaptivity that the construction should reveal. The shaded area sketches the intended locally adaptive confidence band as compared to the globally adaptive band (dashed line) for the triangular density and for fixed sample size $n$. This density is not smoother than Lipschitz at its maximal point but infinitely smooth at both sides. The region where globally and locally adaptive confidence bands coincide up to logarithmic factors (light gray regime in Figure~\ref{fig:tri_den}) should shrink as the sample size increases, resulting in a substantial benefit of the locally adaptive confidence band outside of a shrinking neighborhood of the maximal point (dark grey regime together with a shrinking middle grey transition regime in Figure \ref{fig:tri_den}).

\subsection{Admissible functions} \label{subsec:3.1}

As already pointed out in the introduction, 
no confidence band exists which is simultaneously honest and adaptive. 
It is necessary to impose a condition which guarantees the possibility of recovering the unknown smoothness parameter from the data. The subsequently introduced notion of admissibility aligns to the self-similarity condition as used in \cite{picard_tribouley2000} and \cite{gine_nickl2010} among others. Their self-similarity condition  ensures that the data contains enough information to infer on the function's regularity. As also emphasized in \cite{nickl2015}, self-similarity conditions turn out to be compatible with commonly used adaptive procedures and have been shown to be sufficient and necessary for adaptation to a continuum of smoothing parameters in \cite{bull2012} when measuring the performance by the $L^{\infty}$-loss.
\cite{gine_nickl2010} consider globally adaptive confidence bands over the set
\begin{align} \label{self_sim_global}
	&\bigcup_{\beta_* \leq \beta \leq \beta^*} \Bigg\{ p \in \mathcal P (\beta,L) : p \geq \delta \text{ on } [-\varepsilon, 1+\varepsilon], \  
	\frac{c}{2^{j\beta}} \leq \Vert K_j \ast p - p \Vert_{\sup}Ê\text{ for all }Êj\geq j_0 \Bigg\}
\end{align} 
for some constant $c>0$ and $0 < \varepsilon < 1$, where $\beta^* = l+1$ with $l$ the order of the kernel. They work on the scale of H\"older-Zygmund rather than H\"older classes. For this reason they include the corresponding bias upper bound condition which is not automatically satisfied for $\beta=\beta^*$ in that case. 

\begin{remark} \label{second_der}
As mentioned in \cite{gine_nickl2010}, if $K(\cdot) = \frac{1}{2} \mathbbm{1} \{Ê\cdot \in [-1,1]Ê\}$ is the rectangular kernel, the set of all twice differentiable densities $p \in \mathcal P(2,L)$ that are supported in a fixed compact interval $[a,b]$ satisfies \eqref{self_sim_global} with a constant $c > 0$. The reason is that due to the constraint of being a probability density, $\Vert p'' \Vert_{\sup}$  is bounded away from zero uniformly over this class, in particular $p''$ cannot vanish everywhere.
\end{remark}

\noindent A localized version of the self-similarity condition characterizing the above class reads as follows.

\smallskip
\noindent
For any nondegenerate interval $(u,v) \subset [0,1]$, there exists some $\beta \in [\beta_*, \beta^*]$ with $p_{\arrowvert (u,v)}\in\mathcal P_{(u,v)}(\beta,L^*)$ and
\begin{align}\label{self_sim_local}
c \cdot 2^{-j\beta}Ê\leq \sup_{s \in (u+2^{-j}, v - 2^{-j})}Ê\vert (K_j \ast p)(s)-p(s) \vert
\end{align}
for all $j \geq j_0\vee \log_2(1/(v-u))$.

\begin{remark} \label{remark1} Inequality \eqref{self_sim_local} can be satisfied only for
\begin{align*}
	\tilde \beta = \tilde \beta_p(U) = \sup \left\{ \beta \in (0,\infty] : p_{\vert U} \in \mathcal H_{U}( \beta) \right\}.
\end{align*}
The converse is not true, however.

\smallskip
	\indent $(i)$ There exist functions $p: U \to \R$, $U \subset \R$ some interval, which are not H\"older continuous to their exponent $\tilde \beta$. The Weierstra{\ss} function $W_1: U \to \R$ with
	\begin{align*}
		W_1(\cdot) = \sum_{n=0}^{\infty} 2^{-n} \cos \left( 2^n \pi \, \cdot \right)
	\end{align*}
	is such an example. Indeed, \cite{hardy1916} proves that
	\begin{align*}
		W_1(x+h)-W_1(x) = O\left( \vert h \vert \log \left( \frac{1}{\vert h \vert}Ê\right) \right),
	\end{align*}
	which implies the H\"older continuity to any parameter $\beta < 1$, hence $\tilde \beta \geq 1$. Moreover, he shows in the same reference that $W_1$ is nowhere differentiable, meaning that it cannot be Lipschitz continuous, that is $\tilde\beta=1$ but $W_1 \notin \mathcal H_U(\tilde \beta)$. 
		
	\smallskip
	$(ii)$ It can also happen that $p_{\vert U}Ê\in \mathcal H_U(\tilde \beta)$ but
\begin{align} \label{klein_lip}
	\limsup_{\delta \to 0}Ê\sup_{\substack{\vert x-y \vert \leq \delta \\ x,y \in U}}Ê\frac{\vert p^{(\lfloor \tilde \beta \rfloor)}(x) - p^{(\lfloor \tilde\beta \rfloor)}(y) \vert}{\vert x-y \vert^{\tilde\beta-\lfloor \tilde\beta \rfloor}} = 0,
\end{align}
meaning that the left-hand side of \eqref{self_sim_local} is violated. In the analysis literature, the subset of functions in $\mathcal H_U(\tilde\beta)$ satisfying \eqref{klein_lip} is called little Lipschitz (or little H\"older) space. As a complement of an open and dense set, it forms a nowhere dense subset of $\mathcal H_U(\tilde \beta)$. 
\end{remark}

\noindent Due to the localization, a condition like \eqref{self_sim_local} rules out examples which seem to be typical to statisticians. Assume that $K$ is a kernel of order $l$ with $l \geq 1$, and recall $\beta^* = l+1$. Then \eqref{self_sim_local} excludes for instance the triangular density in Figure \ref{fig:tri_den} because both sides are linear, in particular the second derivative exists and vanishes when restricted to an interval $U$ which does not contain the maximal point. In contrast to the observation in Remark \ref{second_der}, $\Vert p''Ê\Vert_{U}$ may vanish for subintervals $U \subset [a,b]$. For the same reason, densities with a constant piece are excluded. In general, if $p$ restricted to the $2^{-j_0}$-enlargement of $U$ is a polynomial of order at most $l$, \eqref{self_sim_local} is violated as the left-hand side is not equal to zero. In view of these deficiencies, a condition like \eqref{self_sim_local} is insufficient for statistical purposes. 

\smallskip
\noindent To circumvent this deficit, we introduce $\Vert \cdot \Vert_{\beta,\beta^*,U}$ by
\begin{align} \label{mod_hoelder}
	\Vert p \Vert_{\beta,\beta^*,U} = \sum_{k=0}^{\lfloor \beta \wedge \beta^* \rfloor} \big\Vert p^{(k)} \big\Vert_{U} + \sup_{\substack{x,y \, \in \, U \\ x \neq y}} \frac{\left\vert p^{(\lfloor \beta \wedge \beta^* \rfloor)}(x) - p^{(\lfloor \beta \wedge \beta^*\rfloor)}(y) \right\vert}{\vert x-y \vert^{\beta-\lfloor \beta \wedge \beta^* \rfloor}}
\end{align}
for $\beta > 0$ and for some bounded open subinterval $U \subset \R$. As verified in Lemma~\ref{ineq_holdernorm1}, $\Vert p \Vert_{\beta_1,\beta^*,U}Ê\leq \Vert p \Vert_{\beta_2,\beta^*,U}$ for $0 < \beta_1 \leq \beta_2 < \infty$ and $\vert U \vertÊ\leq 1$. With the help of $\Vert \cdot \Vert_{\beta,\beta^*,U}$, we formulate a localized self-similarity type condition in the subsequent Assumption \ref{self_sim}, which does not  exclude these prototypical densities as mentioned above. For any bounded open interval $U \subset \R$, let $\mathcal H_{\beta^*,U}(\beta,L)$ be the set of functions $p:U\to\R$ admitting derivatives up to the order $\lfloor \beta \wedge \beta^* \rfloor$ with $\Vert p \Vert_{\beta,\beta^*,U} \leq L$. Moreover, $\mathcal H_{\beta^*,U}(\beta)$ is the set of functions $p : U \to \R$, such that $\Vert p \Vert_{\beta,\beta^*,U}$ is well-defined and finite. Correspondingly, $\mathcal H_{\beta^*,U}(\infty,L) = \bigcap_{\beta>0}Ê\mathcal H_{\beta^*,U}(\beta,L)$ and $\mathcal H_{\beta^*,U}(\infty) = \bigcap_{\beta>0}Ê\mathcal H_{\beta^*,U}(\beta)$. Define furthermore
\begin{align} \label{betaU}
	\beta_p(U) = \sup \left\{ \beta \in (0,\infty] : p_{\vert U} \in \mathcal H_{\beta^*,U}(\beta,L^*) \right\}.
\end{align}
\begin{remark} \label{stern}
If for some open interval $U \subset [0,1]$ the derivative $p_{\vert U}^{(\beta^*)}$ exists and
$$p_{\vert U}^{(\beta^*)}Ê\equiv 0,$$ then $\Vert p \Vert_{\beta,\beta^*,U}$ is finite uniformly over all $\beta > 0$. If 
$$p_{\vert U}^{(\beta^*)}Ê\not\equiv 0,$$ then $ \beta^*$, $\Vert p \Vert_{\beta,\beta^*,U}$ is finite if and only if $\beta \leq \beta^* $ as a consequence of the mean value theorem. That is, $\beta_p(U) \in (0,\beta^*] \cup \{Ê\infty \}$.
\end{remark}



\begin{assumption} \label{self_sim}
For sample size $n \in \N$, some $0<\varepsilon<1$, $0<\beta_*<1$, and $L^* > 0$, a~density $p$ is said to be admissible if $p \in \mathcal P_{(-\varepsilon,1+\varepsilon)}(\beta_*,L^*)$ and the following holds true: for any $t \in [0,1]$ and for any $h \in \mathcal G_{\infty}$ with
$$ \mathcal G_{\infty} =  \{ 2^{-j} : j \in \N, \, j \geq j_{\min} = \lceil 2 \vee \log_2 (2/\varepsilon) \rceil \},$$
there exists some $\beta \in [\beta_*,\beta^*] \cup \{Ê\inftyÊ\}$ such that the following conditions are satisfied for $u = h$ or $u = 2h$:
\begin{align} \label{assumption1}
	p_{\vert B(t,u)} \in \mathcal H_{\beta^*,B(t,u)}(\beta,L^*)
\end{align}
and
\begin{align} \label{admissible_density}
	\sup_{s \in B(t,u-g)} \left\vert (K_g \ast p)(s) - p(s) \right\vert \geq \frac{g^{\beta}}{\log n}
\end{align}
for all $g \in \mathcal G_{\infty}$ with $g \leq u/8$.

\smallskip
\noindent The set of admissible densities is denoted by $\mathscr P_n^{\adm} = \mathscr P_n^{\adm}(K,\beta_*,L^*,\varepsilon)$.
\end{assumption}

\begin{lemma} \label{remark2}
Any admissible density $p \in \mathscr P_n^{\adm}(K,\beta_*,L^*,\varepsilon)$ can satisfy \eqref{assumption1} and \eqref{admissible_density} for $\beta = \beta_p(B(t,u))$ only.
\end{lemma}

\noindent By construction, the collection of admissible densities is increasing with the number of observations, that is $\mathscr P_n^{\adm} \subset \mathscr P_{n+1}^{\adm}$, $n \in \N$. The logarithmic denominator even weakens the assumption for growing sample size, permitting smaller and smaller Lipschitz constants. 



\begin{remark}
Assumption \ref{self_sim} does not require an admissible function to be totally "unsmooth" everywhere. For instance, if $K$ is the rectangular kernel  and $L^*$ is sufficiently large, the triangular density as depicted in Figure \ref{fig:tri_den} is (eventually -- for sufficiently large $n$) admissible. It is globally not smoother than Lipschitz, and the bias lower bound condition \eqref{admissible_density} is (eventually) satisfied for $\beta=1$ and pairs $(t,h)$ with $\vert t - 1/2 \vert \leq (7/8)h$. Although the bias lower bound condition to the exponent $\beta^*=2$ is not satisfied for any $(t,h)$ with $t \in [0,1]Ê\setminus (1/2-h,1/2+h)$, these tuples $(t,h)$ fulfill \eqref{assumption1} and \eqref{admissible_density} for $\beta = \infty$, which is not excluded anymore by the new Assumption \ref{self_sim}. Finally, if the conditions \eqref{assumption1} and \eqref{admissible_density} are not simultaneously satisfied for some pair $(t,h)$ with
\begin{align*}
	\frac{7}{8} h < \left\vert t - \frac{1}{2} \right\vert < h,
\end{align*}
then they are fulfilled for the pair $(t,2h)$ and $\beta = 1$, because $\vert t - 1/2 \vert < (7/8)2h$.
\end{remark}

\noindent In view of this remark, it is crucial not to require \eqref{assumption1} and \eqref{admissible_density} to hold for \textit{every} pair $(t,h)$. We now denote by
\begin{align*}
	\mathscr P_n &= \mathscr P_n(K,\beta_*,L^*,\varepsilon,M) = \left\{p \in  \mathscr P_n^{\adm}(K,\beta_*,L^*,\varepsilon): \inf_{x\in[-\varepsilon,1+\varepsilon]}p(x) \geq M \right\} 
\end{align*}
the set of admissible densities  being bounded below by $M>0$ on $[-\varepsilon,1+\varepsilon]$. We restrict our considerations to combinations of parameters for which the class $\mathscr P_n$ is non-empty.

\smallskip
\noindent The remaining results of this subsection are about the massiveness of the function classes $\mathscr P_n$.  They are stated for the particular case of the rectangular kernel. Other kernels may be treated with the same idea; verification of \eqref{admissible_density} however appears to require a case-by-case analysis for different kernels. The following proposition demonstrates that the pointwise minimax rate of convergence remains unchanged when passing from the class $\mathcal H(\beta,L^*)$ to $\mathscr P_n \cap \mathcal H(\beta,L^*)$. 

\begin{proposition}[Lower pointwise risk bound] \label{low_ptw_riskbound}
For the rectangular kernel $K_R$ there exists some constant $M>0$, such that for any $t \in [0,1]$, for any $\beta \in [\beta_*,1]$, for any $0<\varepsilon<1$, and for any $k \geq k_0(\beta_*)$ there exists some $x>0$ and some $L(\beta)>0$ with
\begin{align*}
	\inf_{T_n} \sup_{\substack{p \in \mathscr P_k : \\ p_{\vert (-\varepsilon,1+\varepsilon)} \in \mathcal H_{(-\varepsilon,1+\varepsilon)}(\beta,L)}} \P_p^{\otimes n} \left( n^{\frac{\beta}{2\beta+1}}Ê\left\vert T_n(t) - p(t) \right\vert \geq x \right) > 0
\end{align*}
for all $L \geq L(\beta)$, for the class $\mathscr P_k = \mathscr P_k(K_R,\beta_*,L^*,\varepsilon,M)$, where the infimum is running over all estimators $T_n$ based on $X_1,\dots, X_n$.
\end{proposition} 

\noindent Note that the classical construction for the sequence of hypotheses in order to prove minimax lower bounds consists of a smooth density distorted by small $\beta$-smooth perturbations, properly scaled with the sample size $n$. However, there does not exist a fixed constant $c>0$, such that all of its members are contained in the class \eqref{self_sim_global}. Thus, the  constructed hypotheses in our proof are substantially more complex, for which reason we restrict attention to $\beta \leq 1$.

\smallskip
\noindent Although Assumption \ref{self_sim} is getting weaker for growing sample size, some densities are permanently excluded from consideration. The following proposition states that the exceptional set of permanently excluded densities is pathological.

\begin{proposition} \label{nowhere_dense} 
	For the rectangular kernel $K_R(\cdot) = \frac{1}{2}Ê\mathbbm{1}Ê\{ \cdot \in [-1,1]Ê\}$, let
	\begin{align*}
		\mathscr R = \bigcup_{n \in \N}Ê\mathscr P_n^{\adm}(K_R,\beta_*,L^*,\varepsilon).
	\end{align*}
	Then, for any $t \in [0,1]$, for any $h \in \mathcal G_{\infty}$ and for any $\beta \in [\beta_* , 1)$, the set
	$$ \mathcal P_{B(t,h)}(\beta, L^*) \setminus \mathscr R_{\vert B(t,h)} $$
	is nowhere dense in 
	$\mathcal P_{B(t,h)}(\beta, L^*)$
	 with respect to $\Vert \cdot \Vert_{\beta,B(t,h)}$. 
\end{proposition}

\noindent Among more involved approximation steps, the proof reveals the existence of functions with the same regularity in the sense of Assumption \ref{self_sim} on \textit{every} interval for $\beta \in (0,1)$. This property is closely related to but does not coincide with the concept of mono-H\"older continuity from the analysis literature, see for instance \cite{barral_durand_jaffard_seuret2013}. \cite{hardy1916} shows that the Weierstra{\ss} function is mono-H\"older continuous for $\beta\in(0,1)$. For any $\beta\in(0,1]$, the next lemma shows that Weierstra{\ss}' construction
\begin{align} \label{wf}
	W_{\beta}(\cdot) = \sum_{n=0}^\infty 2^{-n\beta} \cos (2^n \pi \, \cdot )
\end{align}
satisfies the bias condition \eqref{admissible_density} for the rectangular kernel to the exponent $\beta$ on any subinterval $B(t,h)$, $t \in [0,1]$, $h \in \mathcal G_{\infty}$. 

\begin{lemma} \label{weier_lemma}
For all $\beta \in (0,1)$, the Weierstra{\ss} function $W_{\beta}$ as defined in \eqref{wf} satisfies $W_{\beta \vert U} \in \mathcal H_U(\beta,L_W)$ with some Lipschitz constant $L_W = L_W(\beta)$ for every open interval $U$. For the rectangular kernel $K_R$ and $\beta \in (0,1]$, the Weierstra{\ss} function fulfills the bias lower bound condition
\begin{align*}
	\sup_{s \in B(t,h-g)} \left\vert (K_{R,g} \ast W_{\beta})(s) - W_{\beta}(s) \right\vert 
		&> \left( \frac{4}{\pi}-1 \right) g^{\beta}
\end{align*}
for any $t \in \R$ and for any $g,h \in \mathcal G_{\infty}$ with $g \leq h/2$.
\end{lemma}

\noindent The whole scale of parameters $\beta \in[\beta_*,1]$ in Proposition \ref{nowhere_dense}  can be covered by passing over from H\"older classes to H\"older-Zygmund classes in the definition of $\mathscr P_n$. 
Although the Weierstra{\ss}  function $W_1$ in \eqref{wf} is not Lip\-schitz, a classical result, see \cite{heurteaux2005} or \cite{mauldin_williams1986} and references therein, states that $W_1$ is indeed contained in the Zygmund class $\Lambda_1$. That is, it satisfies
\begin{align*}
	\vert W_1(x+h) - W_1(x-h) - 2W_1(x) \vert \leq C \vert h \vert
\end{align*}
for some $C>0$ and for all $x \in \R$ and for all $h>0$. Due to the symmetry of the rectangular kernel $K_R$, it therefore fulfills the bias upper bound
$$
\left\Vert  K_{R,g} \ast W_{1} - W_{1} \right\Vert_{\sup} \leq C'g^{\beta}\ \ \ \text{for all $g\in (0,1]$}.
$$
The local adaptivity theory can be likewise developed on the scale of H\"older-Zygmund rather than H\"older classes -- here, we restrict attention to H\"older classes because they are commonly considered in the theory of kernel density estimation.

\subsection{Construction of the confidence band} \label{Subsec:construction}

The new confidence band is based on a kernel density estimator with variable bandwidth incorporating a localized but not the fully pointwise \cite{lepski1990} bandwidth selection procedure. A suitable discretization and a locally constant approximation allow to piece the pointwise constructions together in order to obtain a continuum of confidence statements. The complex construction makes the asymptotic calibration of the confidence band to the level $\alpha$ non-trivial. Whereas the analysis of the related globally adaptive procedure of \cite{gine_nickl2010} reduces to the limiting distribution of the supremum of a stationary Gaussian process, our locally adaptive approach leads to a highly non-stationary situation. An essential component is therefore the identification of a stationary process as a least favorable case by means of Slepian's comparison inequality.

\medskip
\noindent We now describe the procedure. The interval $[0,1]$ is discretized into equally spaced grid points, which serve as evaluation points for the locally adaptive estimator.
 We discretize by a mesh of width
\begin{align*}
	\delta_n = \left\lceil 2^{\frac{j_{\min}}{\beta_*}} \left(\frac{\log \tn}{\tn} \right)^{-\kappa_1} (\log \tn)^{\frac{2}{\beta_*}} \right\rceil^{-1}
\end{align*}
with $\kappa_1 \geq 1/(2\beta_*)$ and set $\mathcal H_n = \{ k\delta_n : k \in \Z \}$. Fix now constants
\begin{align} \label{c22}
	c_1 > \frac{2}{\beta_* \log 2} \quad \text{and}Ê\quad \kappa_2 > c_1 \log 2 + 4.
\end{align}
Consider the set of bandwidth exponents
\begin{align*}
	\mathcal J_n = \left\{ j \in \mathbb{N} \ : \  j_{\min} \leq j \leq j_{\max} = \left\lfloor \log_2 \left(\frac{\tn}{(\log \tn)^{\kappa_2}}\right) \right\rfloor \right\}.
\end{align*} 
The bound $j_{\min}$ ensures that $2^{-j}Ê\leq \varepsilon \wedge (1/4)$ for all $j \in \mathcal J_n$, and therefore avoids that infinite smoothness in \eqref{beta_n} and the corresponding local parametric rate is only attainable in trivial cases as the interval under consideration is $[0,1]$. The bound $j_{\max}$ is standard and particularly guarantees consistency of the kernel density estimator with minimal bandwidth within the dyadic grid of bandwidths
\begin{align*} 
	\mathcal{G}_n = \Big\{ 2^{-j} : j \in \mathcal{J}_n \Big\}.
\end{align*}
We define the set of admissible bandwidths for $t \in [0,1]$ as
\begin{align} \label{setA}
\begin{split}
	\mathcal{A}_n(t) = \Big\{ j \in \mathcal{J}_n : \max_{s \in B\left(t,\frac{7}{8} \cdot 2^{-j}\right) \cap \mathcal H_n} \left\vert \hat{p}_{n}^{(2)} (s,m) - \hat{p}_{n}^{(2)}(s,m') \right\vert \leq c_2 \sqrt{\frac{\log \tn}{\tn 2^{- m}}} \\ \text{ for all } m,m' \in \mathcal{J}_n \text{ with } m > m' > j+2 \Big\},
\end{split}
\end{align}
with constant $c_2 = c_2(A, \nu, \beta_*, L^*, K,\varepsilon)$ specified in the proof of Proposition \ref{bw_band}. Furthermore, let
\begin{align} \label{lepski_bw}
	\hat{j}_n(t) = \min \mathcal{A}_n(t), \quad t \in [0,1],
\end{align}
and $\hat h_n(t) = 2^{-\hat j_n(t)}$. Note that a slight difference to the classical Lepski procedure is the additional maximum in \eqref{setA}, which reflects the idea of adapting localized but not completely pointwise for fixed sample size $n$. The bandwidth \eqref{lepski_bw} is determined for all mesh points $k\delta_n, k \in T_n = \{1, \ldots, \delta_n^{-1} \}$ in $[0,1]$, and set piecewise constant in between. Accordingly, with
\begin{align*}
	\hat h_{n,1}^{loc}(k) &= 2^{-\hat j_n((k-1) \delta_n) - u_n}, \quad \hat h_{n,2}^{loc}(k) = 2^{-\hat j_n(k \delta_n) - u_n},
\end{align*}
where $u_n = c_1 \log \log \tn$ is some sequence implementing the undersmoothing, the estimators are defined as
\begin{align} \label{loc_const}
\begin{split}
	\hat h_n^{loc}(t) &= \hat h_{n,k}^{loc} = \min \left\{ \hat h_{n,1}^{loc}(k), \hat h_{n,2}^{loc}(k) \right\} \qquad \text{and} \\
	\hat p_n^{loc}(t,h) &= \hat p_n^{(1)}(k\delta_n,h)
\end{split}
\end{align}
for $t \in I_k = [(k-1) \delta_n, k \delta_n)$, $k \in T_n \setminus \{Ê\delta_n^{-1}Ê\}$, $I_{\delta_n^{-1}} = [1- \delta_n, 1]$. The following theorem lays the foundation for the construction of honest and locally adaptive confidence bands.

\begin{theorem}[Least favorable case] \label{limitdistribution}
For the estimators defined in \eqref{loc_const} and normalizing sequences
\begin{align*} 
	a_n = c_3 (-2 \log \delta_n)^{1/2}, \quad b_n =  \frac{3}{c_3} \left\{ (-2 \log \delta_n)^{1/2} - \frac{\log(-\log \delta_n) + \log 4\pi}{2 (-2\log \delta_n)^{1/2}} \right\},
\end{align*} 
with $c_3 = \sqrt{2}/TV(K)$, it holds
\begin{align*}
		&\liminf_{n \to \infty} \inf_{p \in \mathscr P_n} \P_p^{\otimes n} \left( a_n \left( \sup_{t \in [0,1]} \sqrt{\tn \hat h_n^{loc}(t)} \left\vert \hat p_n^{loc}(t, \hat h_n^{loc}(t)) - p(t) \right\vert - b_n \right) \leq x \right) \\
		&\hspace{8cm} \geq 2\,\P\Big(\sqrt{L^*} G \leq x\Big)-1
\end{align*}
for some standard Gumbel distributed random variable $G$.
\end{theorem}

\noindent The proof of Theorem \ref{limitdistribution} is based on several completely non-asymptotic approximation techniques. The asymptotic Koml\'os-Major-Tusn\'ady-approximation technique, used in \cite{gine_nickl2010}, has been evaded using non-asymptotic Gaussian approximation results recently developed in \cite{chernozhukov_chetverikov_kato2014Approx}. The essential component of the proof of Theorem \ref{limitdistribution} is the application of Slepian's comparison inequality to reduce considerations from a non-stationary Gaussian process to the least favorable case of a maximum of $\delta_n^{-1}$ independent and identical standard normal random variables. 

\noindent With $q_{1-\alpha/2}$ denoting the $(1-\alpha/2)$-quantile of the standard Gumbel distribution, we define the confidence band as the family of piecewise constant random intervals $C_{n,\alpha} = (C_{n,\alpha}(t))_{t \in [0,1]}$ with
\begin{align} \label{conf_band}
\begin{split}
	C_{n,\alpha}(t) = \left[ \hat p_n^{loc}(t, \hat h_n^{loc}(t)) - \frac{q_n(\alpha)}{\sqrt{\tn \hat h_n^{loc}(t)}}, \ \ \hat p_n^{loc}(t, \hat h_n^{loc}(t)) + \frac{q_n(\alpha)}{\sqrt{\tn \hat h_n^{loc}(t)}} \right]
\end{split}
\end{align}
and
\begin{align*}
	q_n(\alpha) = \frac{\sqrt{L^*} \cdot q_{1-\alpha/2}}{a_n} + b_n.
\end{align*}
For fixed $\alpha > 0$, $q_n(\alpha) = O(\sqrt{\log n})$ as $n$ goes to infinity.


\begin{corollary}[Honesty]
The confidence band as defined in \eqref{conf_band} satisfies
	\begin{align*}
\liminf_{n\to\infty}Ê\inf_{p \in \mathscr P_n} \P_p^{\otimes n} \Big( p(t) \in C_{n,\alpha}(t) \text{ for every } t \in [0,1] \Big) \, \geq \,1-\alpha.
	\end{align*}
\end{corollary}

\subsection{Local H\"older regularity and local adaptivity} \label{subsec:3.3}
In the style of global adaptivity in connection with confidence sets one may call a confidence band $C_{n,\alpha} = (C_{n,\alpha}(t))_{t \in [0,1]}$ locally adaptive if for every interval $U \subset [0,1]$,
\begin{align*}
	\limsup_{n \to \infty}\sup_{p \in {\mathscr {P}_n}_{\arrowvert U_{\delta}} \cap \mathcal H_{\beta^*,U_{\delta}}(\beta,L^*)} \P_p^{\chi_2} \Big( \vert C_{n,\alpha}(t) \vert \geq \eta \cdot  r_n(\beta) \text{ for some }Êt \in U \Big) \rightarrow 0
\end{align*}
as $\eta \to \infty$, for every $\beta \in [\beta_*,\beta^*]$, where $U_{\delta}$ is the open $\delta$-enlargement of $U$. As a consequence of the subsequently formulated Theorem \ref{loc_adapt}, our confidence band satisfies this notion of local adaptivity up to a logarithmic factor. However, in view of the imagination illustrated in Figure \ref{fig:tri_den} the statistician aims at a stronger notion of adaptivity, where the asymptotic statement is not formulated for an arbitrary but fixed interval $U$ only. Precisely, the goal would be to adapt even to some pointwise or local H\"older regularity, two well established notions from analysis.

\begin{definition}[Pointwise H\"older exponent, \cite{seuret_levyvehel2002}]
Let $p: \R \to \R$ be a function, $\beta>0$, $\beta \notin \N$, and $t \in \R$. Then $p \in \mathcal H_t(\beta)$ if and only if there exists a real $R>0$, a polynomial $P$ with degree less than $\lfloor \beta \rfloor$, and a constant $c$ such that
\begin{align*}
	\vert p(x) - P(x-t) \vert \leq c \vert x-t \vert^{\beta}
\end{align*}
for all $x \in B(t,R)$. The pointwise H\"older exponent is denoted by
\begin{align*}
	\beta_p(t) = \sup \{ \beta : p \in \mathcal H_t(\beta) \}.
\end{align*}
\end{definition}

\begin{definition}[Local H\"older exponent, \cite{seuret_levyvehel2002}] \phantom{a} \\
Let $p: \Omega \to \R$ be a function and $\Omega \subset \R$ an open set. One classically says that $p \in \mathcal H_{loc}(\beta, \Omega)$, where $0<\beta<1$, if there exists a constant $c$ such that
\begin{align*}
	\left\vert p(x) - p(y) \right\vert \leq c \vert x-y \vert^{\beta}
\end{align*}
for all $x,y \in \Omega$. If $m < \beta < m+1$ for some $m \in \N$, then $p \in \mathcal H_{loc}(\beta,\Omega)$ means that there exists a constant $c$ such that 
\begin{align*}
	\left\vert \partial^m p(x) - \partial^m p(y) \right\vert \leq c \vert x-y \vert^{\beta-m}
\end{align*}
for all $x,y \in \Omega$. Set now
\begin{align*}
	\beta_p(\Omega) = \sup \{ \beta : p \in \mathcal H_{loc}(\beta,\Omega) \}. 
\end{align*}
Finally, the local H\"older exponent in $t$ is defined as
\begin{align*}
	\beta_p^{loc}(t) = \sup \{ \beta_p(O_i) : i \in I \},
\end{align*}
where $(O_i)_{i \in I}$ is a decreasing family of open sets with $\cap_{i \in I} O_i = \{ t \}$. [By Lemma~2.1 in \cite{seuret_levyvehel2002}, this notion is well defined, that is, it does not depend on the particular choice of the decreasing sequence of open sets.]
\end{definition}

\noindent The next proposition however shows that attaining the minimax rates of convergence corresponding to the pointwise or local H\"older exponent (possibly inflated by some logarithmic factor) uniformly over $\mathscr{P}_n$ is an unachievable goal.

\begin{proposition} \label{optimal_adaptation}
For the rectangular kernel $K_R$ there exists some constant $M>0$, such that for any $t \in [0,1]$, for any $\beta \in [\beta_*,1]$, for any $0<\varepsilon<1$, and for any $k \geq k_0(\beta_*)$ there exists some $x>0$ and constants $L=L(\beta)>0$ and $c_4°	 = c_4°(\beta)>0$ with
\begin{align*}
	\inf_{T_n} \sup_{p \in \mathscr S_k(\beta)} \P_p^{\otimes n} \left( n^{\frac{\beta}{2\beta+1}} \left\vert T_n(t) - p(t) \right\vert \geq x \right) > 0\ \ \ \text{for all $k \geq k_0(\beta_*)$}
\end{align*}
with
\begin{align*}
	\mathscr S_k(\beta) &= \mathscr S_k(L,\beta,\beta_*,M,K_R,\varepsilon)  \\
	&= \Big\{ p \in \mathscr P_k(K_R,\beta_*,L,\varepsilon,M) : \exists \, r \geq c_4° \, n^{-\frac{1}{2\beta+1}} \\
	&\hspace{2.5cm} \text{ such that } \, p_{\vert B(t,r)} \in \mathcal H_{B(t,r)}(\infty,L) \Big\} \cap \mathcal H_{(-\varepsilon,1+\varepsilon)}(\beta,L),
\end{align*}
where the infimum is running over all estimators $T_n$ based on $X_1,\dots, X_n$.
\end{proposition}

\noindent The proposition furthermore reveals that if a density $p \in \mathscr P_k$ is H\"older smooth to some exponent $\eta > \beta$ on a ball around $t$ with radius at least of the order $n^{-1/(2\beta+1)}$, then no estimator for $p(t)$ can achieve a better rate than $n^{-\beta/(2\beta+1)}$. We therefore
introduce an $n$-dependent statistical notion of local regularity for any point $t$. Roughly speaking, we intend it to be the maximal $\beta$ such that the density attains this H\"older exponent within $B(t,h_{\beta,n})$, where $h_{\beta,n}$ is of the optimal adaptive bandwidth order $(\log n/n)^{1/(2\beta+1)}$. 
We realize this idea with $\Vert \cdot \Vert_{\beta,\beta^*,U}$ as defined in \eqref{mod_hoelder} and used in Assumption \ref{self_sim}. 

\begin{definition}[$n$-dependent local H\"older exponent] \label{ndep_hoelder}
With the classical optimal bandwidth within the class $\mathcal H(\beta)$
\begin{align*} 
	h_{\beta,n} = 2^{-j_{\min}} \cdot \left( \frac{\log \tn}{\tn} \right)^{\frac{1}{2\beta +1}},
\end{align*}
define the class $\mathcal H_{\beta^*,n,t}(\beta,L)$ as the set of functions $p : B(t,h_{\beta,n}) \to \R$, such that $p$ admits derivatives up to the order $\lfloor \beta \wedge \beta^* \rfloor$ and $\Vert p \Vert_{\beta,\beta^*,B(t,h_{\beta,n})} \leq L$, and $\mathcal H_{\beta^*,n,t}(\beta)$ the class of functions $p : B(t,h_{\beta,n}) \to \R$ for which $\Vert p \Vert_{\beta,\beta^*,B(t,h_{\beta,n})}$ is well-defined and finite. The $n$-dependent local H\"older exponent for the function $p$ at point $t$ is defined as
\begin{align} \label{beta_n}
	 \beta_{n,p}(t) = \sup \Big\{ \beta > 0 : p_{\vert B(t,h_{\beta,n})} \in \mathcal H_{\beta^*,n,t }(\beta,L^*) \Big\}.
\end{align}
If the supremum is running over the empty set, we set $\beta_{n,p}(t)=0$.
 \end{definition}

\noindent Finally, the next theorem shows that the confidence band adapts to the $n$-dependent local H\"older exponent.

\begin{theorem}[Strong local adaptivity] \label{loc_adapt}
There exists some $\gamma = \gamma(c_1)$, such that
\begin{align*}
	\limsup_{n \to \infty} \sup_{p \in \mathscr P_n} \P_p^{\chi_2} \left( \sup_{t \in [0,1]} \vert C_{n,\alpha}(t) \vert \cdot \left( \frac{\log \tn}{\tn} \right)^{-\frac{\beta_{n,p}(t)}{2\beta_{n,p}(t)+1}} \geq (\log \tn)^{\gamma} \right) = 0.
\end{align*}
\end{theorem}

\noindent Note that the case $\beta_{n,p}(t) = \infty$ is not excluded in the formulation of Theorem \ref{loc_adapt}. That is, if $p_{\vert U}$ can be represented as a polynomial of degree strictly less than~$\beta^*$, the confidence band attains even adaptively the parametric width $n^{-1/2}$, up to logarithmic factors. In particular, the band can be tighter than $n^{-\beta^*/(2\beta^*+1)}$. In general, as long as $\delta \leq \varepsilon$ and $B(t, h_{\beta^*,n}) \subset U_{\delta}$,
\begin{align*}
	\beta_{n,p}(t) \geq \beta_p(U_{\delta}) \quad \text{ for all } t \in U.
\end{align*}

\begin{corollary}[Weak local adaptivity]
For every interval $U \subset [0,1]$,
\begin{align*}
	\limsup_{n \to \infty}\sup_{p \in {\mathscr {P}_n}_{\arrowvert U_{\delta}} \cap \mathcal H_{\beta^*,U_{\delta}}(\beta,L^*)} \P_p^{\chi_2} \left( \sup_{t \in U} \vert C_{n,\alpha}(t) \vert \geq \left( \frac{\log \tn}{\tn}Ê\right)^{\frac{\beta}{2\beta+1}} (\log \tn)^{\gamma} \right) 
\end{align*}
is equal to zero for every $\beta \in [\beta_*,\beta^*]$, where $U_{\delta}$ is the open $\delta$-enlargement of $U$ and $\gamma$ as in Theorem~\ref{loc_adapt}.
\end{corollary}

%
%

\section{Supplementary notation and results} \label{sec:aux} 

The following auxiliary results are crucial ingredients in the proofs of Theorem \ref{limitdistribution} and Theorem \ref{loc_adapt}. \\
\noindent Recalling the quantity $h_{\beta,n}$ in Definition \ref{ndep_hoelder}, Proposition \ref{bw_band} shows that $2^{-\hat j_n(\cdot)}$ lies in a band around 
\begin{align} \label{h_opt}
	\bar{h}_n(\cdot) = h_{\beta_{n,p}(\cdot),n}
\end{align}
uniformly over all admissible densities $p \in \mathscr P_n$. Proposition \ref{bw_band} furthermore reveals the necessity to undersmooth, which has been already discovered by \cite{bickel_rosenblatt1973}, leading to a bandwidth deflated by some logarithmic factor. Set now
\begin{align*}
	\bar{j}_n(\cdot) = \left\lfloor \log_2 \left( \frac{1}{\bar{h}_n(\cdot)} \right) \right\rfloor + 1, 
\end{align*}
such that the bandwidth $2^{-\bar{j}_n(\cdot)}$ is an approximation of $\bar{h}_n(\cdot)$ by the next smaller bandwidth on the grid $\mathcal{G}_n$ with
$$\frac{1}{2} \bar{h}_n(\cdot) \leq 2^{-\bar{j}_n(\cdot)} \leq \bar{h}_n(\cdot).$$

\medskip

\noindent The next proposition states that the procedure chooses a bandwidth which simultaneously in the location $t$ is neither too large nor too small.

\begin{proposition} \label{bw_band}
	The bandwidth $\hat j_n(\cdot)$ defined in \eqref{lepski_bw} satisfies
	\begin{align*}
		\lim_{n \to \infty} \sup_{p \in \mathscr P_n} \left\{ 1 - \P_p^{\chi_2} \left( \hat j_n(k\delta_n) \in \Big[ k_n(k\delta_n), \, \bar j_n(k\delta_n)+1 \Big] \text{ for all } k \in T_n \right) \right\} = 0
	\end{align*}
	where $k_n(\cdot) = \bar j_n(\cdot) - m_n$, and $m_n = \frac{1}{2} c_1 \log \log \tn$.
\end{proposition}

\begin{lemma} \label{grid_opt}
Let $s,t \in [0,1]$ be two points with $s<t$, and let $z \in (s,t)$. If
\begin{align} \label{stclose}
	\vert s-t \vert \leq \frac{1}{8} \, h_{\beta_*,n}
\end{align}
then
\begin{align*}
	\frac{1}{3} \, \bar h_n(z) \leq \min \left\{ \bar h_n(s), \bar h_n(t) \right\} \leq 3 \, \bar h_n(z).
\end{align*}
\end{lemma}

\begin{lemma} \label{lemma_deviation}
There exist positive and finite constants $c_5° = c_5°(A,\nu,K)$ and $c_6°=c_6°(A,\nu,L^*,K)$, and some $\eta_0 = \eta_0(A, \nu, L^*, K) > 0$, such that
\begin{align*}
	&\sup_{p \in \mathscr P_n} \P_p^{\chi_i} \left( \sup_{s \in \mathcal H_n} \max_{h \in \mathcal G_n} \sqrt{\frac{\tn h}{\log \tn}} \Big\vert \hat{p}_{n}^{(i)} (s,h) - \E_p^{\chi_i} \hat{p}_{n}^{(i)}(s,h) \Big\vert > \eta \right) \leq c_5°  \tn^{- c_6° \eta}, \quad i=1,2
\end{align*}
for sufficiently large $n \geq n_0(A,\nu,L^*,K)$ and for all $\eta \geq \eta_0$.
\end{lemma}


\noindent The next lemma states extends the classical upper bound on the bias for the modified H\"older classes $\mathcal H_{\beta^*,B(t,U)}(\beta,L)$.

\begin{lemma} \label{bias_up}
Let $t \in \R$ and $g,h > 0$. Any density $p:\R\to\R$ with $p_{\vert B(t,g+h)} \in \mathcal H_{\beta^*,B(t,g+h)}(\beta,L)$ for some $0 < \beta \leq \infty$ and some $L>0$ satisfies
\begin{align} \label{hallo}
	\sup_{s \in B(t,g)} \left\vert (K_h \ast p)(s) - p(s) \right\vert \leq b_2 h^{\beta}
\end{align}
for some positive and finite constant $b_2 = b_2(L,K)$.
\end{lemma}

\begin{lemma} \label{bias_zygmund}
For symmetric kernels $K$ and $\beta=1$, the bias bound \eqref{hallo} continues to hold if the Lipschitz balls are replaced by the corresponding  Zygmund balls.
\end{lemma}


\section{Proofs} \label{sec:proofs}

We first prove the results of Section \ref{sec:kap3} in Subsection \ref{subsec:proofkap3} and afterwards proceed with the proofs of the results Section \ref{sec:aux} in Subsection \ref{subsec:prooflemma}. \\
For the subsequent proofs we recall the following notion of the theory of empirical processes.  

\begin{definition}
A class of measurable functions $\mathscr H$ on a measure space $(S, \mathscr S)$ is a $Vapnik$-$\check{C}ervonenkis$ $class$ (VC class) of functions with respect to the envelope $H$ if there exists a measurable function $H$ which is everywhere finite with $\sup_{h \in \mathscr H} \vert h \vert \leq H$ and finite numbers $A$ and $v$, such that
\begin{align*}
	\sup_{Q} N\Big(\mathscr H, \Vert \cdot \Vert_{L^2(Q)}, \varepsilon \Vert H \Vert_{L^2(Q)}\Big) \leq \left( \frac{A}{\varepsilon}Ê\right)^{v}
\end{align*}
for all $0<\varepsilon<1$, where the supremum is running over all probability measures $Q$ on $(S,\mathscr S)$ for which $\Vert H \Vert_{L^2(Q)} < \infty$.
\end{definition}

\noindent \cite{nolan_pollard1987} call a class $Euclidean$ with respect to the envelope $H$ and with characteristics $A$ and $\nu$ if the same holds true with $L^1(Q)$ instead of $L^2(Q)$. The following auxiliary lemma is a direct consequence of the results in the same reference.

\begin{lemma} \label{Eucl_VC}
If a class of measurable functions $\mathscr H$ is Euclidean with respect to a constant envelope $H$ and with characteristics $A$ and $\nu$, then the class 
\begin{align*}
	\tilde{\mathscr H} = \left\{ h - \E_{\P} h : h \in \mathscr H \right\}
\end{align*}
is a VC class  with envelope $2H$ and characteristics $A' = 4\sqrt{A} \vee 2A$ and $\nu' = 3\nu$ for any probability measure $\P$. 
\end{lemma}

\begin{proof}
For any probability measure $\P$ and for any functions $\tilde h_1 = h_1 - \E_{\P} h_1$, $\tilde h_2 = h_2 - \E_{\P} h_2 \in \tilde{\mathscr H}$ with $h_1, h_2 \in \mathscr H$, we have
\begin{align*}
	\Vert \tilde h_1 - \tilde h_2 \Vert_{L^2(Q)} \leq \Vert h_1 - h_2 \Vert_{L^2(Q)} + \Vert h_1 - h_2 \Vert_{L^1(\P)}.
\end{align*}
For any $0<\varepsilon \leq 1$, we obtain as a direct consequence of Lemma 14 in \cite{nolan_pollard1987}
\begin{align} \label{covering_product}
\begin{split}
	&N \left( \tilde{\mathscr H}, L^2(Q),2 \varepsilon \Vert H \Vert_{L^2(Q)} \right) \\
	&\hspace{2cm} \leq N \left( \mathscr H, L^2(Q), \frac{\varepsilon \Vert H \Vert_{L^2(Q)}}{2} \right) \cdot N \left( \mathscr H, L^1(\P), \frac{\varepsilon \Vert H \Vert_{L^1(\P)}}{2} \right).
\end{split}
\end{align}
\cite{nolan_pollard1987}, page 789, furthermore state that the Euclidean class $\mathscr H$ is also a VC class with respect to the envelope $H$ and with
\begin{align*}
	N \left( \mathscr H, L^2(Q), \frac{\varepsilon \Vert H \Vert_{L^2(Q)}}{2} \right) \leq \left( \frac{4\sqrt{A}}{\varepsilon} \right)^{2\nu},
\end{align*}
whereas
\begin{align*}
	N \left( \mathscr H, L^1(\P), \frac{\varepsilon \Vert H \Vert_{L^1(\P)}}{2} \right) \leq \left( \frac{2A}{\varepsilon} \right)^{\nu}.
\end{align*}
Inequality \eqref{covering_product} thus implies
\begin{align*}
	N \left( \tilde{\mathscr H}, L^2(Q), 2\varepsilon \Vert H \Vert_{L^2(Q)} \right) \leq \left( \frac{4\sqrt{A} \vee 2A}{\varepsilon} \right)^{3\nu}.
\end{align*}
\end{proof}

\subsection{Proofs of the results in Section \ref{sec:kap3}} \label{subsec:proofkap3}

\begin{proof}[Proof of Lemma \ref{remark2}]
Let $p \in \mathscr P_n^{\adm}(K,\beta_*,L^*,\varepsilon)$ be an admissible density. That is, for any $t \in [0,1]$ and for any $h \in \mathcal G_{\infty}$ there exists some $\beta \in [\beta_*, \beta^*] \cup \{ \infty \}$, such that for $u=h$ or $u=2h$ both 
\begin{align*}
	p_{\vert B(t,u)} \in \mathcal H_{\beta^*,B(t,u)}(\beta,L^*)
\end{align*}
and
\begin{align*}
	\sup_{s \in B(t,u-g)} \left\vert (K_g \ast p)(s) - p(s) \right\vert \geq \frac{g^{\beta}}{\log n} \quad \text{ for all } g \in \mathcal G_{\infty} \text{ with } g \leq u/8
\end{align*}
hold. By definition of $\beta_p(B(t,u))$ in \eqref{betaU}, we obtain $\beta_p(B(t,u)) \geq \beta$. We now prove by contradiction that also $\beta_p(B(t,u)) \leq \beta$. If $\beta = \infty$, the proof is finished. Assume now that $\beta < \infty$ and that $\beta_p(B(t,u)) > \beta$. Then, by Lemma \ref{ineq_holdernorm1}, there exists some $\beta < \beta' < \beta_p(B(t,u))$ with $p_{\vert B(t,u)} \in \mathcal H_{\beta^*,B(t,u)}(\beta', L^*)$. By Lemma \ref{bias_up}, there exists some constant $b_2 = b_2(L^*,K)$ with
\begin{align*}
	b_2 g^{\beta'} \geq \sup_{s \in B(t,u-g)} \left\vert (K_g \ast p)(s) - p(s) \right\vert \geq \frac{g^{\beta}}{\log n}
\end{align*}
for all $g \in  \mathcal G_{\infty}$ with $g \leq u/8$, which is a contradiction. 
\end{proof}

\begin{proof}[Proof of Proposition \ref{low_ptw_riskbound}]
The proof is based on a reduction of the supremum over the class to a maximum over two distinct hypotheses. 

\paragraph{Part 1}
For $\beta \in [\beta_*, 1)$, the construction of the hypotheses is based on the Weierstra{\ss} function as defined in \eqref{wf} and is depicted in Figure \ref{fig:hypos1}. Consider the function $p_0: \R \to \R$ with
\begin{align*}
	p_0(x) = \begin{cases}
	0, \quad &\text{if } \ \vert x-t \vert \geq \frac{10}{3} \\
	\frac{1}{4} + \frac{3}{16}(x-t+2), \quad &\text{if } \  -\frac{10}{3} < x-t  < -2 \\
	\frac{1}{6} + \frac{1-2^{-\beta}}{12} W_{\beta}(x-t), \quad &\text{if } \  \vert x-t \vert \leq 2 \\
	\frac{1}{4} - \frac{3}{16}(x-t-2), \quad &\text{if } \  2 < x-t  < \frac{10}{3} \\
	\end{cases}
\end{align*}
and the function $p_{1,n} : \R \to \R$ with
\begin{align*}
	p_{1,n}(x) = p_0(x) + q_{t+\frac{9}{4},n}(x) - q_{t,n}(x), \quad x \in \R,
\end{align*}
where
\begin{align*}
	q_{a,n}(x) = \begin{cases}
	0, \quad &\text{if } \ \vert x-a \vert > g_{\beta,n} \\
	\frac{1-2^{-\beta}}{12} \Big(W_{\beta}(x-a) - W_{\beta}(g_{\beta,n})\Big), \quad &\text{if } \ \vert x-a \vert \leq g_{\beta,n}
	\end{cases}
\end{align*}
for $g_{\beta,n} = \frac{1}{4} n^{-1/(2\beta+1)}$ and $a \in \R$. Note that $p_{1,n \vert B(t,g_{\beta,n})}$ is constant with value
\begin{align*}
	p_{1,n}(x) = \frac{1}{6} + \frac{1-2^{-\beta}}{12}ÊW_{\beta}(g_{\beta,n}) \quad \text{ for all } x \in B(t,g_{\beta,n}).
\end{align*}

		\begin{figure}[H] 
		\centering
  		\includegraphics[width=.95\textwidth]{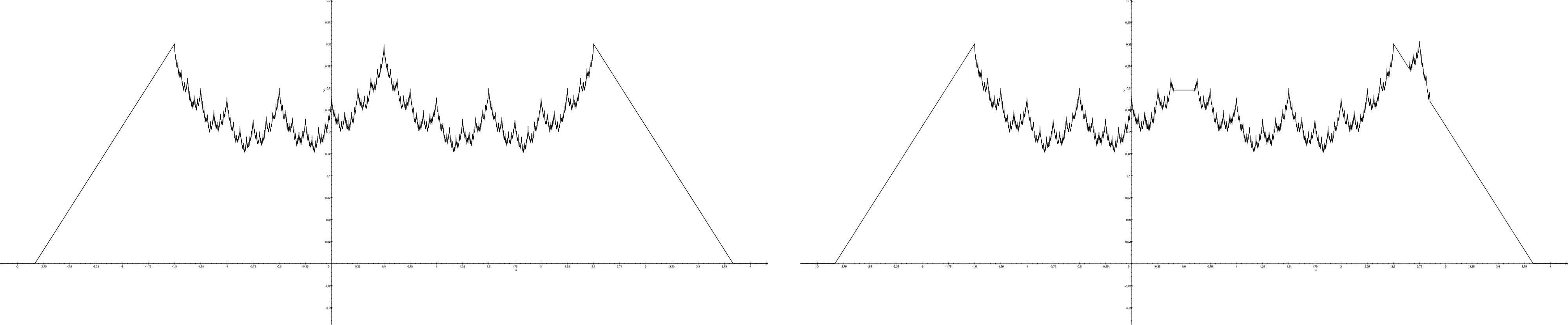}
		\caption{Functions $p_0$ and $p_{1,n}$ for $t=0.5$, $\beta=0.5$ and $n=100$}
		\label{fig:hypos1}
		\end{figure}
		\vspace{-5mm}

\noindent We now show that both $p_0$ and $p_{1,n}$ are contained in the class $\mathscr P_k$ for sufficiently large $k \geq k_0(\beta_*)$ with 
$$p_{0 \vert (-\varepsilon,1+\varepsilon)}, \, p_{1,n \vert (-\varepsilon,1+\varepsilon)}\in \mathcal H_{(-\varepsilon,1+\varepsilon)}(\beta,L^*).$$

\noindent $(i)$ We first verify that $p_0$ integrates to one. Then, it follows directly that also $p_{1,n}$ integrates to one. We have
\begin{align*}
	\int p_0(x) \diff x &= \int_{t-\frac{10}{3}}^{t-2} \left( \frac{1}{4} + \frac{3}{16}(x-t+2) \right) \diff x \\
	&\qquad + \int_{t-2}^{t+2} \left( \frac{1}{6} + \frac{1-2^{-\beta}}{12} W_{\beta}(x-t) \right) \diff x \\
	&\qquad + \int_{t+2}^{t+\frac{10}{3}} \left(\frac{1}{4} - \frac{3}{16}(x-t-2) \right) \diff x \\
	&= \frac{1}{6} + \frac{2}{3} + \frac{1-2^{-\beta}}{12} \int_{-2}^2 W_{\beta}(x) \diff x + \frac{1}{6} \\
	&= 1,
\end{align*}
where the last equality is due to
\begin{align*}
	\int_{-2}^2 W_{\beta}(x) \diff x = \sum_{k=0}^{\infty} 2^{-k\beta} \int_{-2}^2 \cos(2^k \pi x) \diff x = 0.
\end{align*}

\smallskip
\noindent $(ii)$ Next, we check the non-negativity of $p_0$ and $p_{1,n}$ to show that they are probability density functions. We prove non-negativity for $p_0$, whereas non-negativity of $p_{1,n}$ is an easy implication. Since $p_0(-10/3)=0$ and $p_0$ is linear on $(t-10/3,t-2)$ with positive derivative, $p_0$ is non-negative on $(t-10/3,t-2)$. Analogously, $p_0$ is non-negative on $(t+2,t+10/3)$. Note furthermore that
\begin{align} \label{schranke_weier}
	\vert W_{\beta}(x) \vert \leq W_{\beta}(0) = \sum_{k=0}^{\infty} 2^{-k\beta} = \frac{1}{1-2^{-\beta}}
\end{align}
for all $x \in \R$. Thus, for any $x \in \R$ with $\vert x-t \vert \leq 2$, we have
\begin{align*}
	p_0(x) = \frac{1}{6} + \frac{1-2^{-\beta}}{12} W_{\beta}(x-t) \geq \frac{1}{6} - \frac{1}{12} = \frac{1}{12} > 0.
\end{align*}

\smallskip
\noindent $(iii)$ As $p_0$ and also $p_{1,n}$ are bounded from below by $M=1/12$ on $B(t,2)$, we furthermore conclude that they are bounded from below by $M$ on $(-1,2) \subset B(t,2)$, and therefore on any interval $[-\varepsilon,1+\varepsilon]$ with $0<\varepsilon<1$. 

\smallskip
\noindent $(iv)$ We now verify that $p_{0 \vert (-\varepsilon,1+\varepsilon)}, \, p_{1,n \vert (-\varepsilon,1+\varepsilon)}\in \mathcal H_{(-\varepsilon,1+\varepsilon)}(\beta,L(\beta))$ for some positive constant $L(\beta)$. Note again that for any $0<\varepsilon<1$ and any $t \in [0,1]$, the inclusion $(-\varepsilon,1+\varepsilon) \subset B(t,2)$ holds. Thus,
\begin{align*}
	\sup_{\substack{x,y \in (-\varepsilon,1+\varepsilon) \\ x \neq y}} \frac{\vert p_0(x) - p_0(y) \vert}{\vert x-y \vert^{\beta}} = \frac{1-2^{-\beta}}{12} \cdot \sup_{\substack{x,y \in (-\varepsilon,1+\varepsilon) \\ x \neq y}} \frac{\vert W_{\beta}(x-t) - W_{\beta}(y-t) \vert}{\vert (x-t)-(y-t) \vert^{\beta}},
\end{align*}
which is bounded by some constant $c(\beta)$ according to Lemma \ref{weier_lemma}. Together with \eqref{schranke_weier} and with the triangle inequality, we obtain that
\begin{align*}
	p_{0 \vert (-\varepsilon,1+\varepsilon)} \in \mathcal H_{(-\varepsilon,1+\varepsilon)}(\beta,L)
\end{align*}
for some Lipschitz constant $L = L(\beta)$. The H\"older continuity of $p_{1,n}$ is now a simple consequence. The function $p_{1,n}$ is constant on $B(t,g_{\beta,n})$ and coincides with $p_0$ on $(-\varepsilon,1+\varepsilon) \setminus B(t,g_{\beta,n})$. Hence, it remains to investigate combinations of points $x \in (-\varepsilon,1+\varepsilon) \setminus B(t,g_{\beta,n})$ and $y \in B(t,g_{\beta,n})$. Without loss of generality assume that $x \leq t-g_{\beta,n}$. Then,
\begin{align*}
	\frac{\vert p_{1,n}(x)-p_{1,n}(y) \vert}{\vert x-y \vert^{\beta}} = \frac{\vert p_{1,n}(x)-p_{1,n}(t-g_{\beta,n}) \vert}{\vert x-y \vert^{\beta}} \leq \frac{\vert p_{1,n}(x)-p_{1,n}(t-g_{\beta,n}) \vert}{\vert x-(t-g_{\beta,n}) \vert^{\beta}} \leq L,
\end{align*}
which proves that also
\begin{align*}
	p_{1,n \vert (-\varepsilon,1+\varepsilon)} \in \mathcal H_{(-\varepsilon,1+\varepsilon)}(\beta,L).
\end{align*}

\smallskip
\noindent $(v)$ Finally, we address the verification of Assumption \ref{self_sim} for the hypotheses $p_0$ and $p_{1,n}$. Again, for any $t' \in [0,1]$ and any $h \in \mathcal G_{\infty}$ the inclusion $B(t',2h) \subset B(t,2)$ holds, such that in particular
\begin{align*}
	p_{0 \vert B(t',h)} \in \mathcal H_{\beta^*,B(t',h)}(\beta,L_W(\beta))
\end{align*}
for any $t' \in [0,1]$ and for any $h \in \mathcal G_{\infty}$ by Lemma \ref{weier_lemma}. Simultaneously, Lemma \ref{weier_lemma} implies 
\begin{align*}
	\sup_{s \in B(t',h-g)} \left\vert (K_{R,g} \ast p_0)(s) - p_0(s) \right\vert > \frac{1-2^{-\beta_*}}{12} \left( \frac{4}{\pi}-1 \right) g^{\beta} \geq \frac{g^{\beta}}{\log k}
\end{align*}
for all $g \leq h/2$ and for sufficiently large $k \geq k_0(\beta_*)$. That is, for any $t' \in [0,1]$, both \eqref{assumption1} and \eqref{admissible_density} are satisfies for $p_0$ with $u=h$ for any $h \in \mathcal G_{\infty}$. \\
\noindent Concerning $p_{1,n}$ we distinguish between several combinations of pairs $(t',h)$ with $t' \in [0,1]$ and $h \in \mathcal G_{\infty}$. 

\smallskip
\noindent $(v.1)$
If $B(t',h) \cap B(t,g_{\beta,n}) = \emptyset$, the function $p_{1,n}$ coincides with $p_0$ on $B(t',h)$, for which Assumption \ref{self_sim} has been already verified. 

\smallskip
\noindent $(v.2)$ 
If $B(t',h) \subset B(t,g_{\beta,n})$, the function $p_{1,n}$ is constant on $B(t',h)$, such that \eqref{assumption1} and \eqref{admissible_density} trivially hold for $u=h$ and $\beta=\infty$.

\smallskip
\noindent $(v.3)$ If $B(t',h) \cap B(t,g_{\beta,n}) \neq \emptyset$ and $B(t',h) \not\subset B(t,g_{\beta,n})$, we have that $t'+h > t + g_{\beta,n}$ or $t'-h < t-g_{\beta,n}$. As $p_{1,n \vert B(t,2)}$ is symmetric around $t$ we assume $t'+h > t + g_{\beta,n}$ without loss of generality. In this case, 
\begin{align*}
	(t' + 2h - g) - (t+g_{\beta,n}) > 2\left(\frac{h}{2}-g\right),
\end{align*}
such that 
\begin{align*}
	B\left(t'+\frac{3}{2}h, \frac{h}{2} - g \right) \subset B(t',2h-g) \setminus B(t,g_{\beta,n}).
\end{align*}
Consequently, we obtain
\begin{align*}
	\sup_{s \in B(t',2h-g)} \left\vert (K_{R,g} \ast p_{1,n})(s) - p_{1,n}(s) \right\vert \geq \sup_{s \in B\left(t'+\frac{3}{2}h, \frac{h}{2} - g \right)} \left\vert (K_{R,g} \ast p_{1,n})(s) - p_{1,n}(s) \right\vert.
\end{align*}
If $2h \geq 8g$, we conclude that $h/2 \geq 2g$, so that Lemma \ref{weier_lemma} finally proves Assumption \ref{self_sim} for $u=2h$ to the exponent $\beta$ for sufficiently large $k \geq k_0(\beta_*)$.

\smallskip
\noindent Combining $(i) - (v)$, we conclude that $p_0$ and $p_{1,n}$ are contained in the class $\mathscr P_k$ with $p_{0 \vert (-\varepsilon,1+\varepsilon)}, p_{1,n \vert (-\varepsilon,1+\varepsilon)}\in \mathcal H_{(-\varepsilon,1+\varepsilon)}(\beta,L^*)$ for sufficiently large $k \geq k_0(\beta_*)$. The absolute distance of the two hypotheses in $t$ is at least
\begin{align} \label{dist_weier}
	\vert p_0(t) - p_{1,n}(t) \vert &= \frac{1-2^{-\beta}}{12} \left( W_{\beta}(0) - W_{\beta}(g_{\beta,n}) \right) \notag \\
					&= \frac{1-2^{-\beta}}{12} \sum_{k=0}^{\infty} 2^{-k\beta} \left( 1 - \cos(2^k \pi g_{\beta,n}) \right) \notag \\
					&\geq \frac{1-2^{-\beta_*}}{12} 2^{-\tilde k\beta} \left( 1 - \cos(2^{\tilde k} \pi g_{\beta,n}) \right) \notag \\
					&\geq 2c_7 \, g_{\beta,n}^{\beta} 
\end{align}
where $\tilde k \in \N$ is chosen such that $2^{-(\tilde k+1)} < g_{\beta,n} \leq 2^{-\tilde k}$ and
\begin{align*}
	c_7° = c_7°(\beta_*) = \frac{1-2^{-\beta_*}}{24}.
\end{align*}
It remains to bound the distance between the associated product probability measures $\P_0^{\otimes n}$ and $\P_{1,n}^{\otimes n}$. For this purpose, we analyze the Kullback-Leibler divergence between these probability measures, which can be bounded from above by
\begin{align*}
	K(\P_{1,n}^{\otimes n},\P_0^{\otimes n}) &= n \, K(\P_{1,n}, \P_0) \\
		&= n \int p_{1,n}(x) \log \left( \frac{p_{1,n}(x)}{p_0(x)} \right) \mathbbm{1} \left\{ p_0(x) > 0 \right\} \diff x \\
		&= n \int p_{1,n}(x) \log \left(1 + \frac{q_{t+\frac{9}{4},n}(x) - q_{t,n}(x)}{p_0(x)} \right) \mathbbm{1} \left\{ p_0(x) > 0 \right\} \diff x \\
		&\leq n \int q_{t+\frac{9}{4},n}(x) - q_{t,n}(x) + \frac{\left(q_{t+\frac{9}{4},n}(x) - q_{t,n}(x) \right)^2}{p_0(x)} \mathbbm{1} \left\{ p_0(x) > 0 \right\} \diff x \\
		&= n \int \frac{\left(q_{t+\frac{9}{4},n}(x) - q_{t,n}(x) \right)^2}{p_0(x)} \mathbbm{1} \left\{ p_0(x) > 0 \right\} \diff x \\
		&\leq 12 n \int \left(q_{t+\frac{9}{4},n}(x) - q_{t,n}(x) \right)^2 \diff x \\
		&= 24 n \int q_{0,n}(x)^2 \diff x \\
		&= 24 n \left( \frac{1-2^{-\beta}}{12} \right)^2 \int_{-g_{\beta,n}}^{g_{\beta,n}} \left( W_{\beta}(x)-W_{\beta}(g_{\beta,n}) \right)^2 \diff x \\
		&\leq 24 L(\beta)^2 n \left( \frac{1-2^{-\beta}}{12} \right)^2 \int_{-g_{\beta,n}}^{g_{\beta,n}} \left( g_{\beta,n}-x \right)^{2\beta} \diff x \\
		&\leq c_8° n  g_{\beta,n}^{2\beta+1} \\
		&\leq c_8° 
\end{align*}
using the inequality $\log(1+x) \leq x$, $x > -1$, Lemma \ref{weier_lemma}, and
\begin{align*}
	p_0(t+5/2) = \frac{5}{32} > M = \frac{1}{12},
\end{align*}
where
\begin{align*}
	c_8° = c_8°(\beta) = 48 L(\beta)^2 4^{-(2\beta+1)} 2^{2\beta} \left( \frac{1-2^{-\beta}}{12} \right)^2.
\end{align*}
Using now Theorem 2.2 in \cite{tsybakov2009},
\begin{align*}
	&\inf_{T_n} \sup_{\substack{p \in \mathscr P_k : \\ p_{\vert (-\varepsilon,1+\varepsilon)} \in \mathcal H_{(-\varepsilon,1+\varepsilon)}(\beta,L^*)}} \P_p^{\otimes n} \left( n^{\frac{\beta}{2\beta+1}}Ê\left\vert T_n(t) - p(t) \right\vert \geq c_7° \right) \\
	&\hspace{5cm} \geq \max \left\{ \frac{1}{4} \exp(-c_8°), \, \frac{1-\sqrt{c_8°/2}}{2} \right\} > 0.
\end{align*}

\paragraph{Part 2}
For $\beta = 1$, consider the function $p_0: \R \to \R$ with
\begin{align*}
	p_0(x) = \begin{cases}
	0, \quad &\text{if } \ \vert x-t \vert > 4 \\
	\frac{1}{4} - \frac{1}{16} \vert x-t \vert, \quad &\text{if } \  \vert x-t \vert \leq 4
	\end{cases}
\end{align*}
and the function $p_{1,n} : \R \to \R$ with
\begin{align*}
	p_{1,n}(x) = p_0(x) + q_{t+\frac{9}{4},n}(x) - q_{t,n}(x), \quad x \in \R,
\end{align*}
where
\begin{align*}
	q_{a,n}(x) = \begin{cases}
	0, \quad &\text{if } \ \vert x-a \vert > g_{1,n} \\
	\frac{1}{16} (g_{1,n}-\vert x-a \vert), \quad &\text{if } \ \vert x-a \vert \leq g_{1,n}
	\end{cases}
\end{align*}
for $g_{1,n} = n^{-1/3}$ and $a \in \R$. The construction is depicted in Figure \ref{fig:hypos_lip} below.

		\begin{figure}[H] 
		\centering
  		\includegraphics[width=.95\textwidth]{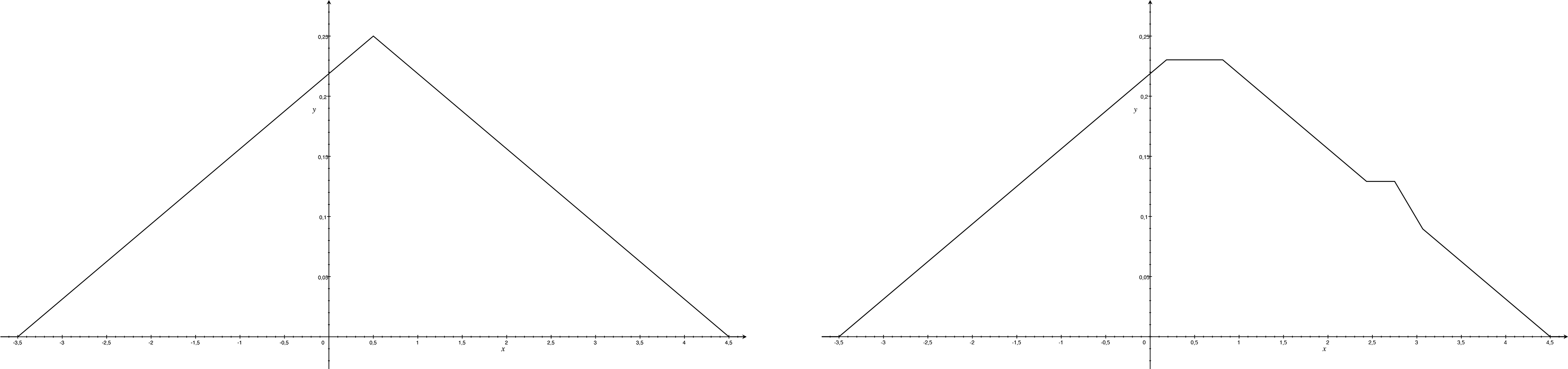}
		\caption{Functions $p_0$ and $p_{1,n}$ for $t=0.5$, $\beta=0.5$ and $n=10$}
		\label{fig:hypos_lip}
		\end{figure}
		\vspace{-5mm}

\smallskip
\noindent $(i)-(iii)$ Easy calculations show that both $p_0$ and $p_{1,n}$ are probability densities, which are bounded from below by $M=1/8$ on $B(t,2)$.

\smallskip
\noindent $(iv)$ We now verify that $p_{0 \vert (-\varepsilon,1+\varepsilon)}, \, p_{1,n \vert (-\varepsilon,1+\varepsilon)}\in \mathcal H_{(-\varepsilon,1+\varepsilon)}(1,L)$ for some Lip\-schitz constant $L>0$. Note again that for any $0<\varepsilon<1$ and any $t \in [0,1]$, the inclusion $(-\varepsilon,1+\varepsilon) \subset B(t,2)$ holds. Thus,
\begin{align*}
	\sup_{\substack{x,y \in (-\varepsilon,1+\varepsilon) \\ x \neq y}} \frac{\vert p_0(x) - p_0(y) \vert}{\vert x-y \vert} = \frac{1}{16} \cdot \sup_{\substack{x,y \in (-\varepsilon,1+\varepsilon) \\ x \neq y}} \frac{\vert \vert x-t \vert - \vert y-t \vert \vert}{\vert x-y \vert} \leq \frac{1}{16}.
\end{align*}
Since $p_0$ has maximal value $1/4$, we obtain that
\begin{align*}
	p_{0 \vert (-\varepsilon,1+\varepsilon)} \in \mathcal H_{(-\varepsilon,1+\varepsilon)}\left(1,\frac{5}{16} \right).
\end{align*}
For the same reasons as before, we also obtain 
\begin{align*}
	p_{1,n \vert (-\varepsilon,1+\varepsilon)} \in \mathcal H_{(-\varepsilon,1+\varepsilon)}\left(1,\frac{5}{16} \right).
\end{align*}

\smallskip
\noindent $(v)$ Finally, we address the verification of Assumption \ref{self_sim} for the hypotheses $p_0$ and $p_{1,n}$. Again, for any $t' \in [0,1]$ and any $h \in \mathcal G_{\infty}$ the inclusion $B(t',2h) \subset B(t,2)$ holds, and we distinguish between several combinations of pairs $(t',h)$ with $t' \in [0,1]$ and $h \in \mathcal G_{\infty}$. We start with $p_0$. 

\smallskip
\noindent $(v.1)$
If $t \notin B(t',h)$, it holds that $\Vert p \Vert_{\beta,B(t',h)} \leq 5/16$ for all $\beta > 0$, such that \eqref{assumption1} and \eqref{admissible_density} trivially hold for $u=h$ and $\beta=\infty$.

\smallskip
\noindent $(v.2)$
In case $t \in B(t',h)$, the function $p_{0 \vert B(t',2h)}$ is not differentiable and 
$$\Vert p_0 \Vert_{1,B(t',2h)} \leq 5/16.$$
Furthermore, $t \in B(t',2h-g)$ for any $g \in\mathcal G_{\infty}$ with $g < 2h/16$ and thus
\begin{align*}
	\sup_{s \in B(t',2h-g)} \left\vert (K_{R,g} \ast p)(s) - p(s) \right\vert &\geq \left\vert (K_{R,g} \ast p)(t) - p(t) \right\vert = \frac{1}{32} g.
\end{align*}
That is, \eqref{assumption1} and \eqref{admissible_density} are satisfied for $u=2h$ and $\beta=1$ for sufficiently large $n \geq n_0$. 

\smallskip
\noindent The density $p_{1,n}$ can be treated in a similar way. It is constant on the interval $B(t,g_{\beta,n})$. If $B(t',h)$ does not intersect with $\{t-g_{\beta,n}, t+g_{\beta,n}\}$, Assumption \ref{self_sim} is satisfied for $u=h$ and $\beta=\infty$. If the two sets intersect, $t-g_{\beta,n}$ or $t+g_{\beta,n}$ is contained in $B(t',2h-g)$ for any $g \in \mathcal G_{\infty}$ with $g < 2h/16$, and we proceed as before.

\smallskip
\noindent Again, combining $(i) - (v)$, it follows that $p_0$ and $p_{1,n}$ are contained in the class $\mathscr P_k$ with $p_{0 \vert (-\varepsilon,1+\varepsilon)}, p_{1,n \vert (-\varepsilon,1+\varepsilon)}\in \mathcal H_{(-\varepsilon,1+\varepsilon)}(1,L)$ for sufficiently large $k \geq k_0$ and some universal constant $L>0$. The absolute distance of the two hypotheses in $t$ equals
\begin{align*}
	\vert p_0(t) - p_{1,n}(t) \vert &= \frac{1}{16} g_{1,n}.
\end{align*}
To bound the Kullback-Leibler divergence between the associated product probability measures $\P_0^{\otimes n}$ and $\P_{1,n}^{\otimes n}$, we derive as before
\begin{align*}
	K(\P_{1,n}^{\otimes n},\P_0^{\otimes n}) &\leq n \int \frac{\left(q_{t+\frac{9}{4},n}(x) - q_{t,n}(x) \right)^2}{p_0(x)} \mathbbm{1} \left\{ p_0(x) > 0 \right\} \diff x \\
		&\leq 16 n \int \left(q_{t+\frac{9}{4},n}(x) - q_{t,n}(x) \right)^2 \diff x \\
		&= 32 n \int q_{0,n}(x)^2 \diff x \\
		&= \frac{1}{12},
\end{align*} 
using $p_0(t+5/2) > 1/16$. Using Theorem 2.2 in \cite{tsybakov2009} again,
\begin{align*}
	&\inf_{T_n} \sup_{\substack{p \in \mathscr P_k : \\ p_{\vert (-\varepsilon,1+\varepsilon)} \in \mathcal H_{(-\varepsilon,1+\varepsilon)}(1,L^*)}} \P_p^{\otimes n} \left( n^{\frac{1}{3}}Ê\left\vert T_n(t) - p(t) \right\vert \geq \frac{1}{32} \right) \\
	&\hspace{5cm} \geq \max \left\{ \frac{1}{4} \exp(-1/12), \, \frac{1-\sqrt{1/24}}{2} \right\} > 0.
\end{align*}
\end{proof}

\begin{proof}[Proof of Proposition \ref{nowhere_dense}]
Define 
\begin{align*}
	\tilde{\mathscr R} = \bigcup_{n \in \N}Ê\tilde{\mathscr R}_n
\end{align*}
with
\begin{align*}
	\tilde{\mathscr R}_n = &\Bigg\{ p \in \mathcal H_{(-\varepsilon,1+\varepsilon)}(\beta_*) : \forall \, t \in [0,1] \ \forall \, h \in \mathcal G_{\infty} \ \exists \,  \beta \in [\beta_*, \beta^*] \text{ with } \\
	&\hspace{0.5cm} p_{\vert B(t,h)} \in \mathcal H_{B(t,h)}(\beta) \text{and } \Vert (K_{R,g} \ast p) - p \Vert_{B(t,h-g)} \geq \frac{g^{\beta}}{\log n} \\
	&\hspace{5cm} \text{ for all }Êg \in \mathcal G_{\infty} \text{ with } g \leq h/8 \Bigg\}.
\end{align*}
Furthermore, let 
\begin{align*}
	E_n(\beta) = &\Bigg\{ p \in \mathcal H_{(-\varepsilon,1+\varepsilon)}(\beta) : \Vert (K_{R,g} \ast p) - p \Vert_{B(t,h-g)} \geq \frac{2}{\log n} g^{\beta} \text{ for all } t \in [0,1], \\
	&\hspace{3.5cm} \text{ for allÊ} \, \,Êh \in \mathcal G_{\infty} \text{, and for all }Êg \in \mathcal G_{\infty} \text{ with } g \leq h/8 \Bigg\}.
\end{align*}
Note that Lemma \ref{weier_lemma} shows that $E_n(\beta)$ is non-empty as soon as
\begin{align*}
	\frac{2}{\log n} \leq 1-\frac{4}{\pi}.
\end{align*}
Note additionally that $E_n(\beta) \subset \tilde{\mathscr R}_n$ for any $\beta \in [\beta_*,\beta^*]$, and
\begin{align*}
	\bigcup_{n \in \N} E_n(\beta) \subset \tilde{\mathscr R}.
\end{align*}
With
\begin{align*}
	A_n(\beta) = \left\{ \tilde f \in \mathcal H_{(-1,2)}(\beta) : \Vert \tilde f-f \Vert_{\beta,(-\varepsilon,1+\varepsilon)} <  \frac{\Vert K_R \Vert_1^{-1}}{\log n} \text{ for some } f \in E_n(\beta) \right\},
\end{align*}
we get for any $\tilde f \in A_n(\beta)$ and a corresponding $f \in E_n(\beta)$ with
\begin{align*}
 	\Vert \check f \Vert_{\beta,(-\varepsilon,1+\varepsilon)} < \Vert K_R \Vert_1^{-1} \frac{1}{\log n}
\end{align*}
and $\check f = \tilde f-f$, the lower bound 
\begin{align*}
	&\Big\Vert (K_{R,g} \ast \tilde f)-\tilde f \Big\Vert_{B(t,h-g)} \\
	& \geq \Big\Vert (K_{R,g} \ast f)-f \Big\Vert_{B(t,h-g)} - \Big\Vert \check f - (K_{R,g} \ast \check f)  \Big\Vert_{B(t,h-g)} \\
	& = \frac{2}{\log n} g^{\beta} - \sup_{s \in B(t,h-g)} \left\vert \int K_R(x) \left\{ \check f(s+gx) - \check f(s) \right\} \diff x \right\vert \\
		& \geq \frac{2}{\log n} g^{\beta} - g^{\beta} \cdot \int \vert K_R(x) \vert \sup_{s \in B(t,h-g)} \sup_{\substack{s' \in B(s,g) \\ s' \neq s}} \frac{\left\vert \check f(s') - \check f(s) \right\vert}{\vert s - s' \vert^{\beta}}  \diff x  \\
		& \geq \frac{2}{\log n} g^{\beta} - g^{\beta} \cdot \Vert K_R \Vert_1 \cdot \Vert \check f \Vert_{\beta,(-\varepsilon,1+\varepsilon)}  \\
		& \geq \frac{1}{\log n} g^{\beta}
\end{align*}
for all $g,h \in \mathcal G_{\infty}$ with $g \leq h/8$ and for all $t \in [0,1]$, and therefore
\begin{align*}
	A = \bigcup_{n \in \N} A_n(\beta) \subset \tilde{\mathscr R}.
\end{align*}
Clearly, $A_n(\beta)$ is open in $\mathcal H_{(-\varepsilon,1+\varepsilon)}(\beta)$. Hence, the same holds true for $A$. Next, we verify that $A$ is dense in $\mathcal H_{(-\varepsilon,1+\varepsilon)}(\beta)$. Let $p \in \mathcal H_{(-\varepsilon,1+\varepsilon)}(\beta)$ and let $\delta > 0$. We now show that there exists some function $\tilde p_{\delta} \in A$ with $\Vert p - \tilde p_{\delta} \Vert_{\beta,(-\varepsilon,1+\varepsilon)} \leq \delta$. For the construction of the function $\tilde p_{\delta}$, set the grid points
\begin{align*}
	t_{j,1}(k) = (4j+1) 2^{-k}, \quad t_{j,2}(k) = (4j+3) 2^{-k}
\end{align*}
for $j \in \{ -2^{k-2}, -2^{k-2}+1, \ldots, 2^{k-1}-1 \}$ and $k \geq 2$. The function $\tilde p_{\delta}$ shall be defined as the limit of a recursively constructed sequence. The idea is to recursively add appropriately rescaled sine waves at those locations where the bias condition is violated. Let $p_{1,\delta} = p$, and denote
\begin{align*}
	J_k = \bigg\{ j \in \{ -2^{k-2}, \ldots, 2^{k-1}-1 \} : &\max_{i=1,2} \Big\vert (K_{R,2^{-k}} \ast p_{k-1,\delta})(t_{j,i}(k)) -p_{k-1,\delta}(t_{j,i}(k)) \Big\vert \\
	 &\hspace{3cm}< \frac{1}{2}c_9° \, \delta \left( 1 - \frac{2}{\pi} \right) 2^{-k\beta} \bigg\}
\end{align*}
for $k \geq 2$, where
\begin{align*}
	c_9° = c_9°(\beta) = \left( \frac{3\pi}{2}  \cdot \frac{1}{1-2^{\beta-1}} + \frac{7}{1-2^{-\beta}} \right)^{-1}.
\end{align*}
For any $k \geq 2$ set
\begin{align*}
	p_{k,\delta}(x) = p_{k-1,\delta}(x) + c_9° \, \delta \sum_{j \in J_k} S_{k,\beta,j}(x)
\end{align*}
with functions
\begin{align*}
	S_{k,\beta,j}(x) = 2^{-k\beta} \sin \left( 2^{k-1} \pi x \right) \mathbbm{1} \left\{ \vert (4j+2) 2^{-k} - x \vert \leq 2^{-k+1} \right\}
\end{align*}
exemplified in Figure \ref{fig:sin_fct}. That is, 
\begin{align*}
	p_{k,\delta}(x) = p(x) + c_9° \, \delta \sum_{l=2}^k \sum_{j \in J_l} S_{l,\beta,j}(x),
\end{align*}
and we define $\tilde p_{\delta}$ as the limit
\begin{align*}
	\tilde p_{\delta}(x) &=  p(x) + c_9° \, \delta \sum_{l=2}^{\infty} \sum_{j \in J_l} S_{l,\beta,j}(x) \\
		&= p_{k,\delta}(x) + c_9° \, \delta \sum_{l=k+1}^{\infty} \sum_{j \in J_l} S_{l,\beta,j}(x).
\end{align*}
The function $\tilde p_{\delta}$ is well-defined as the series on the right-hand side converges: for fixed $l \in \N$, the indicator functions
\begin{align*}
	\mathbbm{1} \left\{ \vert (4j+2) 2^{-k} - x \vert \leq 2^{-k+1} \right\}, \quad j \in \{ -2^{l-2}, -2^{l-2}+1, \ldots, 2^{l-1}-1 \}
\end{align*}
have disjoint supports, such that
\begin{align*}
	 \bigg\Vert \sum_{j \in J_l} S_{l,\beta,j} \bigg\Vert_{(-\varepsilon,1+\varepsilon)} \leq 2^{-l\beta}.
\end{align*}

		\begin{figure}[H] 
		\centering
  		\includegraphics[width=.48\textwidth]{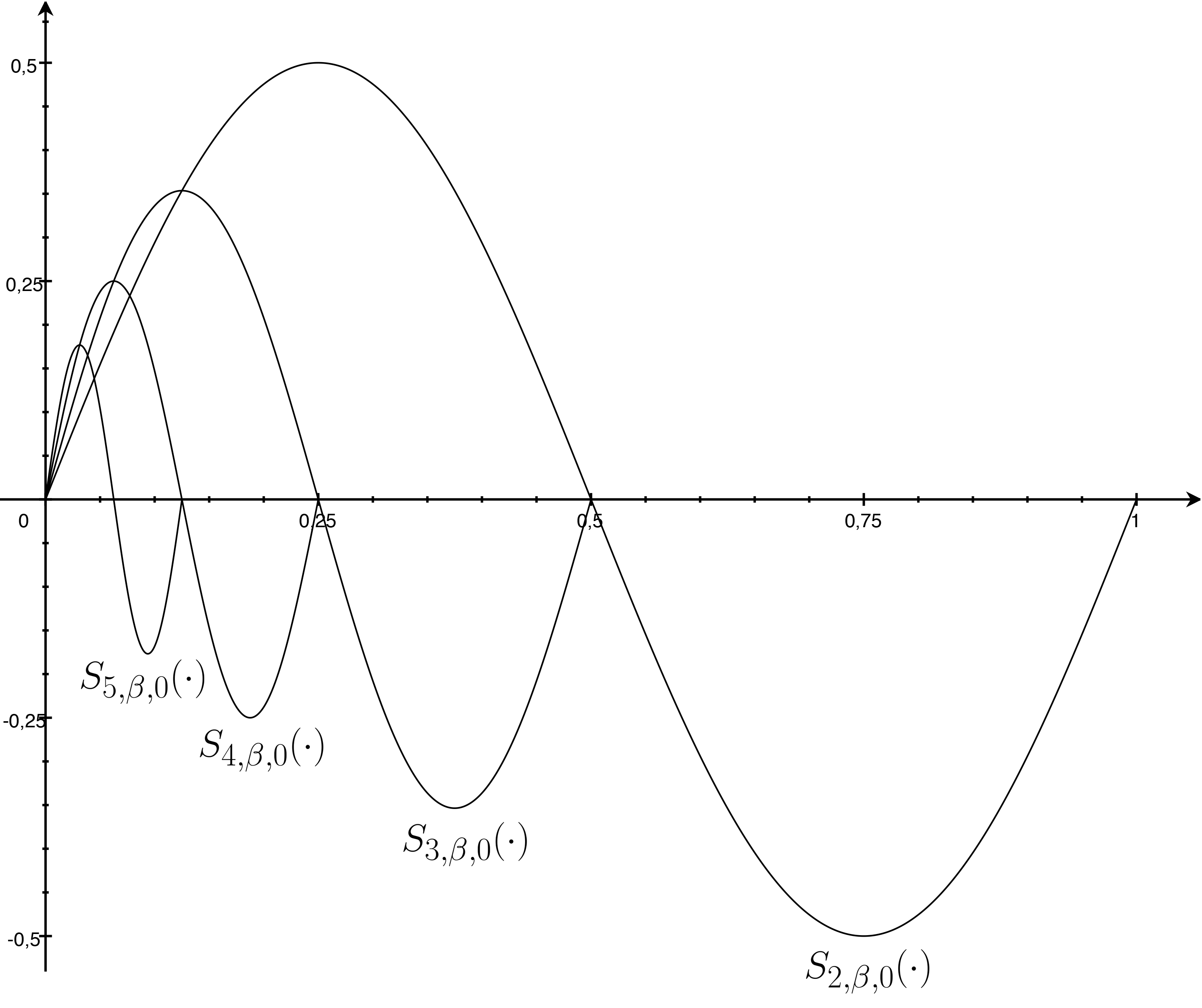}
		\caption{Functions $S_{k,\beta,0}$ for $k=2,\ldots,5$ and $\beta=0.5$}
		\label{fig:sin_fct}
		\end{figure}
		\vspace{-5mm}

\noindent It remains to verify that $\tilde p_{\delta} \in \bigcup_{n \in \N} E_n(\beta) \subset A$ and also $\Vert p - \tilde p_{\delta} \Vert_{\beta,(-\varepsilon,1+\varepsilon)} \leq \delta$. As concerns the inequality $\Vert p - \tilde p_{\delta} \Vert_{\beta,(-\varepsilon,1+\varepsilon)} \leq \delta$, it remains to show that
\begin{align*}
	\left\Vert \sum_{l=2}^{\infty} \sum_{j \in J_l} S_{l,\beta,j} \right\Vert_{\beta,(-\varepsilon,1+\varepsilon)} \leq \frac{1}{c_9°}.
\end{align*}
For $s,t \in (-\varepsilon,1+\varepsilon)$ with $\vert s-t \vert \leq 1$, we obtain
\begin{align}\label{sinus_indikator}
	&\left\vert \sum_{l=2}^{\infty} \sum_{j \in J_l} S_{l,\beta,j}(s) - \sum_{l=2}^{\infty} \sum_{j \in J_l} S_{l,\beta,j}(t) \right\vert \notag \\
	&\hspace{1cm} \leq \sum_{l=2}^{\infty} 2^{-l\beta} \bigg\vert \sin(2^{l-1} \pi s) \sum_{j \in J_l} \mathbbm{1} \{ \vert (4j+2)2^{-l} - s \vert \leq 2^{-l+1} \} \\
	&\hspace{3.5cm} - \sin(2^{l-1} \pi t) \sum_{j \in J_l} \mathbbm{1}\{ \vert (4j+2)2^{-l} - t \vert \leq 2^{-l+1} \} \bigg\vert. \notag
\end{align}
Choose now $k' \in \N$ maximal, such that both
\begin{align*}
	(4j+2) 2^{-k'} - 2^{-k'+1} \leq s \leq (4j+2) 2^{-k'} + 2^{-k'+1}
\end{align*}
and
\begin{align*}
	(4j+2) 2^{-k'} - 2^{-k'+1} \leq t \leq (4j+2) 2^{-k'} + 2^{-k'+1}
\end{align*}
for some $j \in \{ -2^{k'-2}, \ldots, 2^{k'-1}-1 \}$. For $2 \leq l \leq k'$, we have
\begin{align}Ê\label{sin_ohneIndikator}
	&\bigg\vert \sin(2^{l-1} \pi s) \sum_{j \in J_l} \mathbbm{1} \{ \vert (4j+2)2^{-l} - s \vert \leq 2^{-l+1} \} \notag \\
	&\hspace{4.5cm} - \sin(2^{l-1} \pi t) \sum_{j \in J_l} \mathbbm{1} \{ \vert (4j+2)2^{-l} - t \vert \leq 2^{-l+1} \} \bigg\vert \notag \\
	&\qquad \leq \bigg\vert \sin(2^{l-1} \pi s) - \sin(2^{l-1} \pi t) \bigg\vert \notag \\
	&\qquad \leq \min \left\{ 2^{l-1} \pi\vert s-t \vert, 2 \right\}
\end{align}
by the mean value theorem. For $l \geq k'+1$,
\begin{align*}
	&\bigg\vert \sin(2^{l-1} \pi s) \sum_{j \in J_l} \mathbbm{1} \{ \vert (4j+2)2^{-l} - s \vert \leq 2^{-l+1} \} \\
	&\hspace{4.5cm} - \sin(2^{l-1} \pi t) \sum_{j \in J_l} \mathbbm{1} \{ \vert (4j+2)2^{-l} - t \vert \leq 2^{-l+1} \} \bigg\vert \\
	&\leq \max \Bigg\{ \bigg\vert \sin(2^{l-1} \pi s) \bigg\vert, \bigg\vert \sin(2^{l-1} \pi t) \bigg\vert \Bigg\}.
\end{align*}
Furthermore, due to the choice of $k'$, there exists some $z \in [s,t]$ with
\begin{align*}
	\sin(2^{l-1} \pi z) = 0
\end{align*}
for all $l \geq k'+1$. Thus, for any $l \geq k'+1$, by the mean value theorem,
\begin{align*}
	\bigg\vert \sin(2^{l-1} \pi s) \bigg\vert &= \bigg\vert \sin(2^{l-1} \pi s) - \sin(2^{l-1} \pi z) \bigg\vert \\
		&\leq \min \left\{ 2^{l-1} \pi \vert s-zÊ\vert, 1 \right\} \\
		&\leq \min \left\{ 2^{l-1} \pi \vert s-tÊ\vert, 1 \right\}.
\end{align*}
Analogously, we obtain
\begin{align*}
	\bigg\vert \sin(2^{l-1} \pi t) \bigg\vert \leq \min \left\{ 2^{l-1} \pi \vert s-t \vert, 1 \right\}.
\end{align*}
Consequently, together with inequality \eqref{sinus_indikator} and \eqref{sin_ohneIndikator},
\begin{align*}
	&\left\vert \sum_{l=2}^{\infty} \sum_{j \in J_l} S_{l,\beta,j}(s) - \sum_{l=2}^{\infty} \sum_{j \in J_l} S_{l,\beta,j}(t) \right\vert \leq \sum_{l=2}^{\infty} 2^{-l\beta} \min \left\{ 2^{l-1} \pi \vert s-t \vert, \, 2 \right\}.
\end{align*}
Choose now $k \in \N \cup \{ 0 \}$, such that $2^{-(k+1)} < \vert s-t \vert \leq 2^{-k}$. If $k \leq 1$,
\begin{align*}
	\sum_{l=2}^{\infty} 2^{-l\beta} \min \left\{ 2^{l-1} \pi \vert s-t \vert, \, 2 \right\} \leq 2 \frac{2^{-2\beta}}{1-2^{-\beta}} \leq \frac{2}{1-2^{-\beta}} \vert s- t \vert^{\beta}.
\end{align*}
If $k \geq 2$, we decompose
\begin{align*}
	\sum_{l=2}^{\infty} 2^{-l\beta} \min \left\{ 2^{l-1} \pi \vert s-t \vert, \, 2 \right\} &\leq \frac{\pi}{2} \vert s-t \vert \, \sum_{l=0}^k 2^{l(1-\beta)} + 2 \sum_{l=k+1}^{\infty} 2^{-l\beta} \\
		&= \frac{\pi}{2} \vert s-t \vert \, \frac{2^{k(1-\beta)}-2^{\beta-1}}{1-2^{\beta-1}} + 2 \cdot \frac{2^{-(k+1)\beta}}{1-2^{-\beta}} \\
		&\leq \vert s-t \vert^{\beta} \cdot \left( \frac{\pi}{2}  \cdot \frac{1}{1-2^{\beta-1}} + \frac{2}{1-2^{-\beta}} \right).
\end{align*}
Since furthermore
\begin{align*}
	\left\Vert \sum_{l=2}^{\infty} \sum_{j \in J_l} S_{l,\beta,j} \right\Vert_{\sup} \leq \frac{1}{1-2^{-\beta}},
\end{align*}
we have
\begin{align*}
	\left\Vert \sum_{l=2}^{\infty} \sum_{j \in J_l} S_{l,\beta,j} \right\Vert_{\beta,(-\varepsilon,1+\varepsilon)} \leq 3 \left( \frac{\pi}{2}  \cdot \frac{1}{1-2^{\beta-1}} + \frac{2}{1-2^{-\beta}} \right) + \frac{1}{1-2^{-\beta}} = \frac{1}{c_9°}
\end{align*}
and finally $\Vert p-\tilde p_{\varepsilon} \Vert_{\beta,(-\varepsilon,1+\varepsilon)} \leq \delta$. In particular $\tilde p_{\delta} \in \mathcal H_{(-\varepsilon,1+\varepsilon)}(\beta)$.

\smallskip

\noindent We now show that the function $\tilde p_{\delta}$ is contained in $\bigcup_{n \in \N} E_n(\beta) \subset A$. For any bandwidths $g,h \in \mathcal G_{\infty}$ with $g \leq h/8$, it holds that $h-g \geq 4g$. Thus, for any $g=2^{-k}$ with $k \geq 2$ and for any $t \in (-\varepsilon,1+\varepsilon)$, there exists some $j = j(t,h,g) \in \{-2^{k-2}, \ldots, 2^{k-1}-1 \}$ such that both $t_{j,1}(k)$ and $t_{j,2}(k)$ are contained in $B(t,h-g)$, which implies
\begin{align} \label{reduction_2points}
	\sup_{s \in B(t,h-g)} \left\vert (K_{R,g} \ast \tilde p_{\delta})(s) - \tilde p_{\delta}(s) \right\vert \geq \max_{i=1,2} \left\vert (K_{R,g} \ast \tilde p_{\delta})(t_{j,i}(k)) - \tilde p_{\delta}(t_{j,i}(k)) \right\vert.
\end{align}
By linearity of the convolution and the theorem of dominated convergence,
\begin{align} \label{faltung_split}
	&(K_{R,g} \ast \tilde p_{\delta})(t_{j,i}(k)) - \tilde p_{\delta}(t_{j,i}(k)) \notag \\
	&\qquad= (K_{R,g} \ast p_{k,\delta})(t_{j,i}(k)) - p_{k,\delta}(t_{j,i}(k)) \notag \\
		&\qquad \qquad+ c_9° \, \delta \sum_{l=k+1}^{\infty} \sum_{j \in J_l} \Big( (K_{R,g} \ast S_{l,\beta,j})(t_{j,i}(k)) - S_{l,\beta,j}(t_{j,i}(k)) \Big).
\end{align}
We analyze the convolution $K_{R,g} \ast S_{l,\beta,j}$ for $l \geq k+1$. Here,
\begin{align*}
	 \sin \left( 2^{l-1} \pi \, t_{j,1}(k) \right) = \sin \left( 2^{l-k-1} \pi \, (4j+1)  \right) = 0
\end{align*}
and 
\begin{align*}
	 \sin \left( 2^{l-1} \pi \, t_{j,2}(k) \right) = \sin \left( 2^{l-k-1} \pi \, (4j+3)  \right) = 0.
\end{align*}
Hence,
\begin{align*}
	\sum_{j \in J_l} S_{l,\beta,j}(t_{j,i}(k)) = 0, \quad i=1,2
\end{align*}
for any $l \geq k+1$. Furthermore,
\begin{align*}
	(K_{R,g} \ast S_{l,\beta,j})(t_{j,i}(k)) &= \frac{1}{2g} \int_{-g}^g S_{l,\beta,j}(t_{j,i}(k)-x) \diff x \\
		&= \frac{1}{2g} \int_{t_{j,i}(k)-g}^{t_{j,i}(k)+g} S_{l,\beta,j}(x) \diff x, \quad i=1,2.
\end{align*}
Due to the identities
\begin{align*}
	(4j+2) 2^{-k} - 2^{-k+1} &= t_{j,1}(k) - g \\
	(4j+2) 2^{-k} + 2^{-k+1} &= t_{j,2}(k) + g,
\end{align*}
we have either
\begin{align*}
	\left[ (4j+2) 2^{-l} - 2^{-l+1}, (4j+2) 2^{-l} + 2^{-l+1} \right] \subset \left[  t_{j,1}(k) - g,  t_{j,2}(k) + g \right]
\end{align*}
or
\begin{align*}
	\left[ (4j+2) 2^{-l} - 2^{-l+1}, (4j+2) 2^{-l} + 2^{-l+1} \right] \cap \left[  t_{j,1}(k) - g,  t_{j,2}(k) + g \right] = \emptyset
\end{align*}
for any $l \geq k+1$. Therefore, for $i=1,2$,
\begin{align*}
	&\sum_{j \in J_l}(K_{R,g} \ast S_{l,\beta,j})(t_{j,i}(k)) \\
	&\quad = \sum_{j \in J_l}\frac{1}{2g} \int_{t_{j,i}(k)-g}^{t_{j,i}(k)+g} 2^{-l\beta} \sin \left( 2^{l-1} \pi x \right) \mathbbm{1} \left\{ \vert (4j+2) 2^{-l} - x \vert \leq 2^{-l+1} \right\} \diff x \\
	&\quad = 0
\end{align*}
such that equation \eqref{faltung_split} then simplifies to
\begin{align*}
	(K_{R,g} \ast \tilde p_{\delta})(t_{j,i}(k)) - \tilde p_{\delta}(t_{j,i}(k)) &= (K_{R,g} \ast p_{k,\delta})(t_{j,i}(k)) - p_{k,\delta}(t_{j,i}(k)), \quad i=1,2.
\end{align*}
Together with \eqref{reduction_2points}, we obtain
\begin{align*}
	\sup_{s \in B(t,h-g)} \left\vert (K_{R,g} \ast \tilde p_{\delta})(s) - \tilde p_{\delta}(s) \right\vert \geq \max_{i=1,2} \left\vert (K_{R,g} \ast p_{k,\delta})(t_{j,i}(k)) - p_{k,\delta}(t_{j,i}(k)) \right\vert
\end{align*}
for some $j \in \{-2^{k-2}, -2^{k-2}+1, \ldots, 2^{k-2}-1 \}$. If $j \notin J_k$, then
\begin{align*}
	&\max_{i=1,2} \left\vert (K_{R,g} \ast p_{k,\delta})(t_{j,i}(k)) - p_{k,\delta}(t_{j,i}(k)) \right\vert \\
	&\hspace{2cm}= \max_{i=1,2} \left\vert (K_{R,g} \ast p_{k-1,\delta})(t_{j,i}(k)) - p_{k-1,\delta}(t_{j,i}(k)) \right\vert \\
	&\hspace{2cm}\geq \frac{1}{2} c_9° \, \delta \left( 1 - \frac{2}{\pi} \right) g^{\beta}.
\end{align*}
If $j \in J_k$, then
\begin{align*}
	&\max_{i=1,2} \left\vert (K_{R,g} \ast p_{k,\delta})(t_{j,i}(k)) - p_{k,\delta}(t_{j,i}(k)) \right\vert \\
	&\quad \geq c_9° \, \delta \max_{i=1,2} \left\vert (K_{R,g} \ast S_{k,\beta,j})(t_{j,i}(k)) - S_{k,\beta,j}(t_{j,i}(k)) \right\vert \\
	&\qquad - \max_{i=1,2} \left\vert (K_{R,g} \ast p_{k-1,\delta})(t_{j,i}(k)) - p_{k-1,\delta}(t_{j,i}(k)) \right\vert \\
	&\quad \geq c_9° \, \delta \max_{i=1,2} \left\vert (K_{R,g} \ast S_{k,\beta,j})(t_{j,i}(k)) - S_{k,\beta,j}(t_{j,i}(k)) \right\vert - \frac{1}{2} c_9° \, \delta \left( 1 - \frac{2}{\pi} \right) g^{\beta}.
\end{align*}
Similar as above we obtain
\begin{align*}
	&(K_{R,g} \ast S_{k,\beta,j})(t_{j,1}(k)) - S_{k,\beta,j}(t_{j,1}(k)) \\
	&\hspace{2cm} = \frac{1}{2g} \int_{t_{j,1}(k)-g}^{t_{j,1}(k)+g} 2^{-k\beta} \sin \left( 2^{k-1} \pi x \right) \diff x - 2^{-k\beta} \\
	&\hspace{2cm} = \frac{1}{2g} 2^{-k\beta} \int_0^{2^{-k+1}} \sin \left( 2^{k-1} \pi x \right) \diff x - 2^{-k\beta} \\
	&\hspace{2cm} =  g^{\beta} \left( \frac{2}{\pi} - 1 \right)
\end{align*}
as well as
\begin{align*}
	&(K_{R,g} \ast S_{k,\beta,j})(t_{j,2}(k)) - S_{k,\beta,j}(t_{j,2}(k)) =  g^{\beta} \left( 1 - \frac{2}{\pi} \right),
\end{align*}
such that
\begin{align*}
	\max_{i=1,2} \left\vert (K_{R,g} \ast p_{k,\delta})(t_{j,i}(k)) - p_{k,\delta}(t_{j,i}(k)) \right\vert 
	&\geq \frac{1}{2} c_9° \, \delta \left( 1 - \frac{2}{\pi} \right) g^{\beta}.
\end{align*}
Combining the two cases finally gives
\begin{align*}
	\sup_{s \in B(t,h-g)} \left\vert (K_{R,g} \ast \tilde p_{\delta})(s) - \tilde p_{\delta}(s) \right\vert \geq \frac{1}{2} c_9° \, \delta \left( 1- \frac{2}{\pi} \right) g^{\beta}.
\end{align*}
In particular, $\tilde p_{\delta} \in E_n(\beta)$ for sufficiently large $n \geq n_0(\beta,\delta)$, and thus $\tilde p_{\delta} \in A$.

\smallskip

\noindent Since $A$ is open and dense in the class $\mathcal H_{(-\varepsilon,1+\varepsilon)}(\beta)$ and $A \subset \tilde{\mathscr R}$, the complement $\mathcal H_{(-\varepsilon,1+\varepsilon)}(\beta) \setminus \tilde{\mathscr R}$ is nowhere dense in $H_{(-\varepsilon,1+\varepsilon)}(\beta)$. Thus, because of
\begin{align*}
	 \mathcal H_{(-\varepsilon,1+\varepsilon)}(\beta)_{\vert B(t,h)} = \mathcal H_{B(t,h)}(\beta),
\end{align*}
and the fact that for any $x \in \mathcal H_{(-\varepsilon,1+\varepsilon)}(\beta)$ and any $z' \in \mathcal H_{B(t,h)}(\beta)$ with
\begin{align*}
	\Vert x_{\vert B(t,h)} -  z' \Vert_{\beta,B(t,h)} < \delta
\end{align*}
there exists an extension $z \in \mathcal H_{(-\varepsilon,1+\varepsilon)}(\beta)$ of $z'$ with
\begin{align*}
	\Vert x-z \Vert_{\beta,(-\varepsilon,1+\varepsilon)} < \delta,
\end{align*}
the set $\mathcal H_{B(t,h)}(\beta) \setminus \tilde{\mathscr R}_{\vert B(t,h)}$ is nowhere dense in $\mathcal H_{B(t,h)}(\beta)$. Since the property "nowhere dense" is stable when passing over to intersections and the corresponding relative topology, we conclude that
$$ \mathcal P_{B(t,h)}(\beta, L^*) \setminus \mathscr R_{\vert B(t,h)} $$
	is nowhere dense in 
	$\mathcal P_{B(t,h)}(\beta, L^*)$
	 with respect to $\Vert \cdot \Vert_{\beta,B(t,h)}$.
\end{proof}

\begin{proof}[Proof of Lemma \ref{weier_lemma}]
As it has been proven in \cite{hardy1916} the Weierstra{\ss}  function $W_{\beta}$ is $\beta$-H\"older continuous everywhere. For the sake of completeness, we state the proof here. Because the Weierstra{\ss} function is $2$-periodic, it suffices to consider  points $s,t \in \R$ with $\vert s-t \vert \leq 1$. Note first that
\begin{align*}
	\vert W_{\beta}(s) - W_{\beta}(t) \vert &\leq 2 \sum_{n=0}^{\infty} 2^{-n\beta} \left\vert \sin \left( \frac{1}{2} 2^n \pi (s+t) \right) \right\vert \cdot \left\vert \sin \left( \frac{1}{2} 2^n \pi (s-t) \right) \right\vert \\
		&\leq 2 \sum_{n=0}^{\infty} 2^{-n\beta} \left\vert \sin \left( \frac{1}{2} 2^n \pi (s-t) \right) \right\vert.
\end{align*}
Choose $k \in \N \cup \{0\}$ such that $2^{-(k+1)} < \vert s-t \vert \leq 2^{-k}$. For all summands with index $n \leq k$, use the inequality $\vert \sin(x) \vert \leq \vert x \vert$ and for all summands with index $n>k$ use $\vert \sin(x) \vert \leq 1$, such that
\begin{align*}
	\vert W_{\beta}(s) - W_{\beta}(t) \vert &\leq 2 \sum_{n=0}^k 2^{-n\beta} \left\vert \frac{1}{2} 2^n \pi (s-t) \right\vert + 2 \sum_{n=k+1}^{\infty} 2^{-n\beta} \\
		&= \pi \left\vert s-t \right\vert \sum_{n=0}^k 2^{n(1-\beta)} + 2 \sum_{n=k+1}^{\infty} 2^{-n\beta}.
\end{align*}
Note that, 
\begin{align*}
	\sum_{n=0}^k 2^{n(1-\beta)} = \frac{2^{(k+1)(1-\beta)}-1}{2^{1-\beta}-1} = \frac{2^{k(1-\beta)}-2^{\beta-1}}{1-2^{\beta-1}} \leq \frac{2^{k(1-\beta)}}{1-2^{\beta-1}},
\end{align*}
and, as $2^{-\beta}<1$,
\begin{align*}
	\sum_{n=k+1}^{\infty} 2^{-n\beta} 
	&= \frac{2^{-(k+1)\beta}}{1-2^{-\beta}}.
\end{align*}
Consequently, we have
\begin{align*}
\begin{split}
	\vert W_{\beta}(s) - W_{\beta}(t) \vert &\leq \pi \left\vert s-t \right\vert \frac{2^{k(1-\beta)}}{1-2^{\beta-1}} + 2\frac{2^{-(k+1)\beta}}{1-2^{-\beta}} \\
		&\leq \left\vert s-t \right\vert^{\beta} \left( \frac{\pi}{1-2^{\beta-1}} + \frac{2}{1-2^{-\beta}} \right).
\end{split}
\end{align*}
Furthermore
\begin{align*}
	\Vert W_{\beta} \Vert_{\sup} \leq \sum_{n=0}^{\infty} 2^{-n\beta} = \frac{1}{1-2^{-\beta}},
\end{align*}
so that for any interval $U \subset \R$,
\begin{align*}
	\Vert W_{\beta} \Vert_{\beta,U} \leq \frac{\pi}{1-2^{\beta-1}} + \frac{3}{1-2^{-\beta}}.
\end{align*}

\noindent We now turn to the proof of bias lower bound condition. For any $0<\beta \leq 1$, for any $h \in \mathcal G_{\infty}$, for any $g=2^{-k} \in \mathcal G_{\infty}$ with $g \leq h/2$, and for any $t \in \R$, there exists some $s_0 \in [t-(h-g),t+(h-g)]$ with $\cos \left( 2^k \pi s_0 \right) = 1$, since the function $x \mapsto \cos(2^k \pi x)$ is $2^{1-k}$-periodic. Note that in this case also
\begin{align} \label{s0=0}
	\cos \left( 2^n \pi s_0 \right) = 1\quad \text{ for all } n \geq k.
\end{align}
The following supremum is now lower bounded by
\begin{align*}
	&\sup_{s \in B(t,h-g)} \left\vert \int K_{R,g}(x-s) W_{\beta}(x) \diff x - W_{\beta}(s) \right\vert \\*
	&\hspace{5cm} \geq \left\vert \int_{-1}^1 K_R(x) W_{\beta}(s_0+gx) \diff x - W_{\beta}(s_0) \right\vert.
\end{align*}
As furthermore
\begin{align*}
	\sup_{x \in \R} \left\vert K_R(x) 2^{-n\beta} \cos \left( 2^n \pi (s_0+gx) \right) \right\vert \leq \Vert K_R \Vert_{\sup} \cdot 2^{-n\beta}
\end{align*}
and
\begin{align*}
	\sum_{n=0}^{\infty} \Vert K_R \Vert_{\sup} \cdot 2^{-n\beta} = \frac{\Vert K_R \Vert_{\sup}}{1-2^{-\beta}} < \infty,
\end{align*}
the dominated convergence theorem implies
\begin{align*}
	 &\left\vert \int_{-1}^1 K_R(x) W_{\beta}(s_0+gx) \diff x - W_{\beta}(s_0) \right\vert = \left\vert \sum_{n=0}^{\infty} 2^{-n\beta} I_n(s_0,g) \right\vert
\end{align*}
with
\begin{align*}
	 I_n(s_0,g) = \int_{-1}^1 K_R(x) \cos \left( 2^n \pi (s_0+gx) \right) \diff x - \cos \left( 2^n \pi s_0 \right).
\end{align*}
Recalling \eqref{s0=0}, it holds for any index $n \geq k$
\begin{align} \label{weier_large}
	I_n(s_0,g) &= \frac{1}{2} \cdot \frac{\sin(2^n \pi (s_0+g)) - \sin(2^n \pi (s_0-g))}{2^n \pi g} - 1\notag \\
	& = \frac{\sin(2^n \pi g) }{2^n \pi g} - 1\notag \\
	& = -1.
\end{align}
Furthermore, for any index $0 \leq n \leq k-1$ holds
\begin{align} \label{n<k}
	I_n(s_0,g)&= \frac{1}{2} \cdot \frac{\sin(2^n \pi (s_0+g)) - \sin(2^n \pi (s_0-g))}{2^n \pi g} - \cos \left( 2^n \pi s_0 \right) \notag \\
	& = \cos(2^n \pi s_0) \left( \frac{\sin(2^n \pi g)}{2^n \pi g} - 1 \right).
\end{align}
Using this representation, the inequality $\sin(x)\leq x$ for $x\geq 0$, and Lemma \ref{A3}, we obtain
\begin{align*}
	2^{-n\beta} I_n(s_0,g) &\leq 2^{-n\beta} \left( 1 - \frac{\sin(2^n \pi g)}{2^n \pi g} \right) \\
	&\leq 2^{-n\beta} \cdot \frac{(2^n \pi g)^2}{6} \\
	&\leq 2^{-n\beta + 2(n-k)+1}.
\end{align*}
Since $k-n-1 \geq 0$ and $\beta \leq 1$, this is in turn bounded by
\begin{align} \label{weier_small}
	2^{-n\beta} I_n(s_0,g) &\leq 2^{-(2k-n-2)\beta} \cdot 2^{2(n-k)+1+ 2(k-n-1)\beta}\notag \\*
		&\leq 2^{-(2k-n-2)\beta} \cdot 2^{2(n-k)+1+ 2(k-n-1)} \notag\\*
		&\leq 2^{-(2k-n-2)\beta}.
\end{align}
Taking together \eqref{weier_large} and \eqref{weier_small}, we arrive at
\begin{align*}
	\sum_{n=0}^{k-3} 2^{-n\beta} I_n(s_0,g) + \sum_{n=k+1}^{2k-2} 2^{-n\beta} I_n(s_0,g) &\leq \sum_{n=0}^{k-3} 2^{-(2k-n-2)\beta} - \sum_{n=k+1}^{2k-2} 2^{-n\beta} = 0.
\end{align*}
Since by \eqref{weier_large} also
\begin{align*}
	\sum_{n=2k-1}^{\infty} 2^{-n\beta} I_n(s_0,g) = - \sum_{n=2k-1}^{\infty} 2^{-n\beta} < 0,
\end{align*}
it remains to investigate
\begin{align*}
	\sum_{n=k-2}^k 2^{-n\beta} I_n(s_0,g).
\end{align*}
For this purpose, we distinguish between the three cases
\begin{align*}
	(i) \quad &\cos(2^{k-1} \pi s_0) = \cos(2^{k-2} \pi s_0) = 1 \\
	(ii) \quad &\cos(2^{k-1} \pi s_0) = -1, \ \cos(2^{k-2} \pi s_0) = 0 \\ 
	(iii) \quad &\cos(2^{k-1} \pi s_0) = 1, \ \cos(2^{k-2} \pi s_0) = -1
\end{align*}
and subsequently use the representation in \eqref{n<k}. In case $(i)$, obviously
\begin{align*}
	\sum_{n=k-2}^{k} 2^{-n\beta} I_n(s_0,g) \leq -2^{-k\beta} < 0.
\end{align*}
using $\sin(x) \leq x$ for $x \geq 0$ again. In case $(ii)$, we obtain for $\beta \leq 1$
\begin{align*}
	\sum_{n=k-2}^{k} 2^{-n\beta} I_n(s_0,g) &= 2^{-k\beta} 2^{\beta} \left( 1 - \frac{\sin(\pi/2)}{\pi/2} \right) - 2^{-k\beta} \leq  2^{-k\beta} \left( 1 -\frac{4}{\pi} \right) <0.
\end{align*}
Finally, in case $(iii)$, for $\beta \leq 1$,
\begin{align*}
	&\sum_{n=k-2}^{k} 2^{-n\beta} I_n(s_0,g) \\
	&\quad = 2^{-(k-1)\beta} \left( \frac{\sin(\pi/2)}{\pi/2} - 1 \right) - 2^{-(k-2)\beta} \left( \frac{\sin(\pi/4)}{\pi/4} - 1 \right) - 2^{-k\beta} \\
	&\quad = 2^{-(k-1)\beta} \left( \left( \frac{2}{\pi} - 1 \right) + 2^{\beta} \left( 1 - \frac{\sin(\pi/4)}{\pi/4} \right) \right) - 2^{-k\beta} \\
	&\quad < 2^{-(k-1)\beta} \left( \frac{2}{\pi} + 1 - 8 \, \frac{\sin(\pi/4)}{\pi} \right) - 2^{-k\beta} \\*
	&\quad < -2^{-k\beta} \\*
	&\quad < 0.
\end{align*}
That is,
\begin{align*}
	&\sup_{s \in B(t,h-g)} \left\vert \int K_{R,g}(x-s) W_{\beta}(x) \diff x - W_{\beta}(s) \right\vert \notag \\
		&\hspace{2cm} \geq \left\vert \int_{-1}^1 K_R(x) W_{\beta}(s_0+gx) \diff x - W_{\beta}(s_0) \right\vert \notag \\
		&\hspace{2cm} = - \sum_{n=0}^{\infty} 2^{-n\beta} I_n(s_0,g) \notag \\
		&\hspace{2cm} \geq - \sum_{n=k-2}^{k} 2^{-n\beta} I_n(s_0,g) \notag \\
		&\hspace{2cm} >  \left( \frac{4}{\pi}-1 \right) g^{\beta}.
\end{align*}
\end{proof}

\begin{proof}[Proof of Theorem \ref{limitdistribution}]
The proof is structured as follows. First, we show that the bias term is negligible. Then, we conduct several reduction steps to non-stationary Gaussian processes. We pass over to the supremum over a stationary Gaussian process by means of Slepian's comparison inequality, and finally, we employ extreme value theory for its asymptotic distribution.

\paragraph{Step 1 (Negligibility of the bias)}

Due to the discretization of the interval $[0,1]$ in the construction of the confidence band and due to the local variability of the confidence band's width, the negligibility of the bias is not immediate. For any $t \in [0,1]$, there exists some $k_t \in T_n$ with $t \in I_{k_t}$. Hence,
\begin{align*}
	&\sqrt{\tn \hat h_n^{loc}(t)} \left\vert \E_p^{\chi_1} \hat p_n^{loc}(t, \hat h_n^{loc}(t)) - p(t) \right\vert \\
	&\quad = \sqrt{\tn \hat h_{n,k_t}^{loc}} \Big\vert \E_p^{\chi_1} \hat p_n^{(1)}(k_t\delta_n, \hat h_{n,k_t}^{loc}) - p(t) \Big\vert \\
	&\quad \leq \sqrt{\tn \hat h_{n,k_t}^{loc}} \Big\vert \E_p^{\chi_1} \hat p_n^{(1)}(k_t\delta_n, \hat h_{n,k_t}^{loc}) - p(k_t\delta_n) \Big\vert + \sqrt{\tn \hat h_{n,k_t}^{loc}} \Big\vert p(k_t\delta_n) - p(t) \Big\vert.
\end{align*}
Assume $\hat j_n(k\delta_n) \geq k_n(k\delta_n) = \bar j_n(k\delta_n) - m_n$ for all $k \in T_n$, where $m_n$ is given in Proposition \ref{bw_band}. Since $\delta_n \leq \frac{1}{8} h_{\beta_*,n}$ for sufficiently large $n \geq n_0(\beta_*,\varepsilon)$,
\begin{align*}
	\hat h_{n,k_t}^{loc} &=  2^{m_n-u_n} \cdot \min \left\{ 2^{-\hat j_n((k_t-1) \delta_n) - m_n}, 2^{-\hat j_n(k_t \delta_n) - m_n} \right\} \\
		&\leq 2^{m_n-u_n} \cdot \min \left\{ \bar h_n((k_t-1)\delta_n), \bar h_n(k_t\delta_n) \right\} \\
		&\leq 3 \cdot 2^{m_n-u_n} \cdot \bar h_n(t)
\end{align*} 
by Lemma \ref{grid_opt}. In particular, $\delta_n + \hat h_{n,k_t}^{loc} \leq 2^{-(\bar j_n(t)+1)}$ holds for sufficiently large $n \geq n_0(c_1,\beta_*)$, so that Assumption \ref{self_sim}, Lemma \ref{remark2}, and Lemma \ref{bias_up} yield
\begin{align} \label{bias1}
	&\sup_{p \in \mathscr P_n} \sqrt{\tn \hat h_{n,k_t}^{loc}} \left\vert \E_p^{\chi_1} \hat p_n^{(1)}(k_t\delta_n, \hat h_{n,k_t}^{loc}) - p(k_t\delta_n) \right\vert \notag \\*
	&\hspace{2cm} \leq \sup_{p \in \mathscr P_n}  \sqrt{\tn \hat h_{n,k_t}^{loc}} \sup_{s \in B(t,\delta_n)} \left\vert \E_p^{\chi_1}  \hat p_n^{(1)}(s, \hat h_{n,k_t}^{loc})  - p(s) \right\vert \notag \\
	&\hspace{2cm} \leq \sup_{p \in \mathscr P_n}  b_2 \sqrt{\tn \hat h_{n,k_t}^{loc}} \left( \hat h_{n,k_t}^{loc} \right)^{\beta_p(B(t, 2^{-\bar j_n(t)}))} \notag \\
	&\hspace{2cm} \leq \sup_{p \in \mathscr P_n}  b_2 \sqrt{\tn \hat h_{n,k_t}^{loc}} \left( \hat h_{n,k_t}^{loc} \right)^{\beta_p(B(t, \bar h_n(t)))} \notag \\
	&\hspace{2cm} \leq \sup_{p \in \mathscr P_n} b_2 \left( 3 \cdot 2^{m_n-u_n} \right)^{\frac{2\beta_*+1}{2}} \sqrt{\tn \bar h_n(t)} \bar h_n(t)^{\beta_{n,p}(t)} \notag \\
	&\hspace{2cm} \leq c_{10}° \cdot (\log \tn)^{-\frac{1}{4} c_1 (2\beta_*+1)\log 2}
\end{align}
for some constant $c_{10} = b_2 \cdot 3^{(2\beta_*+1)/2}$, on the event
$$ \left\{Ê\hat j_n(k\delta_n) \geq k_n(k\delta_n)  \text{ for all }Êk \in T_nÊ\right\}. $$
Now, we analyze the expression
\begin{align*}
	\sqrt{\tn \hat h_{n,k_t}^{loc}} \Big\vert p(k_t\delta_n) - p(t) \Big\vert.
\end{align*}
For $t \in I_k$ and for $n \geq n_0$,
\begin{align*}
	\delta_n^{\beta_*} \leq 2^{-j_{\min}} \left( \frac{\log \tn}{\tn} \right)^{\kappa_1 \beta_*} 
		\leq 2^{-j_{\min}} \left( \frac{\log \tn}{\tn} \right)^{\frac{1}{2}} 
		\leq 2^{-j_{\min}} \left( \frac{\log \tn}{\tn} \right)^{\frac{\beta_{n,p}(t)}{2\beta_{n,p}(t)+1}},
\end{align*}
such that on the same event
\begin{align} \label{bias2}
	\sup_{p \in \mathscr P_n} \sqrt{\tn \hat h_{n,k_t}^{loc}} \left\vert p(k_t\delta_n) - p(t) \right\vert &\leq \sup_{p \in \mathscr P_n} \sqrt{3} L^* \cdot 2^{\frac{1}{2} (m_n-u_n)} \sqrt{\tn \bar h_n(t)} \cdot \delta_n^{\beta_*} \notag \\
		&\leq c_{11}° \cdot (\log \tn)^{-\frac{1}{4} c_1 \log 2}
\end{align}
for some constant $c_{11}°=c_{11}°(\beta_*,L^*)$. Taking \eqref{bias1} and \eqref{bias2} together,
\begin{align*}
	&\sup_{p \in \mathscr P_n} \sup_{t \in [0,1]} a_n \sqrt{\tn \hat h_n^{loc}(t)} \left\vert \E_p^{\chi_1} \hat p_n^{loc}(t, \hat h_n^{loc}(t)) - p(t) \right\vert \mathbbm{1}Ê\left\{Ê\hat j_n(k\delta_n) \geq k_n(k\delta_n)  \forall \, k \in T_nÊ\right\} \\
	 &\hspace{11.5cm} \leq \varepsilon_{1,n},
\end{align*}
with
\begin{align*}
	\varepsilon_{1,n} = c_{10}° \cdot a_n (\log n)^{-\frac{1}{4} c_1 (2\beta_*+1)\log 2} + c_{11}° \cdot a_n (\log n)^{-\frac{1}{4} c_1 \log 2}.
\end{align*}
According to the definition of $c_1$ in \eqref{c22}, $\varepsilon_{1,n}$ converges to zero. Observe furthermore that 
\begin{align} \label{sup_process}
	\sup_{t \in [0,1]} \sqrt{\tn \hat h_n^{loc}(t)} \left\vert \hat p_n^{loc}(t, \hat h_n^{loc}(t)) - p(t) \right\vert
\end{align}
can be written as
\begin{align*}
	&\max_{k \in T_n} \sup_{t \in I_k} \sqrt{\tn \hat h_n^{loc}(t)} \left\vert \hat p_n^{loc}(t, \hat h_n^{loc}(t)) - p(t) \right\vert \\
	&\qquad = \max_{k \in T_n} \sqrt{\tn \hat h_{n,k}^{loc}} \max \left\{ \hat p_n^{(1)}(k\delta_n, \hat h_{n,k}^{loc}) - \inf_{t \in I_k} p(t), \ \ \sup_{t \in I_k} p(t) - \hat p_n^{(1)}(k\delta_n, \hat h_{n,k}^{loc}) \right\}
\end{align*}
with the definitions in \eqref{loc_const}. That is, the supremum in \eqref{sup_process} is measurable. Then, by means of Proposition \ref{bw_band}, with $x_{1,n} = x - \varepsilon_{1,n}$,
\begin{align} \label{first_max}
	&\inf_{p \in \mathscr P_n} \P_p^{\otimes n} \left( a_n \left\{ \sup_{t \in [0,1]} \sqrt{\tn \hat h_n^{loc}(t)} \left\vert \hat p_n^{loc}(t, \hat h_n^{loc}(t)) - p(t) \right\vert - b_n \right\} \leq x \right) \notag\\
	&\quad \geq \inf_{p \in \mathscr P_n} \P_p^{\otimes n} \left( a_n \left\{ \sup_{t \in [0,1]} \sqrt{\tn \hat h_n^{loc}(t)} \left\vert \hat p_n^{loc}(t, \hat h_n^{loc}(t)) - p(t) \right\vert - b_n \right\} \leq x, \right.\notag\\
	&\left. \hspace{3.9cm} \phantom{\frac{\left\vert \hat p_n^{loc}(t, \hat h_n^{loc}(t)) - p(t) \right\vert}{\hat z_n(t)}} \hat j_n(k\delta_n) \geq k_n(k\delta_n) \text{ for all } k \in T_n \right) \notag\\
	&\quad \geq \inf_{p \in \mathscr P_n} \P_p^{\otimes n} \left( a_n \left\{ \sup_{t \in [0,1]} \sqrt{\tn \hat h_n^{loc}(t)} \left\vert \hat p_n^{loc}(t, \hat h_n^{loc}(t)) - \E_p^{\chi_1} \hat p_n^{loc}(t, \hat h_n^{loc}(t)) \right\vert - b_n \right\} \leq x_{1,n}, \right. \notag\\
	&\left. \phantom{\frac{\left\vert \hat p_n^{loc}(t, \hat h_n^{loc}(t)) - p(t) \right\vert}{\hat z_n(t)}} \hspace{4cm} \hat j_n(k\delta_n) \geq k_n(k\delta_n) \text{ for all } k \in T_n  \right) \notag\\
	&\quad \geq \inf_{p \in \mathscr P_n}  \P_p^{\otimes n} \left( a_n \left\{ \sup_{t \in [0,1]} \sqrt{\tn \hat h_n^{loc}(t)} \left\vert \hat p_n^{loc}(t, \hat h_n^{loc}(t)) - \E_p^{\chi_1} \hat p_n^{loc}(t, \hat h_n^{loc}(t)) \right\vert - b_n \right\} \leq x_{1,n} \right) \notag\\
	&\hspace{5cm}-\sup_{p \in \mathscr P_n} \P_p^{\chi_2} \left( \hat j_n(k\delta_n) < k_n(k\delta_n) \text{ for some } k \in T_n  \right) \notag\\
	&\quad = \inf_{p \in \mathscr P_n} \E_p^{\otimes n} \Bigg[ \P_p^{\otimes n} \Bigg( a_n \Bigg\{ \max_{k \in T_n} \sqrt{\tn \hat h_{n,k}^{loc}} \left\vert \hat p_n^{(1)}(k \delta_n, \hat h_{n,k}^{loc}) - \E_p^{\chi_1} \hat p_n^{(1)}(k \delta_n, \hat h_{n,k}^{loc}) \right\vert \\
	&\hspace{8.5cm} - b_n \Bigg\} \leq x_{1,n} \Bigg\vert \chi_2 \Bigg) \Bigg] + o(1)\notag
\end{align} 
for $n \to \infty$.

\paragraph{Step 2 (Reduction to the supremum over a non-stationary Gaussian process)}
In order to bound \eqref{first_max} from below note first that
\begin{align*}
	&\P_p^{\otimes n} \Bigg( a_n \Bigg\{ \max_{k \in T_n} \sqrt{\tn \hat h_{n,k}^{loc}} \left\vert \hat p_n^{(1)}(k \delta_n, \hat h_{n,k}^{loc}) - \E_p^{\chi_1} \hat p_n^{(1)}(k \delta_n, \hat h_{n,k}^{loc}) \right\vert - b_n \Bigg\} \leq x_{1,n} \Bigg\vert \chi_2 \Bigg) \\
	& \geq \P_p^{\otimes n} \Bigg( a_n \Bigg\{ \max_{k \in T_n} \sqrt{\frac{\tn \hat h_{n,k}^{loc}}{p(k\delta_n)}} \left\vert \hat p_n^{(1)}(k \delta_n, \hat h_{n,k}^{loc}) - \E_p^{\chi_1} \hat p_n^{(1)}(k \delta_n, \hat h_{n,k}^{loc}) \right\vert - b_n \Bigg\} \leq \frac{x_{1,n}}{\sqrt{L^*}}  \Bigg\vert \chi_2 \Bigg).
\end{align*}
Using the identity $\vert x \vert = \max \{ x,-x\}$, we arrive at
\begin{align*}
	&\P_p^{\otimes n} \Bigg( a_n \Bigg\{ \max_{k \in T_n} \sqrt{\tn \hat h_{n,k}^{loc}} \left\vert \hat p_n^{(1)}(k \delta_n, \hat h_{n,k}^{loc}) - \E_p^{\chi_1} \hat p_n^{(1)}(k \delta_n, \hat h_{n,k}^{loc}) \right\vert - b_n \Bigg\} \leq x_{1,n} \Bigg\vert \chi_2 \Bigg) \notag\\*
	&\hspace{9.7cm} \geq 1-P_{1,p}-P_{2,p}
\end{align*}
with
\begin{align*}
	P_{1,p} &= \P_p^{\otimes n} \Bigg( a_n \Bigg\{ \max_{k \in T_n} \sqrt{\frac{\tn \hat h_{n,k}^{loc}}{p(k\delta_n)}} \Big( \hat p_n^{(1)}(k \delta_n, \hat h_{n,k}^{loc}) - \E_p^{\chi_1} \hat p_n^{(1)}(k \delta_n, \hat h_{n,k}^{loc}) \Big) \\
	&\hspace{8cm} - b_n \Bigg\} > \frac{x_{1,n}}{\sqrt{L^*}}  \Bigg\vert \chi_2 \Bigg) \\
	P_{2,p} &= \P_p^{\otimes n} \Bigg( a_n \Bigg\{ \max_{k \in T_n} \sqrt{\frac{\tn \hat h_{n,k}^{loc}}{p(k\delta_n)}} \Big( \E_p^{\chi_1} \hat p_n^{(1)}(k \delta_n, \hat h_{n,k}^{loc}) - \hat p_n^{(1)}(k \delta_n, \hat h_{n,k}^{loc}) \Big) \\
	&\hspace{8cm} - b_n \Bigg\} > \frac{x_{1,n}}{\sqrt{L^*}}  \Bigg\vert \chi_2 \Bigg).
\end{align*}
In order to approximate the maxima in $P_{1,p}$ and $P_{2,p}$ by a supremum over a Gaussian process, we verify the conditions in Corollary 2.2 developed recently in \cite{chernozhukov_chetverikov_kato2014Approx}. For this purpose, consider the empirical process
\begin{align*}
	\G_n^p f = \frac{1}{\sqrt{\tn}} \sum_{i=1}^{\tn} \Big( f(X_i) - \E_p f(X_i) \Big), \quad f \in \mathcal F_n
\end{align*}
indexed by
\begin{align*}
	\mathcal F_n^p = \left\{ f_{n,k} \ : \ k \in T_n \right\}
\end{align*}
with
\begin{align*}
	f_{n,k}:\R &\to \R \\
		x &\mapsto \left( \tn \hat h_{n,k}^{loc} \, p(k\delta_n) \right)^{-\frac{1}{2}}K \left( \frac{k\delta_n - x}{\hat h_{n,k}^{loc}} \right) .
\end{align*}
Note that \cite{chernozhukov_chetverikov_kato2014Approx} require the class of functions to be centered. We subsequently show that the class $\mathcal F_n^p$ is Euclidean, which implies by Lemma \ref{Eucl_VC} that the corresponding centered class is VC. It therefore suffices to consider the uncentered class $\mathcal F_n^p$. Note furthermore that $f_{n,k}$ are random functions but depend on the second sample $\chi_2$ only. Conditionally on $\chi_2$, any function $f_{n,k} \in \mathcal F_n^p$ is measurable as $K$ is continuous. Due to the choice of $\kappa_2$ and due to
\begin{align*} 
	\hat h_{n,k}^{loc} \geq 2^{-u_n} \cdot \frac{(\log \tn)^{\kappa_2}}{\tn} \geq \frac{(\log \tn)^{\kappa_2-c_1\log 2}}{\tn}
\end{align*} 
the factor
\begin{align} \label{hlocbound}
	\left( \tn \hat h_{n,k}^{loc} \, p(k\delta_n) \right)^{-\frac{1}{2}} \leq \frac{1}{\sqrt{M}} (\log \tn)^{\frac{1}{2}Ê(c_1\log 2 - \kappa_2)}
\end{align}
tends to zero logarithmically. We now show that $\mathcal F_n^p$ is Euclidean with envelope
\begin{align*}
	F_n = \frac{\Vert K \Vert_{\sup}}{\sqrt{M}} (\log \tn)^{\frac{1}{2}Ê(c_1\log 2 - \kappa_2)}.
\end{align*}
Note first that
\begin{align*}
	\mathcal F_n^p \subset \mathscr F = \left\{ f_{u,h,t} : t \in \R, \ 0<u \leq \frac{1}{\sqrt{M}} (\log \tn)^{\frac{1}{2}Ê(c_1\log 2 - \kappa_2)}, \ 0<h\leq 1\right\}
\end{align*}
with
\begin{align*}
	f_{u,h,t}(\cdot) = u \cdot K \left( \frac{t - \cdot}{h} \right).
\end{align*}
Hence,
\begin{align*}
	N\left(\mathcal F_n^p, \Vert \cdot \Vert_{L^1(Q)}, \varepsilon F_n \right) \leq N\left(\mathscr F, \frac{\Vert \cdot \Vert_{L^1(Q)}}{F_n}, \varepsilon \right)
\end{align*}
for all probability measures $Q$ and it therefore suffices to show that $\mathscr F$ is Euclidean. To this aim, note that for any $f_{u,h,t}, f_{v,g,s} \in \mathscr F$ and for any probability measure $Q$,
\begin{align*}
	&\frac{\Vert f_{u,h,t} - f_{v,g,s} \Vert_{L^1(Q)}}{F_n} \\
	&\hspace{2cm}\leq \frac{\Vert f_{u,h,t} - f_{v,h,t} \Vert_{L^1(Q)}}{F_n} + \frac{\Vert f_{v,h,t} - f_{v,g,s} \Vert_{L^1(Q)}}{F_n} \\
	&\hspace{2cm}\leq \vert u-v \vert \cdot \frac{\Vert K \Vert_{\sup}}{F_n} + \frac{1}{\Vert K \Vert_{\sup}} \left\Vert K \left( \frac{t - \cdot}{h} \right) - K \left( \frac{s - \cdot}{g} \right) \right\Vert_{L^1(Q)}.
\end{align*}
Thus, using the estimate of the covering numbers in \eqref{covering_K} and Lemma 14 in \cite{nolan_pollard1987}, there exist constants $A' = A'(A,K)$ and $\nu' = \nu+1$ with
\begin{align*}
	\sup_{Q} N\left(\mathscr F, \frac{\Vert \cdot \Vert_{L^1(Q)}}{F_n}, \varepsilon \right) \leq  \left( \frac{A'}{\varepsilon} \right)^{\nu'}
\end{align*}
for all $0<\varepsilon\leq 1$. That is, $\mathscr F$ is Euclidean with the constant function $F_n$ as envelope, and in particular
\begin{align} \label{FnVC}
\begin{split}
	\limsup_{n \to \infty}Ê\sup_{Q} N\left(\mathcal F_n^p, \Vert \cdot \Vert_{L^1(Q)}, \varepsilon F_n \right) \leq  \left( \frac{A'}{\varepsilon} \right)^{\nu'}.
\end{split}
\end{align}
Hence, by Lemma \ref{Eucl_VC}, the $\P_p$-centered class $\mathcal F_n^{p,0}$ corresponding to $\mathcal F_n^p$ is VC with envelope $2F_n$ and
\begin{align*}
	N\left(\mathcal F_n^{p,0}, \Vert \cdot \Vert_{L^2(Q)}, 2\varepsilon F_n \right) \leq  \left( \frac{A''}{\varepsilon} \right)^{\nu''}
\end{align*}
and VC characteristics $A'' = A''(A,K)$ and $\nu'' = \nu''(\nu)$. Next, we verify the Bernstein condition
\begin{align*}
	\sup_{p \in \mathscr P_n} \sup_{f \in \mathcal F_n^{p,0}}Ê\int \vert f(y) \vert^l p(y) \diff y \leq \sigma_n^2 B_n^{l-2}
\end{align*}
for some $B_n \geq \sigma_n > 0$ and $B_n \geq 2F_n$ and $l=2,3,4$. 
First, for $f_{n,k}^0 \in \mathcal F_n^{p,0}$,
\begin{align*}
	&\max_{k \in T_n} \int \vert f_{n,k}^0(y) \vert^2 \, p(y) \diff y \\
	&\quad = \max_{k\in T_n} \Big( \tn \, p(k\delta_n)  \Big)^{-1} \int_{-1}^1 \left\{ K( x ) - \hat h_{n,k}^{loc} \int K(y) p\left(k\delta_n+\hat h_{n,k}^{loc} y\right) \diff y \right\}^2 p\left(k\delta_n+\hat h_{n,k}^{loc}x\right) \diff x \\
	&\quad \leq \sigma_n^2
\end{align*}
with
\begin{align*}
	\sigma_n^2 = \frac{2L^* (\Vert K \Vert_{\sup}+L^* \Vert K \Vert_1)^2}{M\tn}.
\end{align*}
Also, using \eqref{hlocbound},
\begin{align*}
	&\max_{k \in T_n} \int \vert f_{n,k}(y) \vert^3 \, p(y) \diff y \\
	&\quad= \max_{k\in T_n} \left( \tn \hat h_{n,k}^{loc} \, p(k\delta_n) \right)^{-3/2} \hat h_{n,k}^{loc} \int_{-1}^1 \left\{ K( x ) - \hat h_{n,k}^{loc} \int K(x) p\left(k\delta_n+\hat h_{n,k}^{loc} y\right) \diff y \right\}^3 p\left(k\delta_n+\hat h_{n,k}^{loc} x\right) \diff x \\
		&\quad\leq \sigma_n^2 (\Vert K \Vert_{\sup} + L^* \Vert K \Vert_1) \cdot \max_{k \in T_n} \left( \tn \hat h_{n,k}^{loc} \, p(k\delta_n) \right)^{-1/2} \\
		&\quad\leq \sigma_n^2 \cdot B_n
\end{align*}
with
$$B_n = \max \left\{ 1 + L^* \frac{\Vert K \Vert_1}{\Vert K \Vert_{\sup}}, 2 \right\} F_n.$$
The condition
\begin{align*}
	\sup_{p \in \mathscr P_n} \sup_{f \in \mathcal F_n^{p,0}} \int \vert f(y) \vert^4 p(y) \diff y \leq \sigma_n^2 B_n^2
\end{align*}
follows analogously. Furthermore, it holds that $\Vert 2F_n \Vert_{\sup} = B_n$. According to Corollary~2.2 in \cite{chernozhukov_chetverikov_kato2014Approx}, for sufficiently large $n \geq n_0(c_1,\kappa_2,L^*,K)$ such that $B_n \geq \sigma_n$, there exists a random variable
\begin{align*}
	Z_{n,p}^0 \overset{\mathcal D}{=} \max_{f \in \mathcal F_n^{p,0}}  G_{\P_p} f ,
\end{align*}
and universal constants $c_{12}°$ and $c_{13}°$, such that for $\eta = \frac{1}{4} (\kappa_2 - c_1 \log 2 - 4)>0$ 
\begin{align*}
	&\sup_{p \in \mathscr P_n} \P \left( a_n \sqrt{\tn} \left\vert \max_{f \in \mathcal F_n^{p,0}}  \G_n f  - Z_{n,p}^0 \right\vert > \varepsilon_{2,n} \big\vert \chi_2 \right) \leq c_{12}° \left( (\log \tn)^{-\eta} + \frac{\log \tn}{\tn} \right),
\end{align*}
where 
\begin{align*}
	\varepsilon_{2,n} &= a_n \left( \frac{B_n K_n}{(\log \tn)^{-\eta/2}} + \frac{\tn^{1/4} \sqrt{B_n \sigma_n} K_n^{3/4}}{(\log \tn)^{-\eta/2}} + \frac{\tn^{1/3} (B_n\sigma_n^2 K_n^2)^{1/3}}{(\log \tn)^{-\eta/3}} \right)
\end{align*}
with $K_n = c_{13}°\nu''(\log \tn \vee \log(A'' B_n/\sigma_n))$, and $G_{\P_p}$ is a version of the $\P_p$-Brownian motion. That is, it is centered and has the covariance structure
\begin{align*}
	\E_p^{\chi_1} f(X_1) g(X_1)
\end{align*}
for all $f,g \in \mathcal F_n^{p,0}$. As can be seen from an application of the It\^o isometry, it possesses  in particular  the distributional representation
\begin{align} \label{rep_bb}
	(G_{\P_p} f)_{f \in {\mathcal F}_n^{p,0}} \overset{\mathcal D}{=} \left( \int f(x) \sqrt{p(x)} \diff W(x) \right)_{f \in {\mathcal F}_n^{p,0}},
\end{align}
where $W$ is a standard Brownian motion independent of $\chi_2$. An easy calculation furthermore shows that $\varepsilon_{2,n}$ tends to zero for $n \to \infty$ logarithmically due to the choice of $\eta$. Finally, 
\begin{align*}
	&\sup_{p \in \mathscr P_n} P_{1,p} \\
	& \leq \sup_{p \in \mathscr P_n}  \P_p^{\otimes n} \left( a_n \left( \sqrt{\tn} \max_{f \in \mathcal F_n^p} \mathbb G_n^p f  - b_n \right) > \frac{x_{1,n}}{\sqrt{L^*}}, \ a_n \sqrt{\tn} \left\vert \max_{f \in \mathcal F_n^{p,0}}  \G_n^p f  - Z_{n,p}^0 \right\vert \leq \varepsilon_{2,n} \Big\vert \chi_2 \right) \\
	&\qquad +\sup_{p \in \mathscr P_n} \P_p^{\otimes n} \left( a_n\sqrt{\tn} \left\vert \max_{f \in \mathcal F_n^{p,0}}  \G_n^p f  - Z_{n,p}^0 \right\vert > \varepsilon_{2,n} \Big\vert \chi_2 \right) \\
	 & \leq \sup_{p \in \mathscr P_n}  \P_p^{\otimes n} \left( a_n \left( \sqrt{\tn} \, Z_{n,p}^0 - b_n \right) > x_{2,n} \Big\vert \chi_2 \right) + o(1)
\end{align*}
for $n \to \infty$, with 
$$x_{2,n} = \frac{x_{1,n}}{\sqrt{L^*}} - \varepsilon_{2,n} = \frac{x - \varepsilon_{1,n}}{\sqrt{L^*}} - \varepsilon_{2,n} = \frac{x}{\sqrt{L^*}}+o(1).$$
The probability $P_{2,p}$ is bounded in the same way, leading to
\begin{align*}
	&\inf_{p \in \mathscr P_n} \P_p^{\otimes n} \Bigg( a_n \Bigg\{ \max_{k \in T_n} \sqrt{\tn \hat h_{n,k}^{loc}} \left\vert \hat p_n^{(1)}(k \delta_n, \hat h_{n,k}^{loc}) - \E_p^{\chi_1} \hat p_n^{(1)}(k \delta_n, \hat h_{n,k}^{loc}) \right\vert - b_n \Bigg\} \leq x_{1,n} \Bigg\vert \chi_2 \Bigg) \notag\\
	&\hspace{4cm} \geq 2 \, \inf_{p \in \mathscr P_n} \P_p^{\otimes n} \left( a_n \left( \sqrt{\tn} \, Z_{n,p}^0 - b_n \right) \leq x_{2,n} \Big\vert \chi_2 \right) -1+ o(1).
\end{align*}
Next, we show that there exists some sequence $(\varepsilon_{3,n})$ converging to zero, such that
\begin{align} \label{max-max}
	\sup_{p \in \mathscr P_n} \P \left( a_n \sqrt{\tn} \left\vert \max_{f \in \mathcal F_n^{p,0}} G_{\P_p} f - \max_{f \in \mathcal F_n^p} G_{\P_p} f \right\vert > \varepsilon_{3,n} \Big\vert \chi_2 \right) = o(1).
\end{align}
For this purpose, note first that
\begin{align*}
	&\sup_{p \in \mathscr P_n} \P \left( a_n \sqrt{\tn} \left\vert \max_{f \in \mathcal F_n^{p,0}} G_{\P_p} f - \max_{f \in \mathcal F_n^p} G_{\P_p} f \right\vert > \varepsilon_{3,n} \Big\vert \chi_2 \right) \\
	&\hspace{4cm} \leq \sup_{p \in \mathscr P_n} \P \left( \vert Y \vert a_n \sqrt{\tn}  \max_{f \in \mathcal F_n^p} \left\vert \P_p f \right\vert > \varepsilon_{3,n} \Big\vert \chi_2 \right)
\end{align*}
with $Y \sim \mathcal N(0,1)$. Due to the choice of $c_1$
\begin{align*}
	a_n \sqrt{\tn} \max_{f \in \mathcal F_n^p} \left\vert \P_p f \right\vert \leq a_n \frac{L^* \Vert K \Vert_1}{\sqrt{M}} 2^{-u_n/2} = o(1),
\end{align*}
which proves \eqref{max-max}. Following the same steps as before
\begin{align*}
	&\inf_{p \in \mathscr P_n} \P_p^{\otimes n} \left( a_n \left( \sqrt{\tn} \, Z_{n,p}^0 - b_n \right) \leq x_{2,n} \Big\vert \chi_2 \right) \\
	&\hspace{4cm} \geq \inf_{p \in \mathscr P_n} \P\left( a_n \left( \sqrt{\tn} \, \max_{f \in \mathcal F_n^p} G_{\P_p} f - b_n \right) \leq x_{3,n} \Big\vert \chi_2 \right) + o(1)
\end{align*}
with $x_{3,n} = x_{2,n}-\varepsilon_{3,n}$.

\medskip

\noindent Finally we conduct a further approximation 
\begin{align*}
	 \left( Y_{n,p}(k) \right)_{k \in T_n} = \left( \frac{1}{\sqrt{\hat h_{n,k}^{loc}}} \int K \left( \frac{k\delta_n-x}{\hat h_{n,k}^{loc}} \right) \diff W(x) \right)_{k \in T_n}
\end{align*}
to the process
\begin{align*}
	\left( \sqrt{\tn} \int f_{n,k}(x) \sqrt{p(x)} \diff W(x) \right)_{k \in T_n} \overset{\mathcal D}{=} \left(  \sqrt{\tn} \, G_{\P_p} f_{n,k} \right)_{k \in T_n}
\end{align*}
in order to obtain to a suitable intermediate process for Step 3. With
\begin{align*}
	V_{n,p}(k)&= \sqrt{\tn} \, W( f_{n,k} \sqrt{p}) -  Y_{n,p}(k) \\
		& =\sqrt{\tn} \int  f_{n,k}(x) \left( \sqrt{p(x)} - \sqrt{p(k\delta_n)} \right) \diff W(x),
\end{align*}
it remains to show that
\begin{align} \label{pisier}
	\lim_{n \to \infty} \sup_{p \in \mathscr P_n} \P^W \left( a_n \max_{k \in T_n} \vert V_{n,p}(k) \vert > \varepsilon_{4,n}  \right) = 0
\end{align}
for some sequence $(\varepsilon_{4,n})_{n \in \N}$ converging to zero. Note first that
\begin{align*}
 &\max_{k \in T_n}\E^W V_{n,k}^2 \\
		&=  \max_{k \in T_n}\tn\int f_{n,k}(x)^2 \left( \sqrt{p(x)} - \sqrt{p(k\delta_n)} \right)^2 \diff x \\
		& = \max_{k \in T_n}\frac{1}{p(k\delta_n)} \int K(x)^2 \left( \sqrt{p(k\delta_n + \hat h_{n,k}^{loc}x)} - \sqrt{p(k\delta_n)} \right)^2 \diff x \\
	& \leq \max_{k \in T_n}\frac{1}{p(k\delta_n)} \int K(x)^2 \left\vert p(k\delta_n + \hat h_{n,k}^{loc}x) - p(k\delta_n) \right\vert \diff x \\
	& \leq \max_{k \in T_n}\frac{L^* \Vert K \Vert_2^2}{p(k\delta_n)} \left( \hat h_{n,k}^{loc} \right)^{\beta_*} \\
	& \leq \frac{L^* \Vert K \Vert_2^2}{M} \left( \log \tn \right)^{-c_1 \beta_*  \log 2}.
\end{align*}
Denoting by $\Vert \cdot \Vert_{\psi_2}$ the Orlicz norm corresponding to $\psi_2(x) = \exp(x^2)-1$, we deduce for sufficiently large $n \geq n_0(c_1,\beta_*,L^*,K,M)$
\begin{align*}
	 &\sup_{p \in \mathscr P_n} \left\Vert a_n \cdot \max_{k \in T_n} \vert V_{n,p}(k) \vert \right\Vert_{\psi_2} \\
	&\quad \leq  \sup_{p \in \mathscr P_n} a_n \cdot c(\psi_2) \, \psi_2^{-1} \left( \delta_n^{-1} \right) \, \max_{k \in T_n} \Vert V_{n,p}(k) \Vert_{\psi_2} \\
	&\quad\leq a_n \cdot c(\psi_2) \, \sqrt{\log \left( 1+ \delta_n^{-1} \right)} \, \left( \frac{L^* \Vert K \Vert_2^2}{M} \left( \log \tn \right)^{-c_1 \beta_*  \log 2} \right)^{\frac{1}{2}} \Vert Y \Vert_{\psi_2} \\
	&\quad\leq a_n \cdot c \, \sqrt{\log \tn} \, \left( \log \tn \right)^{-\frac{1}{2} c_1 \beta_*  \log 2}.
\end{align*}
The latter expression converges to zero due to the choice of $c_1$ in \eqref{c22}. Thus, \eqref{pisier} is established.
Following the same steps as before, we obtain
\begin{align*}
	&\inf_{p \in \mathscr P_n} \P_p^{\otimes n} \left( a_n \left( \sqrt{\tn} \max_{k \in T_n}  G_{\P_p} f_{n,k}  - b_n \right) \leq x_{3,n} \Big\vert \chi_2 \right) \\
	&\hspace{3cm} \geq \inf_{p \in \mathscr P_n} \P^W \left( a_n \left( \max_{k \in T_n} Y_{n,p}(k) - b_n \right) \leq x_{4,n} \right) + o(1) 
\end{align*}
for $n \to \infty$, with $x_{4,n} = x_{3,n} - \varepsilon_{4,n}$. 

\paragraph{Step 3 (Reduction to the supremum over a stationary Gaussian process)}
We now use Slepian's comparison inequality in order to pass over to the least favorable case. Since $K$ is symmetric and of bounded variation, it possesses a representation
\begin{align*}
	K(x) = \int_{-1}^x g \diff P
\end{align*}
for all but at most countably many $x \in [-1,1]$, where $P$ is some symmetric probability measure on $[-1,1]$ and $g$ is some measurable odd function with $\vert g \vert \leq TV(K)$. Using this representation, and denoting by
\begin{align}
	W_{k,l}(z) &= \sqrt{\frac{1}{\hat h_{n,k}^{loc}}} \left\{ W(k\delta_n+\hat h_{n,k}^{loc}) - W(k\delta_n+z\hat h_{n,k}^{loc}) \right\} \notag\\
		&\qquad - \sqrt{\frac{1}{\hat h_{n,l}^{loc}}} \left\{ W(l\delta_n+\hat h_{n,l}^{loc}) - W(l\delta_n+z\hat h_{n,l}^{loc}) \right\} \notag \\
\begin{split}\label{tildeW}
	\tilde W_{k,l}(z) &= \sqrt{\frac{1}{\hat h_{n,k}^{loc}}} \left\{ W(k\delta_n-z\hat h_{n,k}^{loc}) - W(k\delta_n+z\hat h_{n,k}^{loc}) \right\} \\
		&\qquad + \sqrt{\frac{1}{\hat h_{n,l}^{loc}}} \left\{ W(l\delta_n+z\hat h_{n,l}^{loc}) - W(l\delta_n-z\hat h_{n,l}^{loc}) \right\},
\end{split}
\end{align}
Fubini's theorem with one stochastic integration and the Cauchy-Schwarz inequality yield for any $k,l \in T_n$ 
\begin{align*}
	&\E_W \Big( Y_{n,p}(k) - Y_{n,p}(l) \Big)^2 \\
	&\quad = \E_W \left( \sqrt{\frac{1}{\hat h_{n,k}^{loc}  }} \int \int_{-1}^{\frac{x-k\delta_n}{\hat h_{n,k}^{loc}}} g(z) \diff P(z)  \mathbbm{1} \left\{ \vert x-k\delta_n \vert \leq \hat h_{n,k}^{loc} \right\} \diff W(x) \right. \\
	&\hspace{2cm} \left. - \sqrt{\frac{1}{\hat h_{n,l}^{loc}  }} \int \int_{-1}^{\frac{x-l\delta_n}{\hat h_{n,l}^{loc}}} g(z) \diff P(z) \mathbbm{1} \left\{ \vert x-l\delta_n \vert \leq \hat h_{n,l}^{loc} \right\} \, \diff W(x) \right)^2 \\
	&\quad = \E_W \left( \int_{-1}^1 g(z) \left\{ \sqrt{\frac{1}{\hat h_{n,k}^{loc}  }} \int\mathbbm{1} \left\{ k\delta_n+z \hat h_{n,k}^{loc} \leq x \leq k\delta_n+\hat h_{n,k}^{loc} \right\} \diff W(x) \right. \right. \\
	&\hspace{2cm} \left. \left. - \sqrt{\frac{1}{\hat h_{n,l}^{loc}  }} \int\mathbbm{1} \left\{ l\delta_n+z \hat h_{n,l}^{loc} \leq x \leq l\delta_n+\hat h_{n,l}^{loc} \right\} \diff W(x) \right\} \diff P(z) \right)^2 \\
	&\quad = \E_W \left( \int_{-1}^1 g(z) W_{k,l}(z) \diff P(z) \right)^2 \\
	&\quad = \E_W \left( \int_0^1 g(z) \Big( W_{k,l}(z) - W_{k,l}(-z) \Big) \diff P(z) \right)^2 \\
	&\quad = \E_W \int_0^1 \int_0^1 g(z) g(z') \tilde W_{k,l}(z) \tilde W_{k,l}(z') \diff P(z) \diff P(z') \\
	&\quad \leq \int_0^1 \int_0^1 \vert g(z) g(z') \vert \left\{ \E_W \tilde W_{k,l}(z)^2 \E_W \tilde W_{k,l}(z')^2 \right\}^{1/2} \diff P(z) \diff P(z').
\end{align*}
We verify in Lemma \ref{tildeWmoments} that
\begin{align*}
	\E_W \tilde W_{k,l}(z)^2 \leq 4
\end{align*}
for $z \in [0,1]$, so that
\begin{align*}
	\E_W ( Y_{n,p}(k) - Y_{n,p}(l) )^2 &\leq 4 \left( \int_0^1 \vert g(z) \vert \diff P(z) \right)^2\leq TV(K)^2
\end{align*}
for all $k,l \in T_n$. Consider now the Gaussian process
\begin{align*}
	Y_{n,\min}(k) = \frac{c_{15}°}{\sqrt{\delta_n}} \int K \left( \frac{k\delta_n-x}{\delta_n/2} \right) \diff W(x), \quad k \in T_n,
\end{align*}
with
$$c_{15}° =\frac{TV(K)}{\Vert K \Vert_2}.$$
Furthermore,
\begin{align*}
	\E_W \left( Y_{n,\min}(k) - Y_{n,\min}(l) \right)^2 = \E_W Y_{n,\min}(k)^2 + \E_W Y_{n,\min}(l)^2 = TV(K)^2
\end{align*}
for all $k,l \in T_n$ with $k \neq l$, so that
\begin{align} \label{slepian1}
	\E_W ( Y_{n,p}(k) - Y_{n,p}(l) )^2 \leq \E_W \left( Y_{n,\min}(k) - Y_{n,\min}(l) \right)^2
\end{align}
for all $k,l \in T_n$. In order to apply Slepian's comparison inequality we however need  coinciding second moments. For this aim, we analyze the modified Gaussian processes
\begin{align*}
	\bar Y_{n,p}(k) &= Y_{n,p}(k) + c_{16}° Z \\
	\bar Y_{n,\min}(k) &= Y_{n,\min}(k) + c_{17}° Z
\end{align*}
with
\begin{align*}
	c_{16}° = c_{16}°(K) = \frac{TV(K)}{\sqrt{2}}, \qquad 	c_{17}° = c_{17}°(K) = \Vert K \Vert_2,
\end{align*}
and for some standard normally distributed random variable $Z$ independent of $(Y_{n,p}(k))_{k \in T_n}$ and $(Y_{n,\min}(k))_{k \in T_n}$. Note that these processes have the same increments as the processes before. In particular
\begin{align*}
	\E_W \left( \bar Y_{n,p}(k) - \bar Y_{n,p}(l) \right)^2 &= \E_W  \left(  Y_{n,p}(k) -  Y_{n,p}(l) \right)^2 \\
		&\leq \E_W \left( Y_{n,\min}(k) - Y_{n,\min}(l) \right)^2 \\
		&= \E_W \left( \bar Y_{n,\min}(k) - \bar Y_{n,\min}(l) \right)^2
\end{align*}
for all $k,l \in T_n$ by inequality \eqref{slepian1}. With this specific choice of $c_{16}$ and $c_{17}$, they furthermore have coinciding second moments
\begin{align*}
	\E_W \bar Y_{n,p}(k)^2 = \E_W \bar Y_{n,\min}(k)^2 = \frac{TV(K)^2}{2} + \Vert K \Vert_2^2
\end{align*}
for all $k \in T_n$. Then, 
\begin{align*}
	&\inf_{p \in \mathscr P_n} \P^W \left( a_n \left( \max_{k \in T_n}  Y_{n,p}(k)  - b_n \right) \leq x_{3,n}  \right) \\
		&\hspace{2cm} =\inf_{p \in \mathscr P_n}  \P^W \left( a_n \left( \max_{k \in T_n}  \bar Y_{n,p}(k)  - c_{16}° Z - b_n \right) \leq x_{4,n}  \right) \\
		&\hspace{2cm} \geq \inf_{p \in \mathscr P_n} \P^W \left( a_n \left( \max_{k \in T_n}  \bar Y_{n,p}(k)  - c_{16}° Z - b_n \right) \leq x_{4,n}, \  - Z  \leq \frac{1}{3c_{16}°} b_n \right) \\
		&\hspace{2cm} \geq \inf_{p \in \mathscr P_n} \P^W \left( a_n \left( \max_{k \in T_n}  \bar Y_{n,p}(k)  - \frac{2}{3} b_n \right) \leq x_{4,n} \right) - \P \left( - Z  > \frac{1}{3c_{16}° } b_n \right) \\
		&\hspace{2cm} \geq \inf_{p \in \mathscr P_n} \P^W \left( a_n \left( \max_{k \in T_n}  \bar Y_{n,p}(k)  - \frac{2}{3} b_n \right) \leq x_{4,n} \right) + o(1)
\end{align*}
for $n \to \infty$. Slepian's inequality in the form of Corollary 3.12 in \cite{ledoux_talagrand1991} yields
\begin{align*}
	 &\inf_{p \in \mathscr P_n}\P^W \left( a_n \left( \max_{k \in T_n}  \bar Y_{n,p}(k)  -\frac{2}{3}b_n \right) \leq x_{4,n} \right) \\
	 	&\hspace{2cm} \geq \P^W \left( a_n \left( \max_{k \in T_n} \bar Y_{n,\min}(k) - \frac{2}{3}b_n \right) \leq x_{4,n} \right).
\end{align*}

\paragraph{Step 4 (Limiting distribution theory)}
Finally, we pass over to an iid sequence and apply extreme value theory. Together with
\begin{align*}
	&\P^W \left( a_n \left( \max_{k \in T_n} \bar Y_{n,\min}(k) - \frac{2}{3}b_n \right) \leq x_{4,n} \right) \\
	&\qquad \geq \P^W \left( a_n \left( \max_{k \in T_n} Y_{n,\min}(k) + c_{17}° Z - \frac{2}{3}b_n \right) \leq x_{4,n}, \ Z \leq \frac{1}{3c_{17}°} b_n  \right) \\
	&\qquad \geq \P^W \left( a_n \left( \max_{k \in T_n} Y_{n,\min}(k) - \frac{1}{3}b_n \right) \leq x_{4,n}  \right) - \P \left(Z>\frac{1}{3c_{17}°}b_n \right) \\
	&\qquad = \P^W \left( a_n \left( \max_{k \in T_n} Y_{n,\min}(k) - \frac{1}{3} b_n \right) \leq x_{4,n} \right) + o(1) 
\end{align*}
as $n \to \infty$, we finally obtain
\begin{align*}
	&\inf_{p \in \mathscr P_n} \P_p^{\otimes n} \left( a_n \left( \sup_{t \in [0,1]} \sqrt{\tn \hat h_n^{loc}(t)}\left\vert \hat p_n^{loc}(t, \hat h_n^{loc}(t)) - p(t) \right\vert - b_n \right) \leq x \right) \\
	&\hspace{2cm} \geq 2 \, \P \left( a_n \left( \max_{k \in T_n} Y_{n,\min}(k) - \frac{1}{3} b_n \right) \leq x_{4,n} \right) - 1 + o(1).
\end{align*}
Theorem 1.5.3 in \cite{leadbetter_lindgren_rootzen1983} yields now
\begin{align} \label{conv_fix}
	F_n(x) = \P^W \left( a_n \left( \max_{k\in T_n}  Y_{n,\min}(k) - \frac{1}{3}b_n \right) \leq x \right) \longrightarrow F(x) = \exp(-\exp(-x))
\end{align}
for any $x \in \R$. It remains to show, that $F_n(x_n) \to F(x)$ for some sequence $x_n \rightarrow x$ as $n \to \infty$. Because $F$ is continuous in $x$, there exists for any $\varepsilon > 0$ some $\delta = \delta(\varepsilon) > 0$ such that $\vert y-x \vert \leq \delta$ imlies $\vert F(x) - F(y) \vert \leq \varepsilon/2$. In particular, for $y=x \pm \delta$,
\begin{align} \label{x+delta}
	\vert F(x) - F(x+\delta) \vert \leq \frac{\varepsilon}{2} \quad \text{and} \quad \vert F(x) - F(x-\delta) \vert \leq \frac{\varepsilon}{2}.
\end{align}
As $x_n \to x$, there exists some $N_1 = N_1(\varepsilon)$, such that $\vert x_n - x \vert \leq \delta$ for all $n \geq N_1$. Therefore, employing the monotonicity of $F_n$, 
\begin{align*}
	\vert F_n(x_n) - F(x) \vert \leq \vert F_n(x+\delta) - F(x) \vert \vee \vert F_n(x-\delta) - F(x) \vert
\end{align*}
for $n \geq N_1$, where
\begin{align*}
	\vert F_n(x \pm \delta) - F(x) \vert \leq \vert F_n(x \pm \delta) - F(x \pm \delta) \vert + \vert F(x \pm \delta) - F(x) \vert \leq \varepsilon
\end{align*}
for $n \geq N_2 = N_2(\varepsilon)$ due to \eqref{conv_fix} and \eqref{x+delta}. Consequently,
\begin{align*}
	&\lim_{n \to \infty} \inf_{p \in \mathscr P_n} \P_p^{\otimes n} \left( a_n \left( \sup_{t \in [0,1]} \sqrt{\tn \hat h_n^{loc}(t)} \left\vert \hat p_n^{loc}(t, \hat h_n^{loc}(t)) - p(t) \right\vert - b_n \right) \leq x \right) \\
	&\hspace{2cm} \geq 2 \lim_{n \to \infty} \P \left( a_n \left( \max_{k \in T_n}  Y_{n,\min}(k) - \frac{1}{3}b_n \right) \leq x_{4,n} \right) - 1 + o(1) \\
	&\hspace{2cm} = 2 \, \P \left( \sqrt{L^*} G \leq x \right) - 1 + o(1), \quad n \to \infty,
\end{align*}
for some standard Gumbel distributed random variable $G$. 
\end{proof}

\begin{proof}[Proof of Proposition \ref{optimal_adaptation}]
The proof is based on a reduction of the supremum over the class to a maximum over two distinct hypotheses. 

\paragraph{Part 1}
For $\beta \in [\beta_*, 1)$, the construction of the hypotheses is based on the Weierstra{\ss} function as defined in \eqref{wf}. As in the proof of Proposition \ref{low_ptw_riskbound} consider the function $p_0: \R \to \R$ with
\begin{align*}
	p_0(x) = \begin{cases}
	0, \quad &\text{if } \ \vert x-t \vert \geq \frac{10}{3} \\
	\frac{1}{4} + \frac{3}{16}(x-t+2), \quad &\text{if } \  -\frac{10}{3} < x-t  < -2 \\
	\frac{1}{6} + \frac{1-2^{-\beta}}{12} W_{ \beta}(x-t), \quad &\text{if } \  \vert x-t \vert \leq 2 \\
	\frac{1}{4} - \frac{3}{16}(x-t-2), \quad &\text{if } \  2 < x-t  < \frac{10}{3} \\
	\end{cases}
\end{align*}
and the functions $p_{1,n}, \, p_{2,n} : \R \to \R$ with
\begin{align*}
	p_{1,n}(x) &= p_0(x) + q_{t+\frac{9}{4},n}(x;g_{\beta,n}) - q_{t,n}(x;g_{\beta,n}), \quad x \in \R \\
	p_{2,n}(x) &= p_0(x) + q_{t+\frac{9}{4},n}(x;c_{18}° \cdot g_{\beta,n}) - q_{t,n}(x;c_{18}° \cdot g_{\beta,n}), \quad x \in \R
\end{align*}
for $g_{\beta,n} = \frac{1}{4} n^{-1/(2\beta+1)}$ and $c_{18}° = c_{18}°(\beta) = (2 L_W(\beta))^{-1/\beta}$, where 
\begin{align*}
	q_{a,n}(x;g) = \begin{cases}
	0, \quad &\text{if } \ \vert x-a \vert > g \\
	\frac{1-2^{-\beta}}{12} \Big(W_{\beta}(x-a) - W_{\beta}(g)\Big), \quad &\text{if } \ \vert x-a \vert \leq g
	\end{cases}
\end{align*}
for $a \in \R$ and $g>0$. 

		\begin{figure}[H] 
		\centering
  		\includegraphics[width=.95\textwidth]{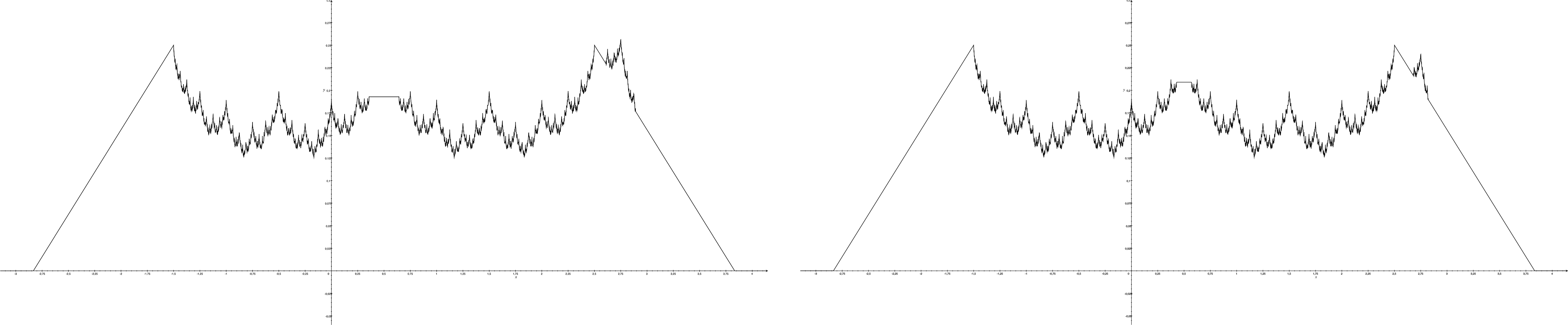}
		\caption{Functions $p_{1,n}$ and $p_{2,n}$ for $t=0.5$, $\beta=0.5$ and $n=50$}
		\label{fig:w2}
		\end{figure}
		\vspace{-5mm}

\smallskip
\noindent Following the lines of the proof of Proposition \ref{low_ptw_riskbound}, both $p_{1,n}$ and $p_{2,n}$ are contained in the class $\mathscr P_k(L,\beta_*,M,K_R,\varepsilon)$ for sufficiently large $k \geq k_0(\beta_*)$. Moreover, both $p_{1,n}$ and $p_{2,n}$ are constant on $B(t,c_{18}° \cdot g_{\beta,n})$, so that
\begin{align*}
	p_{1,n \vert B(t,c_{18}° \cdot g_{\beta,n})}, \, p_{2,n \vert B(t,c_{18}° \cdot g_{\beta,n})} \in \mathcal H_{B(t,c_{18}° \cdot g_{\beta,n})}(\infty,L) \
\end{align*}
for some constant $L=L(\beta)$. Using Lemma \ref{weier_lemma} and \eqref{dist_weier}, the absolute distance of the two hypotheses in $t$ is at least
\begin{align*}
	\vert p_{1,n}(t) - p_{2,n}(t) \vert &= \vert q_{t,n}(t;g_{\beta,n}) - q_{t,n}(t;c_{18}° \cdot g_{\beta,n}) \vert \\
		&= \frac{1-2^{- \beta}}{12} \vert W_{\beta}(g_{\beta,n}) - W_{\beta}(c_{18}° \cdot g_{\beta,n}) \vert \\
		&\geq \frac{1-2^{- \beta_*}}{12} \Big( \vert W_{\beta}(g_{\beta,n}) - W_{\beta}(0) \vert - \vert W_{\beta}(c_{18}° \cdot g_{\beta,n}) - W_{\beta}(0) \vert \Big) \\
		&\geq \frac{1-2^{- \beta_*}}{12} \Big( g_{\beta,n}^{\beta} - L_W(\beta) \left( c_{18}° \cdot g_{\beta,n} \right)^{\beta} \Big) \\*
		&\geq 2 c_{19}° g_{\beta,n}^{\beta}
\end{align*}
where 
\begin{align*}
	c_{19}° = c_{19}°(\beta_*) = \frac{1-2^{-\beta_*}}{48}.
\end{align*}
Since furthermore
\begin{align*}
	\int (p_{2,n}(x) - p_{1,n}(x)) \diff x = 0,
\end{align*}
and $\log(1+x) \leq x$ for $x > -1$, the Kullback-Leibler divergence between the associated product probability measures $\P_{1,n}^{\otimes n}$ and $\P_{2,n}^{\otimes n}$ is bounded from above by
\begin{align*}
	K(\P_{2,n}^{\otimes n},\P_{1,n}^{\otimes n}) &\leq n \, \int \frac{(p_{2,n}(x) - p_{1,n}(x))^2}{p_{1,n}(x)} \diff x \\
		&\leq 12 \, n \, \int (p_{2,n}(x) - p_{1,n}(x))^2 \diff x \\
		&= 24 \, n \, \int (q_{0,n}(x;g_{\beta,n}) - q_{0,n}(x,c_{18}° \cdot g_{\beta,n}))^2 \diff x \\
		&= 24 \, n \, \left( \frac{1-2^{-\beta}}{12} \right)^2 \Bigg( 2 \int_{c_{18}° \cdot g_{\beta,n}}^{g_{\beta,n}} \Big(W_{\beta}(x) - W_{\beta}(g_{\beta,n})\Big)^2 \diff x \\
		&\hspace{3cm}+ \int_{-c_{18}° \cdot g_{\beta,n}}^{c_{18}° \cdot g_{\beta,n}} \Big(W_{\beta}(c_{18}° \cdot g_{\beta,n}) - W_{\beta}(g_{\beta,n}) \Big)^2 \diff x \Bigg) \\
		&\leq 24 \, n \, L_W(\beta)^2 \left( \frac{1-2^{-\beta}}{12} \right)^2 \Bigg( 2 \int_{c_{18}° \cdot g_{\beta,n}}^{g_{\beta,n}} ( g_{\beta,n}-x )^{2\beta} \diff x \\
		&\hspace{7cm}+ 2(1-c_{18}°)^2 c_{18}°  g_{\beta,n}^{2\beta+1} \Bigg) \\
		&= c_{20}°		
\end{align*}
with
\begin{align*}
	c_{20}° = c_{20}°(\beta) = 48 \, L_W(\beta)^2 4^{-(2\beta+1)} \left( \frac{1-2^{-\beta}}{12} \right)^2 \Bigg( \frac{(1-c_{18}°)^{2\beta+1}}{2\beta+1} + (1-c_{18}°)^2 c_{18}°  \Bigg),
\end{align*} 
where we used Lemma \ref{weier_lemma} in the last inequality. Theorem 2.2 in \cite{tsybakov2009} then yields
\begin{align*}
	&\inf_{T_n} \sup_{p \in \mathscr S_k(\beta)} \P_p^{\otimes n} \left( n^{\frac{\beta}{2\beta+1}} \left\vert T_n(t) - p(t) \right\vert \geq c_{19}° \right) \\
	&\hspace{5cm} \geq \max \left\{ \frac{1}{4} \exp(-c_{20}°), \, \frac{1-\sqrt{c_{20}°/2}}{2} \right\} > 0.
\end{align*}

\paragraph{Part 2}
For $\beta=1$, consider the function $p_0: \R \to \R$ with
\begin{align*}
	p_0(x) = \begin{cases}
	0, \quad &\text{if } \ \vert x-t \vert > 4 \\
	\frac{1}{4} - \frac{1}{16} \vert x-t \vert, \quad &\text{if } \  \vert x-t \vert \leq 4
	\end{cases}
\end{align*}
and the functions $p_{1,n}, p_{2,n} : \R \to \R$ with
\begin{align*}
	p_{1,n}(x) &= p_0(x) + q_{t+\frac{9}{4},n}(x;g_{1,n}) - q_{t,n}(x;g_{1,n}) \\
	p_{2,n}(x) &= p_0(x) + q_{t+\frac{9}{4},n}(x;g_{1,n}/2) - q_{t,n}(x;g_{1,n}/2)
\end{align*}
for $g_{1,n} = \frac{1}{4} n^{-1/3}$, where
\begin{align*}
	q_{a,n}(x;g) = \begin{cases}
	0, \quad &\text{if } \ \vert x-a \vert > g \\
	\frac{1}{16} (g-\vert x-a \vert), \quad &\text{if } \ \vert x-a \vert \leq g
	\end{cases}
\end{align*}
for $a \in \R$ and $g>0$. Following the lines of the proof of Proposition \ref{low_ptw_riskbound}, both $p_{1,n}$ and $p_{2,n}$ are contained in the class $\mathscr P_k$ for sufficiently large $k \geq k_0(\beta_*)$. Moreover, both $p_{1,n}$ and $p_{2,n}$ are constant on $B(t, g_{1,n}/2)$, so that
\begin{align*}
	p_{1,n \vert B(t,g_{1,n}/2)}, \, p_{2,n \vert B(t,g_{1,n}/2)} \in \mathcal H_{B(t, g_{1,n}/2)}(\infty,1/4). 
\end{align*}
The absolute distance of $p_{1,n}$ and $p_{2,n}$ in $t$ is given by
\begin{align*}
	\vert p_{1,n}(t) - p_{2,n}(t) \vert = \frac{1}{32} g_{1,n},
\end{align*}
whereas the Kullback-Leibler divergence between the associated product probability measures $\P_{1,n}^{\otimes n}$ and $\P_{2,n}^{\otimes n}$ is upper bounded by
\begin{align*}
	K(\P_{2,n}^{\otimes n},\P_{1,n}^{\otimes n}) &\leq n \, \int \frac{(p_{2,n}(x) - p_{1,n}(x))^2}{p_{1,n}(x)} \diff x \\
		&\leq 16 \, n \, \int (p_{2,n}(x) - p_{1,n}(x))^2 \diff x \\
		&= 32 \, n \, \int (q_{0,n}(x;g_{1,n}) - q_{0,n}(x, g_{1,n}/2))^2 \diff x \\
		&= 32 \, n \,  \Bigg( 2 \int_{g_{1,n}/2}^{g_{1,n}} \left(\frac{1}{16} (g_{1,n}- x) \right)^2 \diff x + \int_{-g_{1,n}/2}^{g_{1,n}/2} \Big( \frac{g_{1,n}}{32} \Big)^2 \diff x  \Bigg) \\	
		&= \frac{2}{3\cdot 32^2} + \frac{1}{32}.
\end{align*}
Together with Theorem 2.2 in \cite{tsybakov2009} the result follows.
\end{proof}

\begin{proof}[Proof of Theorem 3.12]
\noindent Recall the notation of Subsection 3.2, in particular the definitions of $\hat h_n^{loc}(t)$ in (3.12), of $q_n(\alpha)$ in (3.14), of $\beta_{n,p}(t)$ in (3.15), and of $\bar h_n(t)$ in (4.1). Furthermore, set $\tilde \gamma= \tilde \gamma(c_1)=\frac{1}{2} (c_1\log 2 - 1)$. To show that the confidence band is adaptive, note that according to Proposition 4.1 and Lemma 4.2 for any $\delta > 0$ there exists some $n_0(\delta)$, such that
\begin{align*}
	&\sup_{p \in \mathscr P_n} \P_p^{\chi_2} \left( \sup_{t \in [0,1]} \vert C_{n,\alpha}(t) \vert \cdot \left( \frac{\log \tn}{\tn} \right)^{\frac{-\beta_{n,p}(t)}{2\beta_{n,p}(t)+1}} \geq \sqrt{6} \cdot 2^{1-\frac{j_{\min}}{2}} q_n(\alpha) (\log \tn)^{\tilde \gamma} \right) \\
		&\quad = \sup_{p \in \mathscr P_n} \P_p^{\chi_2} \left(  \sup_{t \in [0,1]} \frac{\bar h_n(t)}{\hat h_n^{loc}(t)} \cdot 2^{-u_n} \geq 6 \right) \\
		&\quad =\sup_{p \in \mathscr P_n} \P_p^{\chi_2} \left(  \max_{k \in T_n} \, \sup_{t \in I_k} \, \frac{\bar h_n(t)}{\min \left\{ 2^{-\hat j_n((k-1) \delta_n)}, 2^{-\hat j_n(k \delta_n)} \right\}}  \geq 6 \right) \\
		&\quad \leq\sup_{p \in \mathscr P_n} \P_p^{\chi_2} \left(  \max_{k \in T_n}  \, \frac{\min \left\{ \bar h_n((k-1)\delta_n), \bar h_n(k\delta_n) \right\}}{\min \left\{ 2^{-\hat j_n((k-1) \delta_n)}, 2^{-\hat j_n(k \delta_n)} \right\}}  \geq 2 \right) \\
		&\quad \leq\sup_{p \in \mathscr P_n} \P_p^{\chi_2} \left(  \exists \, k \in T_n : \frac{\min \left\{ 2^{-\bar j_n((k-1) \delta_n)}, 2^{-\bar j_n(k \delta_n)} \right\}}{\min \left\{ 2^{-\hat j_n((k-1) \delta_n)}, 2^{-\hat j_n(k \delta_n)} \right\}}  \geq 1 \right) \\
		&\quad = \sup_{p \in \mathscr P_n} \left\{ 1 - \P_p^{\chi_2} \left(  \forall \, k \in T_n : \frac{\min \left\{ 2^{-\bar j_n((k-1) \delta_n)}, 2^{-\bar j_n(k \delta_n)} \right\}}{\min \left\{ 2^{-\hat j_n((k-1) \delta_n)}, 2^{-\hat j_n(k \delta_n)} \right\}}  < 1 \right) \right\} \\
		&\quad \leq \sup_{p \in \mathscr P_n} \left\{ 1-\P_p^{\chi_2} \left( \hat j_n(k\delta_n) < \bar j_n(k\delta_n) \text{ for all } k \in T_n \right) \right\} \\
		&\quad \leq \delta
\end{align*} 
for all $n \geq n_0(\delta)$.
\end{proof}

\subsection{Proofs of the results in Section \ref{sec:aux}} \label{subsec:prooflemma}

\begin{proof}[Proof of Proposition \ref{bw_band}]

We prove first that
\begin{align} \label{bw_large}
	\lim_{n \to \infty} \sup_{p \in \mathscr P_n} \P_p^{\chi_2} \left( \hat j_n(k\delta_n) > \bar j_n(k\delta_n)+1 \text{ for some } k \in T_n \right) = 0.
\end{align}
Note first that if $\hat j_n(k\delta_n) > \bar j_n(k\delta_n)+1$ for some $k \in T_n$, then $\bar{j}_n(k\delta_n)+1$ cannot be an admissible exponent according to the construction of the bandwidth selection scheme in \eqref{setA}, 
that is, $\bar{j}_n(k\delta_n) + 1 \notin \mathcal{A}_n(k\delta_n)$. By definition of $\mathcal{A}_n(k\delta_n)$ there exist exponents $m_{n,k}, m_{n,k}'  \in \mathcal{J}_n$ with $m_{n,k} > m_{n,k}' \geq \bar{j}_n(k\delta_n)+4$ such that
\begin{align*}
	\max_{s \in B\left(k\delta_n,\frac{7}{8}\cdot2^{-(\bar j_n(k\delta_n)+1)}\right) \cap \mathcal H_n} \vert \hat{p}_{n}^{(2)} (s,m_{n,k}) - \hat{p}_{n}^{(2)}(s,m_{n,k}') \vert > c_2 \sqrt{\frac{\log \tn}{\tn 2^{-m_{n,k}}}}.
\end{align*}
Consequently,
\begin{align*}
	&\P_p^{\chi_2} \left( \hat j_n(k\delta_n) > \bar j_n(k\delta_n)+1 \text{ for some } k \in T_n \right) \\
	& \leq \P_p^{\chi_2} \Bigg( \exists k \in T_n \text{ and } \exists \, m_{n,k}, m_{n,k}' \in \mathcal{J}_n \text{ with } m_{n,k} > m_{n,k}' \geq \bar{j}_n(k\delta_n)+4 \text{ such that } \\
	& \max_{s \in B\left(k\delta_n,\frac{7}{8}\cdot2^{-(\bar j_n(k\delta_n)+1)}\right) \cap \mathcal H_n} \vert \hat{p}_{n}^{(2)} (s,m_{n,k}) - \hat{p}_{n}^{(2)}(s,m_{n,k}') \vert > c_2 \sqrt{\frac{\log \tn}{\tn 2^{-m_{n,k}}}} \, \Bigg) \\
	& \leq \sum_{m \in \mathcal{J}_n} \sum_{m' \in \mathcal{J}_n} \P_p^{\chi_2} \Bigg( m > m' \geq \bar{j}_n(k\delta_n)+4 \text{ and } \\
	&\hspace{2.7cm} \max_{s \in B\left(k\delta_n,\frac{7}{8}\cdot2^{-(\bar j_n(k\delta_n)+1)}\right) \cap \mathcal H_n} \vert \hat{p}_{n}^{(2)} (s,m) - \hat{p}_{n}^{(2)}(s,m') \vert \\
	&\hspace{7cm} > c_2 \sqrt{\frac{\log \tn}{\tn 2^{-m}}} \text{ for some } k \in T_n \Bigg).
\end{align*}
We furthermore use the following decomposition into two stochastic terms and two bias terms
\begin{align*}
	&\max_{s \in B\left(k\delta_n,\frac{7}{8}\cdot2^{-(\bar j_n(k\delta_n)+1)}\right) \cap \mathcal H_n} \left\vert \hat{p}_{n}^{(2)} (s,m) - \hat{p}_{n}^{(2)}(s,m') \right\vert \\
	&\quad \leq \max_{s \in B\left(k\delta_n,\frac{7}{8}\cdot2^{-(\bar j_n(k\delta_n)+1)}\right) \cap \mathcal H_n} \left\vert \hat{p}_{n}^{(2)} (s,m) - \E_p^{\chi_2} \hat{p}_{n}^{(2)} (s,m) \right\vert \\
	&\hspace{2cm} + \max_{s \in B\left(k\delta_n,\frac{7}{8}\cdot2^{-(\bar j_n(k\delta_n)+1)}\right) \cap \mathcal H_n} \left\vert \hat{p}_{n}^{(2)}(s,m') - \E_p^{\chi_2} \hat{p}_{n}^{(2)}(s,m') \right\vert \\
	&\qquad + \sup_{s \in B\left(k\delta_n,\frac{7}{8}\cdot2^{-(\bar j_n(k\delta_n)+1)}\right)} \left\vert \E_p^{\chi_2} \hat{p}_{n}^{(2)} (s,m) - p(s) \right\vert \\
	&\qquad + \sup_{s \in B\left(k\delta_n,\frac{7}{8}\cdot2^{-(\bar j_n(k\delta_n)+1)}\right)} \left\vert \E_p^{\chi_2} \hat{p}_{n}^{(2)}(s,m') - p(s) \right\vert.
\end{align*}
In order to bound the two bias terms, note first that for any $m > m' \geq \bar j_n(k\delta_n)+4$ both
\begin{align*}
	 \frac{7}{8}\cdot2^{-(\bar j_n(k\delta_n)+1)} = 2^{-(\bar j_n(k\delta_n)+1)} - \frac{1}{8} \cdot 2^{-(\bar j_n(k\delta_n)+1)} \leq 2^{-(\bar j_n(k\delta_n)+1)} -  2^{-m}
\end{align*}
and
\begin{align*}
	 \frac{7}{8}\cdot2^{-(\bar j_n(k\delta_n)+1)} = 2^{-(\bar j_n(k\delta_n)+1)} - \frac{1}{8} \cdot 2^{-(\bar j_n(k\delta_n)+1)} \leq 2^{-(\bar j_n(k\delta_n)+1)} -  2^{-m'}.
\end{align*}
According to Assumption \ref{self_sim} and Lemma \ref{remark2},
\begin{align*}
	p_{\vert B \left( k\delta_n, 2^{-(\bar j_n(k\delta_n)+1)} \right)} \in \mathcal H_{\beta^*, B \left( k\delta_n, 2^{-(\bar j_n(k\delta_n)+1)} \right)} \left( \beta_p \left( B \left( k\delta_n, 2^{-\bar j_n(k\delta_n)} \right) \right), L^* \right),
\end{align*}
so that Lemma \ref{bias_up} yields,
\begin{align*}
	 &\sup_{s \in B\left(k\delta_n,\frac{7}{8}\cdot2^{-(\bar j_n(k\delta_n)+1)}\right)} \left\vert \E_p^{\chi_2} \hat{p}_{n}^{(2)} (s,m) - p(s) \right\vert \\
	 &\hspace{4cm} \leq  \sup_{s \in B(k\delta_n,2^{-(\bar j_n(k\delta_n)+1)} -  2^{-m})} \left\vert \E_p^{\chi_2} \hat{p}_{n}^{(2)} (s,m) - p(s) \right\vert \\
	 &\hspace{4cm} \leq b_2 2^{-m \beta_p\left(B\left(k\delta_n,2^{-\bar j_n(k\delta_n)}\right)\right)} \\
	  &\hspace{4cm} \leq b_2 2^{-m \beta_p\left(B\left(k\delta_n,\bar h_n(k\delta_n)\right)\right)} \\
	  &\hspace{4cm} \leq b_2 2^{-m \beta_{n,p}(k\delta_n)},
\end{align*}
with the bandwidth $\bar h_n(\cdot)$ as defined in \eqref{h_opt}, and analogously
\begin{align*}
	 &\sup_{s \in B\left(k\delta_n,\frac{7}{8}\cdot2^{-(\bar j_n(k\delta_n)+1)}\right)} \left\vert \E_p^{\chi_2} \hat{p}_{n}^{(2)} (s,m') - p(s) \right\vert \leq b_2 2^{-m' \beta_{n,p}(k\delta_n)}.
\end{align*}
Thus, the sum of the two bias terms is bounded from above by $2b_2 \bar h_n(k\delta_n)^{\beta_{n,p}(k\delta_n)}$, such that
\begin{align*}
	&\sqrt{\frac{\tn 2^{-m}}{\log \tn}} \Bigg( \sup_{s \in B\left(k\delta_n,\frac{7}{8}\cdot2^{-(\bar j_n(k\delta_n)+1)})\right)} \left\vert \E_p^{\chi_2} \hat{p}_{n}^{(2)} (s,m) - p(s) \right\vert \\
	&\hspace{3cm}+ \sup_{s \in B\left(k\delta_n,\frac{7}{8}\cdot2^{-(\bar j_n(k\delta_n)+1)})\right)} \left\vert \E_p^{\chi_2} \hat{p}_{n}^{(2)}(s,m') - p(s) \right\vert \Bigg) \\
	&\quad \leq \sqrt{\frac{\tn \bar h_n(k\delta_n)}{\log \tn}} \cdot 2b_2 \bar h_n(k\delta_n)^{\beta_{n,p}(k\delta_n)} \\
	&\quad \leq c_{21}°,
\end{align*}
where $c_{21}° = c_{21}°(\beta_*,L^*,\varepsilon) = 2b_2 \cdot 2^{-j_{\min}(2\beta_*+1)/2}$. Thus, it holds
\begin{align*}
	&\P_p^{\chi_2} \left( \hat j_n(k\delta_n) > \bar j_n(k\delta_n)+1 \text{ for some } k \in T_n \right) \\
	&\quad \leq \sum_{m \in \mathcal J_n} \sum_{m' \in \mathcal J_n} \Bigg\{ \P_p^{\chi_2} \Bigg( \max_{k \in T_n} \max_{s \in B\left(k\delta_n,\frac{7}{8}\cdot2^{-(\bar j_n(k\delta_n)+1)}\right) \cap \mathcal H_n} \left\vert \hat{p}_{n}^{(2)} (s,m) - \E_p^{\chi_2} \hat{p}_{n}^{(2)} (s,m) \right\vert \\
	&\hspace{9.5cm} > \frac{c_2-c_{21}°}{2} \sqrt{\frac{\log \tn}{\tn 2^{-m}}} \, \Bigg)  \\
	&\hspace{1.9cm}+ \P_p^{\chi_2} \Bigg( \max_{k \in T_n} \max_{s \in B\left(k\delta_n,\frac{7}{8}\cdot2^{-(\bar j_n(k\delta_n)+1)}\right) \cap \mathcal H_n} \left\vert \hat{p}_{n}^{(2)} (s,m') - \E_p^{\chi_2} \hat{p}_{n}^{(2)} (s,m') \right\vert \\
	&\hspace{8.7cm}> \frac{c_2-c_{21}°}{2} \sqrt{\frac{\log \tn}{\tn 2^{-m'}}} \, \Bigg) \Bigg\} \\
	&\quad \leq 2 \, \vert \mathcal J_n \vert^2 \cdot \P_p^{\chi_2} \left( \sup_{s \in \mathcal H_n} \max_{h \in \mathcal G_n} \sqrt{\frac{\tn h}{\log \tn}} \left\vert \hat{p}_{n}^{(2)} (s,h) - \E_p^{\chi_2} \hat{p}_{n}^{(2)} (s,h) \right\vert > \frac{c_2-c_{21}°}{2} \right).
\end{align*} 
Choose $c_2 = c_2(A, \nu, \beta_*, L^*,K,\varepsilon)$ sufficiently large such that $c_2 \geq c_{21}°+2 \eta_0$, where $\eta_0$ is given in Lemma \ref{lemma_deviation}. Then, Lemma \ref{lemma_deviation} and the logarithmic cardinality of $\mathcal J_n$ yield \eqref{bw_large}. In addition, we show that 
\begin{align} \label{bw_small}
	\lim_{n \to \infty} \sup_{p \in \mathscr P_n} \P_p^{\chi_2} \left( \hat{j}_n(k\delta_n) < k_n(k\delta_n) \text{ for some } k \in T_n \right) = 0.
\end{align}
For $t \in [0,1]$, due to the sequential definition of the set of admissible bandwidths $\mathcal A_n(t)$ in \eqref{setA}, if $\hat j_n(t) < j_{\max}$, then both $\hat j_n(t)$ and $\hat j_n(t)+1$ are contained in $\mathcal A_n(t)$. Note furthermore, that $k_n(t) < j_{\max}$ for any $t \in [0,1]$. Thus, if $\hat{j}_n(k\delta_n) < k_n(k\delta_n)$ for some $k \in T_n$, there exists some index $j < k_n(k\delta_n)+1$ with $j \in \mathcal A_n(k\delta_n)$ and satisfying \eqref{assumption1} and \eqref{admissible_density} for $u= 2^{-j}$ and $t = k\delta_n$.
In particular,
\begin{align*}
	\max_{s \in B\left(k\delta_n,\frac{7}{8}\cdot2^{-j}\right) \cap \mathcal H_n} \left\vert \hat{p}_{n}^{(2)} (s,j+3) - \hat{p}_{n}^{(2)}(s,\bar j_n(k\delta_n)) \right\vert \leq c_2 \sqrt{\frac{\log \tn}{\tn 2^{-\bar j_n(k\delta_n)}}}
\end{align*}
for sufficiently large $n \geq n_0(c_1)$, using that $\bar j_n(k\delta_n) \in \mathcal J_n$ for any $k \in T_n$. Consequently
\begin{align}\label{c3/2}
	&\P_p^{\chi_2} \left( \hat{j}_n(k\delta_n) < k_n(k\delta_n) \text{ for some } k \in T_n \right) \notag \\
	\begin{split} 
	& \leq \sum_{j \in \mathcal J_n} \P_p^{\chi_2} \Bigg( \exists \, k \in T_n : j < k_n(k\delta_n)+1 
	\text{ and } p_{\vert B(k\delta_n,2^{-j})} \in \mathcal H_{\beta^*,B(k\delta_n,2^{-j})}(\beta,L^*) \\
	&\hspace{2cm}\text{and} \sup_{s \in B(k\delta_n,2^{-j}-g)} \left\vert (K_g \ast p)(s) - p(s) \right\vert \geq \frac{g^{\beta}}{\log n} \text{ for all } g \in \mathcal G_{\infty} \text{ with } \\
	&\hspace{2cm}g \leq 2^{-(j+3)} \text{ and } \max_{s \in B\left(k\delta_n,\frac{7}{8}\cdot2^{-j}\right) \cap \mathcal H_n} \left\vert \hat{p}_{n}^{(2)} (s,j+3) - \hat{p}_{n}^{(2)}(s,\bar j_n(k\delta_n)) \right\vert \\
	 &\hspace{9.5cm}\leq c_2 \sqrt{\frac{\log \tn}{\tn 2^{-\bar j_n(k\delta_n)}}} \, \Bigg).
	\end{split}
\end{align}
The triangle inequality yields 
\begin{align*}
	&\max_{s \in B\left(k\delta_n,\frac{7}{8}\cdot2^{-j}\right) \cap \mathcal H_n} \left\vert \hat{p}_{n}^{(2)} (s,j+3) - \hat{p}_{n}^{(2)}(s,\bar j_n(k\delta_n)) \right\vert \\
	&\hspace{2cm} \geq \max_{s \in B\left(k\delta_n,\frac{7}{8}\cdot2^{-j}\right) \cap \mathcal H_n} \left\vert \E_p^{\chi_2} \hat{p}_{n}^{(2)} (s,j+3) - \E_p^{\chi_2} \hat{p}_{n}^{(2)}(s,\bar j_n(k\delta_n)) \right\vert \\
	&\hspace{3cm} - \max_{s \in B\left(k\delta_n,\frac{7}{8}\cdot2^{-j}\right) \cap \mathcal H_n} \left\vert \hat{p}_{n}^{(2)} (s,j+3) - \E_p^{\chi_2} \hat{p}_{n}^{(2)} (s,j+3) \right\vert \\
	&\hspace{3cm} - \max_{s \in B\left(k\delta_n,\frac{7}{8}\cdot2^{-j}\right) \cap \mathcal H_n} \left\vert \hat{p}_{n}^{(2)} (s,\bar j_n(k\delta_n)) - \E_p^{\chi_2} \hat{p}_{n}^{(2)} (s,\bar j_n(k\delta_n)) \right\vert.
\end{align*}
We further decompose
\begin{align*}
	&\max_{s \in B\left(k\delta_n,\frac{7}{8}\cdot2^{-j}\right) \cap \mathcal H_n} \left\vert \E_p^{\chi_2} \hat{p}_{n}^{(2)} (s,j+3) - \E_p^{\chi_2} \hat{p}_{n}^{(2)}(s,\bar j_n(k\delta_n)) \right\vert \\
	&\hspace{4cm} \geq \max_{s \in B\left(k\delta_n,\frac{7}{8}\cdot2^{-j}\right) \cap \mathcal H_n} \left\vert \E_p^{\chi_2} \hat{p}_{n}^{(2)} (s,j+3) - p(s) \right\vert \\
	&\hspace{5cm} \quad - \sup_{s \in B\left(k\delta_n,\frac{7}{8}\cdot2^{-j}\right)} \left\vert \E_p^{\chi_2} \hat{p}_{n}^{(2)}(s,\bar j_n(k\delta_n))  - p(s) \right\vert.
\end{align*}
As Assumption \ref{self_sim} is satisfied for $u = 2^{-j}$ and $t=k\delta_n$, together with Lemma~\ref{remark2} we both have
\begin{align} \label{bdg1}
	p_{\vert B\left(k\delta_n, 2^{-j}\right)} \in \mathcal H_{\beta^*,B\left(k\delta_n, 2^{-j} \right)}\Big(\beta_p\left(B\left(k\delta_n, 2^{-j}\right)\right), L^*\Big)
\end{align}
and
\begin{align} \label{bdg2}
	\sup_{s \in B \left( k\delta_n, 2^{-j} - g \right)} \left\vert (K_g \ast p)(s) - p(s) \right\vert \geq \frac{g^{\beta_p \left( B \left( k\delta_n, 2^{-j} \right) \right)}}{\log n}
\end{align}
for all $g \in \mathcal G_{\infty}$ with $g \leq 2^{-(j+3)}$. In particular, \eqref{bdg1} together with Lemma~\ref{bias_up} gives the upper bias bound
\begin{align*}
	\sup_{s \in B\left(k\delta_n,\frac{7}{8}\cdot2^{-j}\right)} \Big\vert \E_p^{\chi_2} \hat{p}_{n}^{(2)} (s,\bar j_n(k\delta_n)) - p(s) \Big\vert \leq b_2 \cdot 2^{-\bar j_n(k\delta_n) \beta_p \left( B \left( k\delta_n, 2^{-j} \right) \right)}
\end{align*}
for sufficiently large $n \geq n_0(c_1)$, whereas \eqref{bdg2} yields the bias lower bound
\begin{align} \label{assumption_conclusion}
	&\sup_{s \in B\left(k\delta_n,\frac{7}{8}\cdot2^{-j}\right)} \left\vert \E_p^{\chi_2} \hat{p}_{n}^{(2)} (s,j+3) - p(s) \right\vert \notag \\
	&\hspace{2cm} = \sup_{s \in B(k\delta_n,2^{-j} - 2^{-(j+3)})} \left\vert \E_p^{\chi_2} \hat{p}_{n}^{(2)} (s,j+3) - p(s) \right\vert \notag \\
	&\hspace{2cm}\geq \frac{2^{-(j+3)\beta_p \left( B \left( k\delta_n, 2^{-j} \right) \right)}}{\log n}.
\end{align}
To show that the above lower bound even holds for the maximum over the set  $B\left(k\delta_n,\frac{7}{8}\cdot2^{-j}\right) \cap \mathcal H_n$, note that for any point $k\delta_n - \frac{7}{8} 2^{-j} \leq \tilde t \leq k\delta_n + \frac{7}{8} 2^{-j}$ there exists some $t \in \mathcal H_n$ with $\vert t - \tilde t \vert \leq \delta_n$, and
\begin{align} \label{assumption_conclusion2}
	&\left\vert \E_p^{\chi_2} \hat{p}_{n}^{(2)} (t,j+3) - p( t) \right\vert \notag \\
	&\hspace{2cm} = \left\vert \int K(x) \Big\{ p( t+2^{-(j+3)}x)-p( t) \Big\} \diff x \right\vert \notag\\
		&\hspace{2cm}  \geq \left\vert \int K(x) \Big\{ p\left(\tilde t+2^{-(j+3)}x\right)-p\left(\tilde t \,\right) \Big\} \diff x \right\vert \notag\\*
		&\hspace{2cm} \qquad - \int \vert K(x) \vert \cdot \left\vert p(t+2^{-(j+3)}x) - p\left(\tilde t + 2^{-(j+3)}x\right) \right\vert \diff x \notag\\*
		&\hspace{2cm} \qquad - \int \vert K(x) \vert \cdot \vert p(t) - p\left(\tilde t  \right) \vert \diff x \notag\\
		&\hspace{2cm}  \geq \left\vert \int K(x) \Big\{ p\left(\tilde t+2^{-(j+3)}x\right)-p\left(\tilde t \right) \Big\} \diff x \right\vert - 2 \Vert K \Vert_1 L^* \cdot \vert t-\tilde t \vert^{\beta_*},
\end{align}
where
\begin{align*}
	\vert t-\tilde t \vert^{\beta_*} &\leq \delta_n^{\beta_*} \\
		&\leq 2^{-j_{\min}} \left( \frac{\log \tn}{\tn} \right)^{\frac{1}{2}} (\log \tn)^{-2} \\
		&\leq \frac{\bar h_n(k\delta_n)^{\beta_{n,p}(k\delta_n)}}{(\log \tn)^2} \\
		&\leq \frac{2^{-(\bar j_n(k\delta_n)-1) \beta_{n,p}(k\delta_n)}}{(\log \tn)^2} \\
		&\leq \frac{2^{-(j+3) \beta_{n,p}(k\delta_n)}}{(\log \tn)^2}
\end{align*}
for sufficiently large $n \geq n_0(c_1)$. For $n \geq n_0(c_1)$ and $j \in \mathcal J_n$ with $j < k_n(k\delta_n)+1$,
\begin{align*} 
	 2^{-j} > 2^{m_n-1} \cdot 2^{-\bar j_n(k\delta_n)} > \bar h_n(k\delta_n).
\end{align*}
Together with \eqref{bdg1}, this implies
\begin{align}\label{relation_tilde_bar}
	\beta_p(B(k\delta_n, 2^{-j})) \leq \beta_{n,p}(k\delta_n)
\end{align}
since otherwise $p$ would be $\beta$-H\"older smooth with $\beta > \beta_{n,p}(k\delta_n)$ on a ball $B(k\delta_n, r)$ with radius $r > \bar h_n(t)$, which would contradict the definition of $\beta_{n,p}(k\delta_n)$ together with Lemma \ref{ineq_holdernorm1}. This implies
\begin{align*}
	\vert t-\tilde t \vert^{\beta_*} \leq \frac{2^{-(j+3)\beta_p(B(k\delta_n, 2^{-j}))}}{(\log \tn)^2}.
\end{align*}
Together with inequalities \eqref{assumption_conclusion} and \eqref{assumption_conclusion2},
\begin{align*}
	&\max_{s \in B\left(k\delta_n,\frac{7}{8}\cdot 2^{-j}\right) \cap \mathcal H_n} \left\vert \E_p^{\chi_2} \hat{p}_{n}^{(2)} (s,j+3) - p(s) \right\vert \\
	&\qquad \qquad \quad \geq \sup_{s \in B\left(k\delta_n,\frac{7}{8}\cdot 2^{-j}\right)} \left\vert \E_p^{\chi_2} \hat{p}_{n}^{(2)} (s,j+3) - p(s) \right\vert - 2 \Vert K \Vert_1 L^* \frac{2^{-(j+3)\beta_p(B(k\delta_n, 2^{-j}))}}{(\log \tn)^2} \\
	&\qquad \qquad \quad\geq \frac{1}{2} \cdot\frac{2^{-(j+3)\beta_p(B(k\delta_n, 2^{-j}))}}{\log \tn}
\end{align*}
for sufficiently large $n \geq n_0(L^*,K,c_1)$. Altogether, we get for $j < k_n(k\delta_n)+1$,
\begin{align*}
	&\sqrt{\frac{\tn 2^{-\bar j_n(k\delta_n)}}{\log \tn}} \max_{s \in B\left(k\delta_n,\frac{7}{8}\cdot 2^{-j}\right) \cap \mathcal H_n} \left\vert \E_p^{\chi_2} \hat{p}_{n}^{(2)} (s,j+3) - \E_p^{\chi_2} \hat{p}_{n}^{(2)}(s,\bar j_n(k\delta_n)) \right\vert \\*
	& \geq \sqrt{\frac{\tn 2^{-\bar j_n(k\delta_n)}}{\log \tn}} \Bigg( \max_{s \in B\left(k\delta_n,\frac{7}{8}\cdot 2^{-j}\right) \cap \mathcal H_n} \left\vert \E_p^{\chi_2} \hat{p}_{n}^{(2)} (s,j+3) - p(s) \right\vert \\
	&\hspace{4cm} -  \max_{s \in B\left(k\delta_n,\frac{7}{8}\cdot 2^{-j}\right) \cap \mathcal H_n} \left\vert \E_p^{\chi_2} \hat{p}_{n}^{(2)}(s,\bar j_n(k\delta_n)) -p(s) \right\vert \Bigg)\\
	& \geq  \sqrt{\frac{\tn \bar h_n(k\delta_n)}{2\log \tn}} \left( \frac{1}{2} \cdot\frac{2^{-(j+3)\beta_p(B(k\delta_n, 2^{-j}))}}{\log \tn} - b_2 \cdot 2^{-\bar j_n(k\delta_n) \beta_p \left( B \left( k\delta_n, 2^{-j} \right) \right)} \right) \\
	& \geq \sqrt{\frac{\tn \bar h_n(k\delta_n)}{2\log \tn}} 2^{-(\bar j_n(k\delta_n)-1)\beta_p(B(k\delta_n, 2^{-j}))} \left( \frac{1}{2} \cdot \frac{2^{(\bar j_n(k\delta_n)-j-4) \beta_p(B(k\delta_n, 2^{-j}))}}{\log \tn} - b_2 2^{-\beta_*} \right) \\
	& >  \sqrt{\frac{\tn \bar h_n(k\delta_n)}{2\log \tn}} 2^{-(\bar j_n(k\delta_n)-1) \beta_p(B(k\delta_n, 2^{-j}))} \left( \frac{2^{(m_n-5)\beta_*}}{2 \, \log \tn} - b_2 2^{-\beta_*} \right).
\end{align*}
We now show that for $j \in \mathcal J_n$ with $j < k_n(k\delta_n)+1$, we have that
\begin{align} \label{betap_finite}
	\beta_p(B(k\delta_n,2^{-j})) \leq \beta^*.
\end{align}
According to \eqref{relation_tilde_bar}, it remains to show that $\beta_{n,p}(k\delta_n) \leq \beta^*$. If $\beta_{n,p}(k\delta_n) = \infty$, then $\bar j_n(k\delta_n) = j_{\min}$. Since furthermore $j \in \mathcal J_n$ and therefore $j \geq j_{\min}$, this immediately contradicts $j < k_n(k\delta_n)+1$. That is, $j < k_n(k\delta_n)+1$ implies that $\beta_{n,p}(k\delta_n) < \infty$, which in turn implies $\beta_{n,p}(k\delta_n) \leq \beta^*$ according to Remark \ref{stern}. Due to \eqref{c22} and \eqref{betap_finite}, the last expression is again lower bounded by
\begin{align*}
3 c_2 \sqrt{\frac{\tn \bar h_n(k\delta_n)}{\log \tn}} \bar h_n(k\delta_n)^{\beta_p(B(k\delta_n, 2^{-j}))} 2^{j_{\min} \frac{2\beta_p(B(k\delta_n, 2^{-j}))+1}{2}}
\end{align*}
for sufficiently large $n \geq n_0(L^*,K,\beta_*,\beta^*,c_1,c_2)$. Recalling \eqref{relation_tilde_bar}, we obtain
\begin{align*}
	&\sqrt{\frac{\tn 2^{-\bar j_n(k\delta_n)}}{\log \tn}} \max_{s \in B\left(k\delta_n,\frac{7}{8}\cdot 2^{-j}\right) \cap \mathcal H_n} \left\vert \E_p^{\chi_2} \hat{p}_{n}^{(2)} (s,j+3) - \E_p^{\chi_2} \hat{p}_{n}^{(2)}(s,\bar j_n(k\delta_n)) \right\vert \\
	&\qquad \geq 3 c_2 \sqrt{\frac{\tn \bar h_n(k\delta_n)}{\log \tn}} \bar h_n(k\delta_n)^{\beta_{n,p}(k\delta_n)} 2^{j_{\min} \frac{2\beta_p(B(k\delta_n, 2^{-j}))+1}{2}}\\
	&\qquad = 3 c_2.
\end{align*}
Thus, by the above consideration and \eqref{c3/2},
\begin{align*}
	&\P_p^{\chi_2} \left( \hat{j}_n(k\delta_n) < k_n(k\delta_n) \text{ for some } k \in T_n \right) \leq \sum_{j \in \mathcal J_n} \left( P_{j,1} + P_{j,2} \right)
\end{align*}
for sufficiently large $n \geq n_0(L^*,K,\beta_*,\beta^*,c_1,c_2)$, with
\begin{align*}
	P_{j,1} &= \P_p^{\chi_2} \Bigg( \exists k \in T_n : j < k_n(k\delta_n)+1 \text{ and } \sqrt{\frac{\tn 2^{-(j+3)}}{\log \tn}}  \\
	&\hspace{1.5cm} \cdot \max_{s \in B\left(k\delta_n,\frac{7}{8} \cdot 2^{-j}\right) \cap \mathcal H_n} \left\vert \hat{p}_{n}^{(2)} (s,j+3) - \E_p^{\chi_2} \hat{p}_{n}^{(2)} (s,j+3) \right\vert \geq c_2 \Bigg) \\
	P_{j,2} &= \P_p^{\chi_2} \Bigg( \exists k \in T_n : j < k_n(k\delta_n)+1 \text{ and } \sqrt{\frac{\tn 2^{-\bar j_n(k\delta_n)}}{\log \tn}} \\
	&\hspace{0.2cm} \cdot \max_{s \in B\left(k\delta_n,\frac{7}{8} \cdot 2^{-j}\right) \cap \mathcal H_n} \left\vert \hat{p}_{n}^{(2)} (s,\bar j_n(k\delta_n)) - \E_p^{\chi_2} \hat{p}_{n}^{(2)} (s,\bar j_n(k\delta_n)) \right\vert \geq c_2 \Bigg).
\end{align*}
Both $P_{j,1}$ and $P_{j,2}$ are bounded by
\begin{align*}
	P_{j,i} \leq \P_p^{\chi_2} \left( \sup_{s \in \mathcal H_n} \max_{h \in \mathcal G_n} \sqrt{\frac{\tn h}{\log \tn}} \left\vert \hat p_{n}^{(2)}(s,h) - \E_p^{\chi_2} \hat p_{n}^{(2)}(s,h) \right\vert \geq c_2 \right), \quad i=1,2.
\end{align*}
For sufficiently large $c_2 \geq \eta_0$, Lemma \ref{lemma_deviation} and the logarithmic cardinality of $\mathcal J_n$ yield \eqref{bw_small}.
\end{proof}

\begin{proof}[Proof of Lemma \ref{grid_opt}] 
We prove both inequalities separately.

\paragraph{Part (i)}
First, we show that the density $p$ cannot be substantially unsmoother at $z \in (s,t)$ than at the boundary points $s$ and $t$. Precisely, we shall prove that $\min \{ \bar h_n(s), \bar h_n(t) \} \leq 2 \bar h_n(z)$. In case 
$$\beta_{n,p}(s)= \beta_{n,p}(t)= \infty,$$
that is $\bar h_n(s) = \bar h_n(t) = 2^{-j_{\min}}$, we immediately obtain $\bar h_n(z) \geq \frac{1}{2} 2^{-j_{\min}}$ since
\begin{align*}
	B\left(z,\frac{1}{2}2^{-j_{\min}}\right) \subset B(s,\bar h_n(s)) \cap B(t,\bar h_n(t)).
\end{align*}
Hence, we subsequently assume that 
$$\min \{ \beta_{n,p}(s),\beta_{n,p}(t) \}  < \infty. $$
Note furthermore that
\begin{align} \label{minbeta}
	\min \left\{\bar h_n(s), \bar h_n(t) \right\} = h_{\min \{ \beta_{n,p}(s), \beta_{n,p}(t) \},n}.
\end{align}
In a first step, we subsequently conclude that 
\begin{align}\label{fall1}
	z + \frac{1}{2} h_{\min \{ \beta_{n,p}(s), \beta_{n,p}(t) \},n} < s + \bar h_n(s)
\end{align}
or 
\begin{align}\label{fall2}
	z - \frac{1}{2} h_{\min \{ \beta_{n,p}(s), \beta_{n,p}(t) \},n} > t - \bar h_n(t).
\end{align}
Note first that $\vert s-t \vert < h_{\beta,n}$ for all $\beta \geq \beta_*$ by condition \eqref{stclose}. Assume now that \eqref{fall1} does not hold. Then, inequality \eqref{fall2} directly follows as
\begin{align*}
	z - \frac{1}{2} \min\{ \bar h_n(s), \bar h_n(t) \} &= z + \frac{1}{2} \min\{ \bar h_n(s), \bar h_n(t) \} - \min\{ \bar h_n(s), \bar h_n(t) \} \\
		&\geq s +  \bar h_n(s) - \min\{ \bar h_n(s), \bar h_n(t) \} \\
		&\geq t - (t-s) \\
		&> t - \bar h_n(t).
\end{align*}
Vice versa, if \eqref{fall2} does not hold, then a similar calculation as above shows that \eqref{fall1} is true. Subsequently, we assume without loss of generality that \eqref{fall1} holds. That is,
\begin{align} \label{ball_contained}
	s- \bar h_n(s) &< z - \frac{1}{2} h_{\min \{ \beta_{n,p}(s), \beta_{n,p}(t) \},n} \notag \\
	 &< z + \frac{1}{2} h_{\min \{ \beta_{n,p}(s), \beta_{n,p}(t) \},n} \\
	 &< s + \bar h_n(s). \notag
\end{align}
There exists some $\tilde \beta > 0$ with
\begin{align} \label{einschachtelung_min}
	h_{\tilde \beta,n} = \frac{1}{2}Ê\min \{Ê\bar h_n(t), \bar h_n(s) \}.
\end{align}
for sufficiently large $n \geq n_0(\beta_*)$. Equation \eqref{einschachtelung_min} implies that
\begin{align} \label{tildebeta_klein}
	\tilde \beta < \min\{ \beta_{n,p}(s), \beta_{n,p}(t) \} \leq \beta_{n,p}(s).
\end{align}
Finally, we verify that
\begin{align} \label{betazgross}
	\beta_{n,p}(z) \geq \tilde \beta.
\end{align}
Using Lemma \ref{ineq_holdernorm1} as well as \eqref{ball_contained}, \eqref{einschachtelung_min}, and \eqref{tildebeta_klein} we obtain
\begin{align*}
	&\Vert p \Vert_{\tilde \beta, \beta^*,B(z,h_{\tilde \beta,n})} \\
	&\quad = \sum_{k=0}^{\lfloor \tilde \beta \wedge \beta^* \rfloor} \Vert p^{(k)} \Vert_{B\left(z,\frac{1}{2} \min\{\bar h_n(t),\bar h_n(s) \}\right)} \\
	&\qquad + \sup_{\substack{x,y \, \in \, B\left(z,\frac{1}{2} \min\{\bar h_n(t),\bar h_n(s) \}\right) \\ x \neq y}} \frac{\vert p^{(\lfloor \tilde \beta \wedge \beta^* \rfloor)}(x) - p^{(\lfloor \tilde \beta \wedge \beta^*\rfloor)}(y) \vert}{\vert x-y \vert^{\tilde \beta - \lfloor \tilde \beta \wedge \beta^* \rfloor}} \\
	&\quad\leq L^*.
\end{align*}
Consequently, we conclude \eqref{betazgross}. With \eqref{minbeta} and \eqref{einschachtelung_min}, this in turn implies
\begin{align*}
	\min \left\{ \bar h_n(s), \bar h_n(t) \right\} = 2 h_{\tilde \beta,n} \leq 2 h_{\beta_{n,p}(z),n} = 2 \bar h_n(z).
\end{align*}

\paragraph{Part (ii)}
Now, we show that the density $p$ cannot be substantially smoother at $z \in (s,t)$ than at the boundary points $s$ and $t$. Without loss of generality, let $\beta_{n,p}(t) \leq \beta_{n,p}(s)$. We prove the result by contradiction: assume that
\begin{align} \label{assumption_contra}
	\min \left\{ \bar h_n(s), \bar h_n(t) \right\} < \frac{8}{17} \cdot \bar h_n(z).
\end{align}
Since $t-z \leq h_{\beta,n}/8$ for all $\beta \geq \beta_*$ by condition \eqref{stclose}, so that in particular $t-z \leq \bar h_n(t)/8$, we obtain together with \eqref{assumption_contra} that
\begin{align} \label{betastrich}
	\frac{1}{2}\left(z-t+\bar h_n(z)\right) > \frac{1}{2} \left( -\frac{1}{8} \bar h_n(t) + \frac{17}{8} \bar h_n(t) \right) = \bar h_n(t)  > 0.
\end{align}
Because furthermore $\frac{1}{2}(z-t+\bar h_n(z))<1$, there exists some $\beta' = \beta'(n)>0$ with
\begin{align*}
	h_{\beta',n} = \frac{1}{2} \left( z-t+\bar h_n(z) \right).
\end{align*}
This equation in particular implies that $h_{\beta',n} < \frac{1}{2} \bar h_n(z)$ and thus $\beta' < \beta_{n,p}(z)$. Since furthermore $t-z < \bar h_n(z)$ by condition \eqref{stclose} and therefore also
\begin{align*}
	z-\bar h_n(z) < t - h_{\beta',n} < t + h_{\beta',n} < z + \bar h_n(z),
\end{align*}
we immediately obtain
\begin{align*}
	&\Vert p \Vert_{\beta',\beta^*,B(t,h_{\beta',n})}  \leq L^*,
\end{align*}
so that
\begin{align*}
	\beta_{n,p}(t) \geq \beta'.
\end{align*}
This contradicts inequality \eqref{betastrich}.
\end{proof}

\begin{proof}[Proof of Lemma \ref{lemma_deviation}]
Without loss of generality, we prove the inequality for the estimator $\hat p_{n}^{(1)}(\cdot,h)$ based on $\chi_1$. Note first, that
\begin{align*}
	 \sup_{s \in \mathcal H_n} \sup_{h \in \mathcal G_n} \sqrt{\frac{\tn h}{\log \tn}} \left\vert \hat{p}_{n}^{(1)} (s,h) - \E_p^{\chi_1} \hat{p}_{n}^{(1)}(s,h) \right\vert = \sup_{f \in \mathscr E_n} \left\vert \sum_{i=1}^{\tn} \left( f(X_i) - \E_p f(X_i) \right) \right\vert
\end{align*}
with
\begin{align*}
	\mathscr E_n = \left\{ f_{n,s,h}(\cdot) = (\tn h \log \tn)^{-\frac{1}{2}} K \left( \frac{\cdot - s}{h} \right) : s \in \mathcal H_n, \, h \in \mathcal G_n \right\}.
\end{align*}
Observe first that
\begin{align*}
	\sup_{p \in \mathscr P_n} \Var_p(f_{n,s,h}(X_1)) &\leq \sup_{p \in \mathscr P_n} \E_p f_{n,s,h}(X_1)^2 \\
		&= \sup_{p \in \mathscr P_n} \frac{1}{\tn h \log \tn} \int K \left( \frac{x - s}{h} \right)^2 p(x) \diff x \\
		&\leq \frac{L^* \Vert K \Vert_2^2}{\tn \log \tn} \\
		&=: \sigma_n^2
\end{align*}
uniformly over all $f _{n,s,h} \in \mathscr E_n$, and
\begin{align*}
	\sup_{s \in \mathcal H_n} \max_{h \in \mathcal G_n} \Vert f_{n,s,h} \Vert_{\sup} &\leq \max_{h \in \mathcal G_n} \frac{\Vert K \Vert_{\sup}}{\sqrt{\tn h \log \tn}} \\
	&= \Vert K \Vert_{\sup} (\log \tn)^{-\frac{\kappa_2+1}{2}} \\
	&\leq \frac{\Vert K \Vert_{\sup}}{(\log \tn)^{3/2}} =: U_n,
\end{align*}
where the last inequality holds true because by definition of $\kappa_2 \geq 2$ in \eqref{c22}. In particular $\sigma_n \leq U_n$ for sufficiently large $n \geq n_0(L^*, K)$. Since $(\tn h \log \tn)^{-1/2} \leq 1$ for all $h \in \mathcal G_n$ and for all $n \geq n_0$, the class $\mathscr E_n$ satisfies the VC property
\begin{align*}
	\limsup_{n \to \infty}Ê\sup_{Q} N\left(\mathcal E_n, \Vert \cdot \Vert_{L^2(Q)}, \varepsilon \Vert K \Vert_{\sup} \right)  \leq  \left( \frac{A''}{\varepsilon} \right)^{\nu''}
\end{align*}
for some VC characteristics $A''=A''(A,K)$ and $\nu'=\nu+1$, by the same arguments as in \eqref{FnVC}. According to Proposition 2.2 in \cite{gine_guillou2001}, there exist constants $c_{22}° = c_{22}°(A'',\nu'')$ and $c_5° = c_5°(A'',\nu'')$, such that
\begin{align} \label{gineguillou}
	 &\P_p^{\chi_1} \left( \sup_{s \in \mathcal H_n} \max_{h \in \mathcal G_n} \sqrt{\frac{\tn h}{\log \tn}} \left\vert \hat{p}_{n}^{(1)} (s,h) - \E_p^{\chi_1} \hat{p}_{n}^{(1)}(s,h) \right\vert > \eta \right) \notag \\
	 &\quad \leq c_5° \exp \left( - \frac{\eta}{c_5° U_n} \log \left( 1 + \frac{\eta U_n}{c_5° \left( \sqrt{\tn \sigma_n^2} + U_n \sqrt{\log(A'' U_n/\sigma_n)} \right)^2} \right) \right) \\
	 &\quad \leq c_5° \exp \left( - \frac{\eta}{c_5° U_n} \log \left( 1 + c_{23}° \eta U_n \log \tn \right) \right) \notag
\end{align}
uniformly over all $p \in \mathscr P_n$, for all $n \geq n_0(A'',K,L^*)$ with $c_{23}° = c_{23}°(A'',\nu'',L^*,K)$, whenever
\begin{align} \label{gineguillou_condition}
	\eta \geq c_{22}° \left( U_n \log \left( \frac{A''U_n}{\sigma_n} \right) + \sqrt{\tn \sigma_n^2} \sqrt{\log \left( \frac{A''U_n}{\sigma_n} \right)} \ \right).
\end{align}
Since the right hand side in \eqref{gineguillou_condition} is bounded from above by some positive constant $\eta_0 = \eta_0(A'', \nu'', L^*, K)$ for sufficiently large $n \geq n_0(A'', \nu'', L^*, K)$, inequality \eqref{gineguillou} holds in particular for all $n \geq n_0(A'',\nu,K,L^*)$ and for all $\eta \geq \eta_0$. Finally, using the inequality $\log (1+x) \geq \frac{x}{2}$ for $0 \leq x \leq 2$ (Lemma \ref{A2}), we obtain for all $\eta \geq \eta_0$
\begin{align*}
	 &\P_p^{\chi_1} \left( \sup_{s \in \mathcal H_n} \max_{h \in \mathcal G_n} \sqrt{\frac{\tn h}{\log \tn}} \left\vert \hat{p}_{n}^{(1)} (s,h) - \E_p^{\chi_1} \hat{p}_{n}^{(1)}(s,h) \right\vert > \eta \right) \\
	 &\hspace{3cm} \leq c_5° \exp \left( - c_{24}° \eta (\log \tn)^{3/2} \log \left( 1 + c_{25}° \frac{\eta_0}{\sqrt{\log \tn}} \right) \right) \\
	 &\hspace{3cm} \leq c_5° \exp \left( - \frac{1}{2} c_{24}° c_{25}° \eta_0 \eta \log \tn \right)
\end{align*}
uniformly over all $p \in \mathscr P_n$, for all $n \geq n_0(A'',\nu'',K,L^*)$ and positive constants $c_{24}°=c_{24}°(A'',\nu'',K)$ and $c_{25}°=c_{25}°(A'',\nu'',L^*,K)$, which do not depend on $n$ or $\eta$.

\end{proof}

\begin{proof}[Proof of Lemma \ref{bias_up}]
Let $t \in \R$, $g,h > 0$, and
\begin{align*}
	p_{\vert B(t,g+h)} \in \mathcal H_{\beta^*,B(t,g+h)}(\beta,L).
\end{align*}
The three cases $\beta \leq 1$, $1<\beta<\infty$, and $\beta=\infty$ are analyzed separately. In case $\beta \leq 1$, we obtain 
\begin{align*}
	\sup_{s \in B(t,g)} \left\vert (K_h \ast p)(s) - p(s) \right\vert \leq \int \vert K(x) \vert \sup_{s \in B(t,g)} \left\vert p(s + h x) - p(s) \right\vert \diff x,
\end{align*}
where 
\begin{align*}
	&\sup_{s \in B(t,g)} \left\vert p(s + h x) - p(s) \right\vert \leq h^{\beta} \cdot \sup_{\substack{s,s' \in B(t,g+h) \\ s \neq s'}} \frac{\vert p(s') - p(s) \vert}{\vert s' - s \vert^{\beta}} \leq L h^{\beta}.
\end{align*}
In case $1<\beta < \infty$, we use the Peano form for the remainder of the Taylor polynomial approximation. Note that $\beta^* \geq 2$ because $K$ is symmetric by assumption, and $K$ is a kernel of order $\lfloor \beta^* \rfloor = \beta^*-1$ in general, such that
\begin{align} \label{bias_beta_large}
	&\sup_{s \in B(t,g)} \left\vert (K_h \ast p)(s) - p(s) \right\vert \notag \\
		&\quad = \sup_{s \in B(t,g)} \left\vert \int K(x) \Big\{ p(s + h x) - p(s) \Big\} \diff x \right\vert\notag \\
		&\quad = \sup_{s \in B(t,g)} \left\vert \int K(x) \left\{ p(s + h x) - P^p_{s,\lfloor \beta \wedge \beta^* \rfloor}(s+hx) + \sum_{k=1}^{\lfloor \beta \wedge \beta^* \rfloor} \frac{p^{(k)}(s)}{k!} \cdot (hx)^k \right\} \diff x \, \right\vert\notag \\
		&\quad \leq \int \vert K(x) \vert \sup_{s \in B(t,g)} \left\vert p(s + h x) - P^p_{s,\lfloor \beta \wedge \beta^* \rfloor}(s+hx) \right\vert \diff x \notag\\
		&\quad \leq \int \vert K(x) \vert \sup_{s \in B(t,g)} \sup_{s' \in B(s,h)} \Bigg\vert \frac{p^{(\lfloor \beta \wedge \beta^* \rfloor)}(s') - p^{(\lfloor \beta \wedge \beta^* \rfloor)}(s)}{\lfloor \beta \wedge \beta^* \rfloor !} (hx)^{\lfloor \beta \wedge \beta^* \rfloor} \Bigg\vert \diff x\notag \\
		&\quad \leq \frac{h^{\lfloor\beta \wedge \beta^* \rfloor} h^{\beta - \lfloor \beta \wedge \beta^* \rfloor}}{\lfloor \beta \wedge \beta^* \rfloor !} \cdot \int \vert K(x) \vert \sup_{s \in B(t,g)} \sup_{\substack{s' \in B(s,h) \\Ês' \neq s}} \frac{\left\vert p^{(\lfloor \beta \wedge \beta^* \rfloor)}(s') - p^{(\lfloor \beta \wedge \beta^* \rfloor)}(s) \right\vert}{\vert s - s' \vert^{\beta - \lfloor \beta \wedge \beta^* \rfloor}}  \diff x \notag \\
		&\quad \leq L \Vert K \Vert_1 h^{\beta}.
\end{align}
In case $\beta = \infty$, the density $p$ satisfies $p_{\vert B(t,g+h)} \in \mathcal H_{\beta^*,B(t,g+h)}(\beta,L^*)$ for all $\beta > 0$. That is, the upper bound \eqref{bias_beta_large} on the bias holds for any $\beta>0$, implying that 
\begin{align*}
	\sup_{s \in B(t,g)} \left\vert (K_h \ast p)(s) - p(s) \right\vert = 0.
\end{align*}
This completes the proof.
\end{proof}

\begin{proof}[Proof of Lemma \ref{bias_zygmund}]
Note that by symmetry of $K$
\begin{align*}
	(K_h \ast p)(s) - p(s) = \frac{1}{2} \int_{-1}^1 K(x) \Big( p(s+hx)+p(s-hx)-2p(s) \Big) \diff x.
\end{align*}
The upper bound can thus be deduced exactly as in the proof of Lemma \ref{bias_up}. 
\end{proof}

\begin{appendix}

\section{Auxiliary results} \label{AuxResults}

\begin{lemma} \label{tildeWmoments}
For $z \in [0,1]$, the second moments of $\tilde W_{k,l}(z), k,l \in T_n$ as defined in \eqref{tildeW} are bounded by
\begin{align*}
	\E_W \tilde W_{k,l}(z)^2 \leq 4.
\end{align*}
\end{lemma}

\begin{proof}[Proof of Lemma \ref{tildeWmoments}]
As $\tilde W_{k,l}(\cdot) = - \tilde W_{l,k}(\cdot)$, we assume $k \leq l$ without loss of generality. For any $k,l \in T_n$
\begin{align*}
	\E_W \tilde W_{k,l}(z)^2 = \sum_{i=1}^{10} E_i
\end{align*}
with
\begin{align*}
	E_1 &= \frac{1}{\hat h_{n,k}^{loc}  } \E_W W(k\delta_n-z\hat h_{n,k}^{loc})^2 \\
	&= \frac{k\delta_n-z\hat h_{n,k}^{loc}}{\hat h_{n,k}^{loc}  } \\
	E_2 &= -\frac{2}{\hat h_{n,k}^{loc}  } \E_W W(k\delta_n-z\hat h_{n,k}^{loc}) W(k\delta_n+z\hat h_{n,k}^{loc}) \\
	&= -2\frac{k\delta_n-z\hat h_{n,k}^{loc}}{\hat h_{n,k}^{loc}  } \\
	E_3 &= \frac{2}{\sqrt{\hat h_{n,k}^{loc} \hat h_{n,l}^{loc}}} \E_W W(k\delta_n-z\hat h_{n,k}^{loc}) W(l\delta_n+z\hat h_{n,l}^{loc}) \\
	&= 2\frac{k\delta_n-z\hat h_{n,k}^{loc}}{\sqrt{\hat h_{n,k}^{loc} \hat h_{n,l}^{loc}}  } \\
	E_4 &= -\frac{2}{\sqrt{\hat h_{n,k}^{loc} \hat h_{n,l}^{loc}}  } \E_W W(k\delta_n-z\hat h_{n,k}^{loc}) W(l\delta_n-z\hat h_{n,l}^{loc}) \\
	&= -2\frac{\min \{ k\delta_n-z\hat h_{n,k}^{loc}, \ l \delta_n - z \hat h_{n,l}^{loc} \}}{\sqrt{\hat h_{n,k}^{loc} \hat h_{n,l}^{loc}}  } \\
	E_5 &= \frac{1}{\hat h_{n,k}^{loc}  } \E_W W(k\delta_n+z\hat h_{n,k}^{loc})^2 \\*
	&= \frac{k\delta_n+z\hat h_{n,k}^{loc}}{\hat h_{n,k}^{loc}  } \\
	E_6 &= -\frac{2}{\sqrt{\hat h_{n,k}^{loc} \hat h_{n,l}^{loc}}  } \E_W W(k\delta_n+z\hat h_{n,k}^{loc}) W(l\delta_n+z\hat h_{n,l}^{loc}) \\
	&= -2\frac{\min \{ k\delta_n+z\hat h_{n,k}^{loc}, \ l\delta_n+z\hat h_{n,l}^{loc} \}}{\sqrt{\hat h_{n,k}^{loc} \hat h_{n,l}^{loc}}  } \\
	E_7 &= \frac{2}{\sqrt{\hat h_{n,k}^{loc} \hat h_{n,l}^{loc}}  } \E_W W(k\delta_n+z\hat h_{n,k}^{loc}) W(l\delta_n-z\hat h_{n,l}^{loc}) \\
	&= 2\frac{\min \{k\delta_n+z\hat h_{n,k}^{loc}, \l\delta_n-z\hat h_{n,l}^{loc} \}}{\sqrt{\hat h_{n,k}^{loc} \hat h_{n,l}^{loc}}  } \\
	E_8 &= \frac{1}{\hat h_{n,l}^{loc}  } \E_W W(l\delta_n+z\hat h_{n,l}^{loc})^2 \\
	&= \frac{l\delta_n+z\hat h_{n,l}^{loc}}{\hat h_{n,l}^{loc}  } \\
	E_9 &= -\frac{2}{\hat h_{n,l}^{loc}  } \E_W W(l\delta_n+z\hat h_{n,l}^{loc}) W(l\delta_n-z\hat h_{n,l}^{loc}) \\
	&= -2\frac{l\delta_n-z\hat h_{n,l}^{loc}}{\hat h_{n,l}^{loc}  } \\
	E_{10} &= \frac{1}{\hat h_{n,l}^{loc}  } \E_W W(l\delta_n-z\hat h_{n,l}^{loc})^2 \\
	&= \frac{l\delta_n-z\hat h_{n,l}^{loc}}{\hat h_{n,l}^{loc}  }.
\end{align*}
Altogether,
\begin{align*}
	\E_W \tilde W_{k,l}(z)^2 &= 4z + \frac{2}{\sqrt{\hat h_{n,k}^{loc} \hat h_{n,l}^{loc}}} \Bigg( k\delta_n - z \hat h_{n,k}^{loc} - \min \left\{ k\delta_n-z\hat h_{n,k}^{loc}, l\delta_n-z\hat h_{n,l}^{loc} \right\} \\
	&\quad - \min \left\{ k\delta_n + z \hat h_{n,k}^{loc}, l\delta_n + z \hat h_{n,l}^{loc} \right\} + \min \left\{ k\delta_n + z\hat h_{n,k}^{loc}, l\delta_n -z\hat h_{n,l}^{loc} \right\} \Bigg).
\end{align*}
We distinguish between the two cases
\begin{align*}
	(i) \ k\delta_n -z\hat h_{n,k}^{loc} \leq l\delta_n-z\hat h_{n,l}^{loc} \qquad \text{and} \qquad (ii) \ k\delta_n -z\hat h_{n,k}^{loc} > l\delta_n-z\hat h_{n,l}^{loc}.
\end{align*}
In case $(i)$, we obtain
\begin{align*}
	\E_W \tilde W_{k,l}(z)^2 &= 4z + \frac{2}{\sqrt{\hat h_{n,k}^{loc} \hat h_{n,l}^{loc}}} \Bigg(  \min \left\{ k\delta_n + z\hat h_{n,k}^{loc}, l\delta_n -z\hat h_{n,l}^{loc} \right\} \\*
	&\hspace{4cm}- \min \left\{ k\delta_n + z \hat h_{n,k}^{loc}, l\delta_n + z \hat h_{n,l}^{loc} \right\}\Bigg) \\*
	&\leq 4.
\end{align*}
In case $(ii)$, we remain with
\begin{align*}
	\E_W \tilde W_{k,l}(z)^2 &= 4z + \frac{2}{\sqrt{\hat h_{n,k}^{loc} \hat h_{n,l}^{loc}}} \Bigg( k\delta_n - z \hat h_{n,k}^{loc} - \min \left\{ k\delta_n + z \hat h_{n,k}^{loc}, l\delta_n + z \hat h_{n,l}^{loc} \right\}  \Bigg).
\end{align*}
If in the latter expression $k\delta_n +z\hat h_{n,k}^{loc} \leq l\delta_n+z\hat h_{n,l}^{loc}$, then 
\begin{align*}
	\E_W \tilde W_{k,l}(z)^2 &= 4z - \frac{4z\hat h_{n,k}^{loc}}{\sqrt{\hat h_{n,k}^{loc} \hat h_{n,l}^{loc}}} \leq 4.
\end{align*}
Otherwise, if $k\delta_n +z\hat h_{n,k}^{loc} > l\delta_n+z\hat h_{n,l}^{loc}$, we arrive at
\begin{align*}
	\E_W \tilde W_{k,l}(z)^2 &= 4z + \frac{2}{\sqrt{\hat h_{n,k}^{loc} \hat h_{n,l}^{loc}}} \Bigg( (k-l)\delta_n - z \left(\hat h_{n,k}^{loc}+ \hat h_{n,l}^{loc} \right) \Bigg) \leq 4
\end{align*}
because $k \leq l$ and $z \in [0,1]$.
Summarizing,
\begin{align*}
	\E_W \tilde W_{k,l}(z)^2 \leq 4.
\end{align*}
\end{proof}

\begin{lemma} \label{A2}
 For any $x \in [0,1]$, we have
 \begin{align*}
 	e^x - 1 \leq 2x.
 \end{align*}
\end{lemma}

\begin{proof}
	Equality holds for $x=0$, while $e-1 \leq 2$. Hence, the result follows by convexity of the exponential function.
\end{proof}


%

\begin{lemma} \label{A3}
For any $x \in \R \setminus \{0\}$, we have
$$1-\frac{\sin(x)}{x }\leq \frac{x^2}{6}.$$
\end{lemma}

\begin{proof}
Since both sides of the inequality are symmetric in zero, we restrict our considerations to $x > 0$. For positive $x$, it is equivalent to
\begin{align*}
	f(x) = \sin(x) - x + \frac{x^3}{6} \geq 0.
\end{align*}
As $f(0)=0$, it suffices to show that 
$$f'(x) = \cos(x) -  1 + \frac{x^2}{2} \geq 0$$
for all $x > 0$. Since furthermore $f'(0)=0$ and
\begin{align*}
	f''(x) = -\sin(x)+x \geq 0
\end{align*}
for all $x > 0$, the inequality follows.
\end{proof}

\noindent The next lemma shows that the monotonicity of the H\"older norms $\Vert \cdot \Vert_{\beta_1,U} \leq \Vert \cdot \Vert_{\beta_2,U}$ with $0 < \beta_1 \leq \beta_2$ stays valid for the modification $\Vert \cdot \Vert_{\beta,\beta^*,U}$.

\begin{lemma} \label{ineq_holdernorm1}
For $0<\beta_1 \leq \beta_2 < \infty$ and $p \in \mathcal H_{\beta^*,U}(\beta_2)$,
\begin{align*}
	\Vert p \Vert_{\beta_1,\beta^*,U} \leq \Vert p \Vert_{\beta_2,\beta^*,U} 
\end{align*}
for any open interval $U \subset \R$ with length less or equal than $1$.
\end{lemma}

\begin{proof}
If $\beta_1 \leq \beta_2$, but $\lfloor \beta_1 \wedge \beta^* \rfloor = \lfloor \beta_2 \wedge \beta^* \rfloor$, the statement follows directly with
\begin{align*}
	\Vert p \Vert_{\beta_1,\beta^*,U} &= \sum_{k=0}^{\lfloor \beta_2 \wedge \beta^* \rfloor} \Vert p^{(k)} \Vert_{U} + \sup_{\substack{x,y \, \in \, U \\ x \neq y}} \frac{\vert p^{(\lfloor \beta_2 \wedge \beta^* \rfloor)}(x) - p^{(\lfloor \beta_2 \wedge \beta^*\rfloor)}(y) \vert}{\vert x-y \vert^{\beta_1 - \lfloor \beta_2 \wedge \beta^* \rfloor}} \leq \Vert p \Vert_{\beta_2,\beta^*,U}.
\end{align*}
If $\beta_1 < \beta_2$ and also $\lfloor \beta_1 \wedge \beta^* \rfloor < \lfloor \beta_2 \wedge \beta^* \rfloor$, we deduce that $\beta_1 < \beta^*$ and $\lfloor \beta_1 \rfloor + 1 \leq \lfloor \beta_2 \wedge \beta^* \rfloor$. Then, the mean value theorem yields
\begin{align*}
	\Vert p \Vert_{\beta_1,\beta^*,U} &= \sum_{k=0}^{\lfloor \beta_1 \rfloor} \Vert p^{(k)} \Vert_{U} + \sup_{\substack{x,y \, \in \, U \\ x \neq y}} \frac{\vert p^{(\lfloor \beta_1 \rfloor)}(x) - p^{(\lfloor \beta_1 \rfloor)}(y) \vert}{\vert x-y \vert^{\beta_1 - \lfloor \beta_1 \rfloor}} \\
			&\leq \sum_{k=0}^{\lfloor \beta_1 \rfloor} \Vert p^{(k)} \Vert_{U} + \Vert p^{(\lfloor \beta_1 \rfloor + 1)} \Vert_U \sup_{\substack{x,y \, \in \, U \\ x \neq y}} \vert x-y \vert^{1-(\beta_1 - \lfloor \beta_1 \rfloor)} \\
			&\leq \sum_{k=0}^{\lfloor \beta_1 \rfloor + 1} \Vert p^{(k)} \Vert_{U} \\
			&\leq \sum_{k=0}^{\lfloor \beta_2 \wedge \beta^* \rfloor} \Vert p^{(k)} \Vert_{U} \\
			&\leq \Vert p \Vert_{\beta_2,\beta^*,U}.
\end{align*}
\end{proof}

\end{appendix}

\bibliography{reference}
\bibliographystyle{imsart-nameyear}

\end{document}